\PassOptionsToPackage{unicode}{hyperref}
\PassOptionsToPackage{hyphens}{url}
\RequirePackage{setspace}
\documentclass[aap]{imsart}
\RequirePackage{amsthm,amsmath,amsfonts,amssymb}
\RequirePackage[numbers,sort&compress]{natbib}
\RequirePackage[colorlinks,citecolor=blue,urlcolor=blue]{hyperref}
\RequirePackage{graphicx}
\usepackage{color}
\usepackage[utf8]{inputenc}
\usepackage[T1]{fontenc}
\usepackage{comment}
\usepackage{textcomp} 
\usepackage{bbm}
\usepackage{nicefrac}

\IfFileExists{footnotehyper.sty}{\usepackage{footnotehyper}}{\usepackage{footnote}}
\usepackage{graphicx}
\makeatletter
\def\maxwidth{\ifdim\Gin@nat@width>\linewidth\linewidth\else\Gin@nat@width\fi}
\def\maxheight{\ifdim\Gin@nat@height>\textheight\textheight\else\Gin@nat@height\fi}
\makeatother

\setkeys{Gin}{width=\maxwidth,height=\maxheight,keepaspectratio}
\makeatletter
\def\fps@figure{htbp}
\makeatother

\NewDocumentCommand\citeproctext{}{}
\NewDocumentCommand\citeproc{mm}{%
  \begingroup\def\citeproctext{#2}\cite{#1}\endgroup}

\newlength{\cslhangindent}
\newlength{\csllabelwidth}
 {\begin{list}{}{%
  \setlength{\itemindent}{0pt}
  \setlength{\leftmargin}{0pt}
  \setlength{\parsep}{0pt}
  \ifodd #1
   \setlength{\leftmargin}{\cslhangindent}
   \setlength{\itemindent}{-1\cslhangindent}
  \fi
  \setlength{\itemsep}{#2\baselineskip}}}
 {\end{list}}

\usepackage{calc}

\usepackage{booktabs}
\usepackage{multirow}
\usepackage{amsmath}
\usepackage{amsthm}
\usepackage{setspace} 
\usepackage{mathtools}

\usepackage{subcaption} 
\usepackage{float}
\usepackage{algorithm}
\usepackage{algpseudocode}
\graphicspath{ {../_bookdown_files/} {../} }

\newcommand\RR{\mathbb{R}}

\newcommand\mf{\mathbf}
\newcommand\mc{\mathcal}

\newcommand\R{\mathbf{R}}

\newcommand\E{\mathbb{E}}

\newcommand\Null{\operatorname{Null}}
\newcommand\Range{\operatorname{Range}}

\ifLuaTeX
  \usepackage{selnolig}  
\fi
\IfFileExists{bookmark.sty}{\usepackage{bookmark}}{\usepackage{hyperref}}
\IfFileExists{xurl.sty}{\usepackage{xurl}}{} 
\hypersetup{
  pdftitle={Geometric ergodicity of Gibbs samplers for non-Gaussian linear mixed effect models},
  pdfauthor={Elsiddig Awadelkarim, David Bolin, Xiaotian Jin, Alexandre B. Simas, and Jonas Wallin},
  hidelinks,
  pdfcreator={LaTeX via pandoc}}

\usepackage{amsthm}
\theoremstyle{plain}
\newtheorem{theorem}{Theorem}[section]
\newtheorem{lemma}{Lemma}[section]
\newtheorem{corollary}{Corollary}[section]
\newtheorem{proposition}{Proposition}[section]

\theoremstyle{definition}

\newtheorem{assumption}{Assumption}[section]
\theoremstyle{remark}
\newtheorem*{remark}{Remark}

\begin{document}
\begin{frontmatter}
  \title{Geometric ergodicity of Gibbs samplers for linear latent models with GIG variance mixtures}
  \runtitle{Geometric ergodicity of Gibbs samplers for LLnGMs}
  \thankstext{t1}{The authors are ordered alphabetically.}
  \thankstext{t2}{Corresponding author.}

  \begin{aug}
    \author[A]{\fnms{Elsiddig}~\snm{Awadelkarim}\thanksref{t1}\ead[label=e5]{elsiddig.awadelkarim@kaust.edu.sa}},
    \author[A]{\fnms{David}~\snm{Bolin}\ead[label=e1]{david.bolin@kaust.edu.sa}},
    \author[A]{\fnms{Xiaotian}~\snm{Jin}\thanksref{t2}\ead[label=e2]{xiaotian.jin@kaust.edu.sa}},
    \author[A]{\fnms{Alexandre B.}~\snm{Simas}\ead[label=e3]{alexandre.simas@kaust.edu.sa}},
    \and
    \author[B]{\fnms{Jonas}~\snm{Wallin}\ead[label=e4]{jonas.wallin@stat.lu.se}}
    \address[A]{CEMSE, King Abdullah University of Science and Technology, Saudi Arabia \printead[presep={,\ }]{e1,e2,e3,e5}}

    \address[B]{Department of Statistics, Lund University, Sweden \printead[presep={,\ }]{e4}}
  \end{aug}



  \begin{abstract}
    We study geometric ergodicity of the Gibbs sampler for linear latent non-Gaussian models (LLnGMs), a class of hierarchical models in which conditional Gaussian structure is preserved through generalized inverse Gaussian (GIG) variance-mixture augmentation. Two complementary routes to geometric ergodicity are developed for the marginal chain on the mixing variables. First, we show that the associated Markov operator is trace-class, and hence admits a spectral gap, over a large portion of the GIG parameter space. Second, for the remaining boundary and heavy-tail regimes, we establish geometric ergodicity via drift and minorization, subject to an explicit null-smallness condition that quantifies how the drift interacts with the null space of the observation operator. Together, these results cover the full GIG parameter space, including the normal-inverse Gaussian, generalized asymmetric Laplace, and Student-$t$ special cases. The geometric ergodicity of this chain underpins the consistency of Gibbs-based stochastic-gradient estimators for maximum likelihood estimation, and we provide conditions that make the required integrability checks transparent. Numerical experiments illustrate the theoretical findings, contrasting mixing efficiency across parameter regimes and probing the role of the null-smallness constant.
  \end{abstract}



\end{frontmatter}

\section{Introduction}\label{chp:introduction}

\subsection{Geometric Ergodicity of Markov Chains}\label{ergometric-ergodicity}

Markov Chain Monte Carlo (MCMC) methods are sampling algorithms designed to approximate probability distributions. These methods have become fundamental tools in statistical computation across virtually all scientific disciplines.
For Markov chains, ergodicity ensures convergence to a unique stationary distribution independent of the initial distribution. This property is essential because it guarantees that time averages along a single trajectory converge to ensemble averages under the stationary distribution. Geometric ergodicity is a strengthening of ergodicity that quantifies the rate of convergence, guaranteeing that the chain converges to its stationary distribution at an exponential (geometric) rate.

A Markov chain $\{X_t\}_{t\geq 0}$ with state space $\mathcal{X}$ and stationary distribution $\pi$ is geometrically ergodic if constants $C > 0$, $\gamma \in (0, 1)$ exist such that for all probability measures $\mu \in \mathcal{L}_2(\pi)$,
\begin{equation*}
  \|\mu P^t - \pi\|_{\mathrm{TV}} \leq C \gamma^t,
  \label{eq:geometric-ergodicity}
\end{equation*}
where $P$ is the Markov operator associated to the chain, $t\in\mathbb{N}$, and $\|\cdot\|_{\mathrm{TV}}$ denotes the total variation distance. Here, $\mathcal{L}_2(\pi)$ is the space of signed measures on $\mathcal{X}$ that are absolutely continuous with respect to $\pi$ and have square-integrable Radon--Nikodym derivatives:
\begin{equation*}
\mathcal{L}_2(\pi) = \left\{\mu \in \mathrm{SM}(\mathcal{X}): \mu \ll \pi \text{ and } \int \left(\frac{d\mu}{d\pi}\right)^2 d\pi <\infty\right\},
\end{equation*}
where $\mathrm{SM}(\mathcal{X})$ denotes the set of signed measures on $\mathcal{X}$. The action of a measure $\mu$ on the operator $P$ is defined by
\begin{equation*}
\mu P(A) = \int P(x, A) \, d\mu(x)
\end{equation*}
for all measurable sets $A$. 
This property is essential in applications because it quantifies the rate at which the chain converges to stationarity. Specifically, geometric ergodicity guarantees that the number of steps required to achieve a given approximation accuracy scales logarithmically with that accuracy.

The transition kernel of a Markov chain naturally induces a linear integral operator, called the Markov operator, that acts on functions by averaging them over one-step transitions.
This operator is said to be trace-class if its singular values are summable, a condition that implies compactness.
When the chain is Harris ergodic, compactness guarantees that the spectrum of the operator, apart from the simple eigenvalue at one, lies strictly inside the unit disk, yielding a spectral gap and, consequently, geometric ergodicity \cite{pal2017traceclass}.

Geometric ergodicity of Gibbs samplers for high-dimensional linear models has been established using Lyapunov functions and minorization conditions on sub-level sets \cite{rosenthal1995}. This framework has been applied to global-local shrinkage priors, including the Bayesian Lasso \cite{park2008bayesian, khare2013}, Dirichlet-Laplace prior \cite{bhattacharya2015, pal2014}, and horseshoe prior \cite{carvalho2010horseshoe, bhattacharya2022}. The trace-class property for related models has been established in \cite{zhang2019, zhang2017, pal2017traceclass}.

In this work, we extend these results in two directions. First, we consider the generalized inverse Gaussian (GIG) family as the mixing distribution, which contains many of the priors studied previously, including the Normal-Gamma prior of \cite{pal2014} as a special case. Second, we work under a more general model structure than the standard linear regression setting considered in earlier work, so that the results are more more widely applicable. 

In the following subsections, we introduce the model class and introduce stochastic gradient descent (SGD) optimization as  motivating example for why it is important to derive geometric ergodicity for the model class.

\subsection{Latent Gaussian models and the non-Gaussian extensions} 
\label{sec:model}
Hierarchical models are fundamental statistical frameworks for modeling dependencies across multiple scales. Latent Gaussian models (LGMs) form a broad class of hierarchical models in which unobserved latent variables have Gaussian distributions. This class encompasses generalized linear mixed models \citeproc{fong2010}{Fong, Rue, and Wakefield 2010}, survival models \citeproc{martino2011}{Martino, Akerkar, and Rue 2011}, and spatial and spatio-temporal models \citeproc{INLA2018}{Bakka et al. 2018}, among others.

Inference for LGMs can be performed via MCMC sampling from the posterior distribution. For problems with large-scale latent Gaussian fields, the integrated nested Laplace approximation (INLA) methodology often provides more efficient and accurate inference (\cite{rue2017}).
LGMs have the following hierarchical structure:
\begin{equation}
  \begin{aligned}
    \text{Data}  \quad           &
    Y|W, \theta_1, \theta_2 \sim \prod_{i\in\mc{I}} \pi(Y_i|W_i, \theta_1), \\
    \text{Latent field} \quad    &
    W|\theta_2 \sim N(\mu(\theta_2), Q^{-1}(\theta_2)),           \\
    \text{Hyperparameters} \quad &
    \theta_1 \sim \pi(\theta_1), \quad \theta_2 \sim \pi(\theta_2).
  \end{aligned}
\end{equation}
This hierarchy consists of three stages: (i) the observation model, specifying the conditional distribution of data $Y$ given latent variables $W$ and hyperparameters $\theta_1$; (ii) the latent Gaussian field $W$ given hyperparameters $\theta_2$; and (iii) hyperparameter priors.

Despite the popularity and flexibility of Gaussian processes, they can be overly smooth for datasets with the presence of local spikes \citeproc{walder2020}{Walder and Hanks 2020}, and sudden jumps \citeproc{dhull2021}{Dhull and Kumar 2021}.
Non-Gaussian extensions of LGMs address limitations of Gaussian assumptions by allowing latent variables to exhibit 
skewness (asymmetric distributions),
heavy tails (frequent extreme values),
sample-path discontinuities (abrupt jumps),
and directional asymmetry (unequal variability across sides of the distribution).
Such features can improve prediction accuracy and better capture local data characteristics (\cite{paciorek2003}, \cite{bolin2014}, \cite{wallin2015}).
One approach to constructing non-Gaussian processes, introduced in \cite{bolin2014} and extended in \cite{asar2020}, employs variance-mixture representations.

In this work, we specifically consider the following non-Gaussian extension of LGMs with a Gaussian likelihood, which we refer to as a linear latent non-Gaussian model (LLnGM):
\begin{subequations} \label{eq:centered-combined}
  \begin{align}
    \text{Data}  \quad         &
    Y|W, V \sim N(X \beta + A W, \sigma_{\epsilon}^2 I),                 \\
    \text{Latent fields} \quad &
    K(\zeta) W | V \sim N(\mu V - \mu h, \text{diag}(\sigma^2 V)), \label{latent} \\
    \text{MixVars} \quad       &
    V_i \overset{\text{ind.}}{\sim} \text{GIG}(p,a,b).
  \end{align}
\end{subequations}
Here $Y = (Y_1, \ldots, Y_m)^\top$ is the data vector,
$W = (W_1, \ldots, W_n)^\top$ is the latent field, and 
${\zeta = (\zeta_1, \ldots, \zeta_p)^\top}$ are the hyperparameters of the square and invertible matrix $K$ which determines the specific model structure.
Further, $\mu\in\mathbb{R}$ and $\sigma>0$ are parameters controlling variance and skewness of the latent model, and $X$ is a design matrix with corresponding regression coefficients $\beta$. The matrix $A$ maps the latent field to the observation space, and $h$ is a fixed vector that is determined by the specific model choice (examples are given below). 
\textcolor{black}{We can also make the  framework Bayesian by adding priors on the parameters, but this is not our main focus and does not affect the later results.}

The final component of the model is the vector of mixing variables $V = (V_1, \ldots, V_n)^\top$, which are assumed to be independent and follow a generalized inverse Gaussian (GIG) distribution. The GIG density function is
\begin{equation*}
  f(x;p,a,b) = \frac{(a/b)^{p/2}}{2K_p(\sqrt{ab})} x^{p-1} \exp\left\{-\frac{1}{2}(ax + b/x)\right\}, \quad x > 0,
  \label{eq:gig}
\end{equation*}
where \(K_p\) is the modified Bessel function of the second kind of order \(p\), and the parameter space is $\Psi = \Psi_{I}\cup \Psi_{IG}\cup \Psi_{\Gamma}$ where
\begin{equation*}
  \Psi_{I} = \{a>0, b>0, p\in \RR\},\quad
  \Psi_{IG} =  \{a=0, b>0, p<0\}, \quad
   \Psi_{\Gamma} = \{a>0, b=0, p>0\}.
\end{equation*}
Here $\Psi_{I}$ is the interior of the parameter space, $\Psi_{IG}$ is a boundary case corresponding to the inverse Gamma distribution, and $\Psi_{\Gamma}$ is a boundary case corresponding to the gamma distribution.
The GIG family also includes the inverse Gaussian distribution (\(p=1/2\)) as an interior case.
An important property is the scaling behavior: if \(V \sim \text{GIG}(p, a, b)\), then \(cV \sim \text{GIG}(p, a/c, cb)\) for any \(c>0\).

When the GIG distribution is employed as the mixing distribution in a Gaussian variance mixture, the resulting marginal distribution belongs to the generalized hyperbolic (GH) family (\cite{barndorff1978}), a very flexible class that includes the hyperbolic, Student's $t$, generalized asymmetric Laplace (GAL), and normal-inverse Gaussian (NIG) distributions, among others. Specifically, a GH random variable $Z$ admits the representation
\begin{equation}
  Z = \delta + \mu V + \sigma \sqrt{V} U, 
  \label{eq:mixture}
\end{equation}
where $\delta\in\mathbb{R}$, $\mu\in\mathbb{R}$, and $\sigma>0$ are location, drift, and scale parameters, $V \sim \text{GIG}(p, a, b)$, and $U \sim N(0,1)$ is independent of $V$. The full GH distribution has six parameters $(\delta, \mu, \sigma, p, a, b)$; however, as our model has fixed effects which can determine the location we fix $\delta = -\mu h$ for the construction in \eqref{eq:centered-combined} to avoid over parameterization. 
Table \ref{tab:GH} lists some special cases obtained by varying the GIG parameters $(p, a, b)$.

\begin{table}[t]
  \centering
  \caption{\label{tab:GH}Special cases of the GH distribution and their corresponding GIG mixing distributions}
  \centering
  \begin{tabular}[t]{lll}
    \toprule
    Distribution of $X$ & Distribution of $V$           & GIG form of $V$            \\
    \midrule
    t                   & Inv-Gamma($\nu, \nu$)         & GIG($-\nu, 0, 2\nu$)       \\
    NIG                 & Inv-Gaussian($\alpha, \beta$) & GIG($-0.5, \alpha, \beta$) \\
    GAL                 & Gamma($\alpha, \beta$)        & GIG($\alpha, 2\beta, 0$)   \\
    Cauchy              & Inv-Gamma($0.5, \nu$)         & GIG($-0.5, 0, 2\nu$)       \\
    \bottomrule
  \end{tabular}
\end{table}

By specifying different forms of the operator $K$, this framework accommodates a variety of widely used latent non-Gaussian models, including auto-regressive, random walk, SPDE Matérn, and space-time models, see  \cite{cabral2024} for several examples operators $K$ and corresponding models. For each of these models, the vector $h$ is given explicitly. For example, for auto-regressive models, $h=1$, whereas $h$ is determined by the discretization of SPDE Mat\'ern models \cite{bolin2014}. 

A recurring challenge in non-Gaussian latent modeling is that richer marginal behavior typically comes at the cost of losing scalable Gaussian computational tools. LLnGMs circumvent this trade-off by introducing mixing variables that render the latent field conditionally Gaussian. As a result, sparse operator structures (e.g., from autoregressive formulations) can still be used for computationally efficient linear algebra, while the non-Gaussian features are localized in one-dimensional updates for the mixing components. 

\subsection{Stochastic Gradient Descent}\label{sec:sgd}
A natural question is how the model \eqref{eq:centered-combined} can be fitted to data. In certain scenarios, the variational Bayes approach of \cite{cabral2024} can be used, \textcolor{black}{and Monte Carlo Expectation Maximization (MCEM) methods have also been used for certain non-Gaussian models in this class \cite{bolin2014, wallin2015}}. However, the most widely used approach is maximum-likelihood estimation based on a stochastic gradient descent (SGD) method, \cite[see, e.g.][]{asar2020, bolin2020}.
In fact, the general software \texttt{ngme2} (\url{https://github.com/davidbolin/ngme2}) which implements many latent non-Gaussian models of the form \eqref{eq:centered-combined} uses this approach.

SGD, \textcolor{black}{introduced by \cite{robbins1951},} is a widely used iterative optimization method for large-scale problems in machine learning, neural networks, and statistical inference. To minimize a function $f(\theta)$, $\theta \in \RR^n$, using SGD, we start with an initial parameter vector $\theta^{(0)}$ and update iteratively via
\begin{equation}
  \theta^{(t+1)} = \theta^{(t)} - \gamma_t \widehat{\nabla f} (\theta^{(t)}),
  \label{eq:sgd}
\end{equation}
where $\gamma_t > 0$ is the step size at iteration $t$ and $\widehat{\nabla f}(\theta)$ is a random vector satisfying $E[\widehat{\nabla f}(\theta)] = \nabla f(\theta)$. Under appropriate conditions on the step size $\gamma_t$ and relatively mild assumptions, the iterates converge almost surely to a minimizer of $f$ \cite{bottou1998}.

In the context of LLnGMs, the function $f$ is the negative log-likelihood, and sampling of the corresponding random vector $Q_n(\theta)$ is done via Gibbs sampling. The key property that enables Gibbs sampling in this scenario is that 
conditional distribution of $V$ given $X$ remains GIG-distributed. The details of this are given in Section \ref{sec:gibbs} and Section \ref{chp:estimation}. However, because a Gibbs sampler is used in the estimation, having geometric ergodicity justifies using small batches of samples to estimate gradients and provides theoretical guarantees for the estimator.

\begin{color}{black}
It is important to note that likelihood-based estimation via SGD may break down for certain
non-Gaussian noise models.
In particular, when the latent field is driven by generalized asymmetric Laplace (GAL) noise,
the corresponding variance-mixing variable $V$ follows a Gamma distribution with small shape parameter (i.e., $p<1/2$).
As discussed in Section~\ref{sec:L1}, this leads to severe numerical and statistical
difficulties in estimation.
This issue is especially pronounced in SPDE-based models.
Under mesh refinement, the effective shape parameter $p$ decreases with the discretization level,
and for any fixed $p>0$ of the driving noise for the continuous model, it holds that $p\to 0$ for the noise in discretized model as the mesh is refined \cite{wallin2015}.
Consequently, the mixing distribution becomes increasingly degenerate, causing instability in
both likelihood evaluation and gradient estimation.
This phenomenon was already observed in \cite{wallin2015}, which resulted in stability issues for the proposed MCEM algorithm.
One key contribution of this work is that we later show that these issues can be avoided by changing from the so-called centered parametrization in \eqref{eq:centered-combined} to an alternative non-centered parametrization introduced in the next section.
\end{color}

\subsection{Main contributions and outline}
The main contributions of this work are as follows.
\begin{enumerate}
  \item
        We establish geometric ergodicity and the trace-class property for Gibbs samplers applied to the LLnGM  where non-Gaussianity is introduced via the GIG mixing variables. The results allow for general (invertible) operator matrices $K(\zeta)$ and cover the full GIG parameter space, including the NIG, GAL, and Student-$t$ special cases.
  \item
        We analyze SGD estimation via Fisher's identity and identify integrability conditions that are transparent in the non-centered parameterization introduced below, but delicate in original the centered formulation in \eqref{eq:centered-combined}.
        In particular, in the centered parameterization the complete-data score may fail to be integrable under the posterior, preventing a direct application of ergodic theorems, which means that one cannot prove that the SGD method converges. We show that these issues are avoided by using a non-centered parameterization, which thus allows for deriving convergence of the SGD method based on the ergodicity results.
  \item
        We present numerical experiments that illustrate the mixing behavior of the proposed Gibbs samplers and empirically validate the theoretical results.
\end{enumerate}

The outline of the paper is as follows. In Section~\ref{chp:gibbs}, the Gibbs sampler is introduced and the main results are stated. 
The application to SGD is introduced in Section~\ref{chp:estimation}.
The proofs of the main results are given in Section~\ref{chp:proof}  and simulation studies that validate the results are presented in Section~\ref{chp:simulation-study}. The article concludes with a discussion in Section~\ref{sec:discussion} and technical details are provided in three appendices.
As previously mentioned, the R package \texttt{ngme2} implements the proposed estimation approach for the model class, and all numerical experiments are reproducible through this package.

\section{The Gibbs sampler and main results of geometric ergodicity}\label{chp:gibbs}

\subsection{The Gibbs sampler}\label{sec:gibbs}

In the LLnGM, we have unobserved mixing variables $V$ and latent fields $W$. Our goal is to sample from the posterior distribution of the unobserved variables, i.e., $V, W | Y$. It is difficult to sample from the full conditional distribution directly, since we do not have closed-form expressions. However, we can use the Gibbs sampler to sample from the conditional distributions iteratively. We propose the Gibbs sampling algorithm stated in Algorithm~\ref{alg:gibbs_centered}.

\begin{algorithm}[b]
\caption{Gibbs sampler for the LLnGM}\label{alg:gibbs_centered}
\begin{algorithmic}[1]
\State For centered parameterization, $U = W$. For non-centered parameterization $U = M$.
\State Initialize $V^{(0)}$.
\For{$t = 1, 2, \ldots, T$}
    \State Sample $U^{(t)}$ from the conditional distribution $U|V^{(t-1)}, Y$.
    \State Sample $V^{(t)}$ from the conditional distribution $V|U^{(t)}, Y$.
\EndFor
\end{algorithmic}
\end{algorithm}

The formulation in~\eqref{eq:centered-combined}, which we henceforth refer to as the \emph{centered parameterization}, can be inconvenient for our theoretical analysis. To streamline subsequent proofs, we introduce an equivalent \emph{non-centered parameterization} via a transformed latent field.
Specifically, we define the auxiliary latent variable
  $M := K(\zeta)W + \mu h$,
so that the original latent field is recovered via the bijective mapping
  $W = K(\zeta)^{-1}(M-\mu h)$.
Under this transformation, the hierarchical model becomes
\begin{subequations}\label{eq:non-centered-combined}
  \begin{align}
    \text{Data} \quad             &
    Y \mid M,V \sim N\!\bigl(X \beta + A K(\zeta)^{-1}(M-\mu h),\ \sigma_\epsilon^2 I \bigr), \\
    \text{Latent fields} \quad    &
    M \mid V \sim N\!\bigl(\mu V,\ \mathrm{diag}(\sigma^2 V)\bigr), \label{eq:latent-nc}     \\
    \text{Mixing variables} \quad &
    V_i \ \overset{\mathrm{ind.}}{\sim}\ \mathrm{GIG}(p,a,b), \qquad i=1,\dots,n.
  \end{align}
\end{subequations}

Because the transformation between $W$ and $M$ is a diffeomorphism for each fixed $\zeta$, the centered and non-centered parameterizations induce the same marginal likelihood for $Y$ and thus define equivalent statistical models. 
As a secondary benefit, the non-centered representation also simplifies stochastic-gradient inference (Section~\ref{sec:L1}) by improving the integrability properties of certain score components under the posterior.
For the non-centered parameterization in \eqref{eq:non-centered-combined}, the Gibbs sampling algorithm is analogous except that $M$ is sampled instead of $W$.

We now specify the conditional distributions required for both parameterizations. To keep notation simple, we fix some hyperparameter $\zeta$ and drop $\zeta$ from the notation, that is, we simply write $K$ instead of $K(\zeta)$.

\begin{proposition}
\label{prop:conditional-distributions}
  Under the centered parameterization \eqref{eq:centered-combined}, the conditional distribution of $W|V, Y$ is
  \begin{align}
    W|V, Y \sim N\!\left(Q^{-1} \left(\frac{1}{\sigma^2} K^{\top} D_{V}^{-1} \mu (V-h) + \frac{1}{\sigma_{\epsilon}^2} A^\top (Y - X \beta) \right),\ Q^{-1}\right),
    \label{eq:centered-conditional}
  \end{align}
  where $Q = \sigma^{-2} K^{\top} D_{V}^{-1} K + \sigma_{\epsilon}^{-2} A^\top A$ and $D_{V} = \text{diag}(V)$.

  Under the non-centered parameterization \eqref{eq:non-centered-combined}, the conditional distribution of $M|V, Y$ is
  \begin{align}
    M|V, Y \sim N\!\left(
    Q^{-1} \left( \sigma_{\epsilon}^{-2} B^\top (Y - X \beta + \mu B h ) + \frac{\mu}{\sigma^2} \mathbf{1} \right),
    \ Q^{-1}\right),
    \label{eq:non-centered-conditional}
  \end{align}
  where $Q = \sigma_{\epsilon}^{-2} B^\top B + \sigma^{-2} D_{V}^{-1}$, $B = A K^{-1}$, and $\mathbf{1}$ is the vector of ones.

  \medskip
  Moreover, in both parameterizations the full conditional of $V$ factorizes over components and remains in the GIG family.
  Since $Y\mid W,V$ and $Y\mid M,V$ do not depend on $V$, we have
  \begin{equation}
    V_i \mid W,Y \ \sim\ \mathrm{GIG}\!\left(p-\tfrac12,\ a + \frac{\mu^2}{\sigma^2},\ b + \frac{((KW)_i+\mu h_i)^2}{\sigma^2}\right),
    \qquad i=1,\dots,n,
    \label{eq:V_given_W}
  \end{equation}
  and
  \begin{equation}
    V_i \mid M,Y \ \sim\ \mathrm{GIG}\!\left(p-\tfrac12,\ a + \frac{\mu^2}{\sigma^2},\ b + \frac{M_i^2}{\sigma^2}\right),
    \qquad i=1,\dots,n,
    \label{eq:V_given_M}
  \end{equation}
  independently over $i$.
\end{proposition}
The proof is given in Appendix \ref{app:proofs-model}.

\subsection{Gibbs sampling from linear operator's viewpoint}\label{viewpoint}

The Gibbs sampler for the LLnGM induces a Markov chain $(V^{(t)})$ with transition density:
\begin{equation}
  k(V, \tilde{V}) = \int_{\RR^n} \pi(\tilde{V} | W, Y)
  \pi(W | V, Y) dW,
  \label{eq:transition}
\end{equation}
where $Y$ is the observation, $W$ is the latent field, and $V$ is the mixing field.
Let $L_2(\pi)$ denote the Hilbert space of functions that are square integrable with respect to the target posterior density $\pi(V|Y)$:
\begin{equation*}
  L_2(\pi):=\left\{
  f:\RR^n_+ \rightarrow \RR;
  \int_{\RR^n_+} (f(V))^2 \pi(V|Y) dV < \infty
  \right\},
\end{equation*}
where $\RR^n_+ = \{ V \in \RR^n: V_i > 0, i=1, \ldots, n \}$, equipped with the standard inner product:
\begin{equation*}
  (f, g) := (f, g)_{L_2(\pi)} = \int_{\RR^n_+} f(V) g(V) \pi(V|Y) dV.
\end{equation*}
Note that there is a natural identification between $L_2(\pi)$ and $\mathcal{L}_2(\pi)$, where each function $f \in L_2(\pi)$ is identified with the measure $\mu_f(dV) = f(V) \pi(V|Y) dV = f(V) \pi(dV)$.
The Markov transition density $k$ (with respect to the Lebesgue measure) induces a linear operator $\Lambda$ on $L_2(\pi)$:
\begin{equation}
  (\Lambda f)(V) = \int_{\RR^n_+} k(V, \tilde{V}) f(\tilde{V}) d \tilde{V}.
  \label{eq:Lambda-operator}
\end{equation}

We now obtain fundamental properties of the operator $\Lambda$.

\begin{proposition}[Properties of the $V$-marginal Markov operator]\label{prp:Lambda-properties}
  Let $\Lambda$ be defined by \eqref{eq:Lambda-operator} with transition density $k$ given in \eqref{eq:transition},
  and let $\pi$ denote the invariant probability measure of the $V$-chain with density $\pi(\cdot\mid Y)$.
  Then $\Lambda$ satisfies $\Lambda\mf{1}=\mf{1}$ and is a contraction on $L_2(\pi)$, i.e., $\|\Lambda\|\le 1$.
  Moreover, $\Lambda$ is self-adjoint and positive semidefinite on $L_2(\pi)$; consequently,
  $\mathrm{spec}(\Lambda)\subset[0,1]$.
\end{proposition}

\begin{proof}
  Since $k(V,\tilde V)$ is a Markov transition density, $k\ge 0$ and
  $\int k(V,\tilde V)\,d\tilde V=1$ for all $V$, which implies $\Lambda\mf{1}=\mf{1}$.
  We start by establishing the contraction property.
  For $\pi$-a.e.\ $V$, by Jensen's inequality,
  \[
    (\Lambda f(V))^2=\Big(\int k(V,\tilde V)f(\tilde V)\,d\tilde V\Big)^2
    \le \int k(V,\tilde V)f(\tilde V)^2\,d\tilde V.
  \]
  Integrating both sides against $\pi(V\mid Y)\,dV$ and using the invariance property
  $\int k(V,\tilde V)\pi(V\mid Y)\,dV=\pi(\tilde V\mid Y)$ yields
  $\|\Lambda f\|_{L_2(\pi)}^2\le \|f\|_{L_2(\pi)}^2$.

  Next, we establish self-adjointness and positive semidefiniteness.
  For $f,g\in L_2(\pi)$,
  \begin{align*}
    (\Lambda f,g)_{L_2(\pi)}
     & =\iint k(V,\tilde V)\,f(\tilde V)\,g(V)\,\pi(V\mid Y)\,d\tilde V\,dV                              \\
     & =\iiint f(\tilde V)\,g(V)\,\pi(\tilde V\mid W,Y)\,\pi(W\mid V,Y)\,\pi(V\mid Y)\,dW\,d\tilde V\,dV \\
     & =\iiint f(\tilde V)\,g(V)\,\pi(\tilde V\mid W,Y)\,\pi(V\mid W,Y)\,\pi(W\mid Y)\,dW\,d\tilde V\,dV \\
     & =\int \pi(W\mid Y)\Big[\int f(\tilde V)\pi(\tilde V\mid W,Y)\,d\tilde V\Big]
    \Big[\int g(V)\pi(V\mid W,Y)\,dV\Big]\,dW                                                            \\
     & =\mathbb E\!\left[\mathbb E\!\left[f(V)\mid W,Y\right]\,
      \mathbb E\!\left[g(V)\mid W,Y\right]\ \Big|\ Y\right],
  \end{align*}
  which is symmetric in $(f,g)$, establishing self-adjointness.
  Setting $g=f$ gives
  \[
    (\Lambda f,f)_{L_2(\pi)}
    =\mathbb E\!\left[\Big(\mathbb E[f(V)\mid W,Y]\Big)^2\ \Big|\ Y\right]\ge 0,
  \]
  which establishes positive semidefiniteness.
  Finally, since $\Lambda$ is self-adjoint and a contraction, $\mathrm{spec}(\Lambda)\subset[-1,1]$,
  and positive semidefiniteness further restricts $\mathrm{spec}(\Lambda)\subset[0,1]$.
\end{proof}

\begin{proposition}[Aperiodicity and Harris recurrence]\label{prp:chain-ergodicity}
  Assume the transition density satisfies $k(V,\tilde V)>0$ for all $V,\tilde V\in\RR^n_+$.
  Then the $V$-marginal chain is $\psi$-irreducible, aperiodic, and Harris recurrent.
\end{proposition}

\begin{proof}
  Strict positivity of $k$ implies $\psi$-irreducibility (e.g., with $\psi$ being the Lebesgue measure on $\RR^n_+$)
  and aperiodicity; see, e.g., \cite{roberts1994}.
  Since the chain admits the invariant probability measure $\pi(\cdot\mid Y)$, it follows that the chain is Harris recurrent;
  see \cite[Corollary~2]{tierney1994}.
\end{proof}

By \cite[Theorem~2.1]{roberts1997}, the Gibbs chain is geometrically ergodic if and only if the Markov operator $\Lambda$ admits a \emph{positive spectral gap} at $1$:
\[
  \sup\bigl\{|\lambda|:\lambda\in \mathrm{spec}(\Lambda),\ \lambda\neq 1\bigr\}<1 .
\]
Proposition~\ref{prp:Lambda-properties} shows that the operator $\Lambda$ defined in \eqref{eq:Lambda-operator} is self-adjoint, positive semidefinite, and a contraction on $L_2(\pi)$ satisfying $\Lambda \mf{1}=\mf{1}$; consequently, $\mathrm{spec}(\Lambda)\subset[0,1]$ with $1$ being an eigenvalue.
Moreover, Proposition~\ref{prp:chain-ergodicity} establishes that the chain is Harris recurrent; combined with the existence of the unique invariant probability measure $\pi(\cdot\mid Y)$, this ensures that the eigenvalue $1$ is simple.

\begin{proposition}[Trace-class operators have an $L_2(\pi)$ spectral gap]\label{prop:tc-implies-geo}
  Assume that $\Lambda$ is the Markov operator on $L_2(\pi)$ induced by the $V$-marginal transition density $k$
  as in \eqref{eq:Lambda-operator}. If $\Lambda$ is trace-class on $L_2(\pi)$, then
  $\Lambda$ admits a positive spectral gap in $L_2(\pi)$ at $1$; that is,
  \[
    \lambda^\star := 1-\|\Lambda_0\| \in (0,1],
    \qquad
    \Lambda_0:=\Lambda\big|_{L_{2,0}(\pi)},\quad L_{2,0}(\pi):=\mf{1}^\perp.
  \]
  Consequently, the chain $(V^{(t)})$ is geometrically ergodic in $L_2(\pi)$.
  Moreover, for any probability measure $\nu \ll \pi$ with density $d\nu/d\pi \in L_2(\pi)$, that is, for any probability measure $\nu\in\mathcal{L}_2(\pi)$,
  there exists a finite constant $C(\nu)$ such that
  \[
    \|\nu\Lambda^m-\pi\|_{\mathrm{TV}}\le C(\nu)\,(1-\lambda^\star)^m,\qquad m\ge 0.
  \]
\end{proposition}

\begin{proof}
  By Proposition~\ref{prp:Lambda-properties}, $\Lambda$ is self-adjoint, positive semidefinite, and a contraction on $L_2(\pi)$;
  thus $\mathrm{spec}(\Lambda)\subset[0,1]$ and $\Lambda\mf{1}=\mf{1}$.
  Since $\Lambda$ is trace-class, it is necessarily compact; consequently, its spectrum consists only of eigenvalues with finite
  multiplicity, having $0$ as the sole possible accumulation point. Because $1$ is a simple eigenvalue (as established above), all remaining eigenvalues
  belong to $[0,1)$, which yields $\|\Lambda_0\|<1$ and hence $\lambda^\star>0$.
  Geometric ergodicity in $L_2(\pi)$ and the stated total variation bound follow from standard spectral methods; see, for instance, \cite[Theorem~2.1]{roberts1997}.
\end{proof}

Our next goal is to relate trace-classness to an integrability condition on the transition density $k$; to this end, we introduce the standardized kernel
\[
  \hat{k}(V,\tilde V):=\frac{k(V,\tilde V)}{\pi(\tilde V\mid Y)}.
\]
This allows us to express the action of $\Lambda$ as an integral with respect to the measure $\pi(d\tilde V) = \pi(\tilde V|Y) d\tilde V$:
\[
  (\Lambda f)(V) = \int f(\tilde{V}) k(V, \tilde{V}) d\tilde{V} = \int f(\tilde{V}) \hat{k}(V, \tilde{V}) \pi(d\tilde{V}).
\]
By \cite[Theorem~2]{qin2019}, $\Lambda$ is trace-class on $L_2(\pi)$ if and only if
\[
  \int_{\RR^n_+} \hat{k}(V,V) \pi(dV) < \infty,
\]
which reduces to the following condition on the transition density:
\begin{equation}\label{eq:tc-diagonal}
  \int_{\RR^n_+} k(V,V)\,dV < \infty .
\end{equation}
By Proposition~\ref{prop:tc-implies-geo}, establishing the $L_2(\pi)$ geometric ergodicity of the $V$-marginal Gibbs chain reduces to verifying the integrability condition~\eqref{eq:tc-diagonal}. To show that this verification can be performed in either the centered or non-centered formulation, the next proposition establishes the equivalence of their transition kernels.
\begin{proposition}[Equivalence of transition kernels]\label{prop:kernel-equivalence}
  The transition kernels induced by the centered and non-centered parameterizations coincide; that is, for all $V \in \RR^n_+$ and $Y \in \RR^m$,
  \begin{equation*}
    \int_{\RR^n} \pi(\tilde V \mid W, Y) \pi(W \mid V, Y) \, dW = \int_{\RR^n} \pi(\tilde V \mid M, Y) \pi(M \mid V, Y) \, dM.
  \end{equation*}
\end{proposition}

\begin{proof}
  Fix any $V\in\RR^n_+$ and $Y\in\RR^m$, and assume that $K(\zeta)$ is invertible.
  Consider the bijection
  \[
    T:\RR^n\to\RR^n,\qquad M=T(W):=KW+\mu h,
  \]
  and let $|K|:=|\det K|$ denote the absolute value of the Jacobian determinant.
  In the centered parameterization, the full conditional distribution of $\tilde V$ given $(W,Y)$ satisfies
  \[
    \pi(\tilde V\mid W,Y)\ \propto\ \pi(\tilde V)\,\pi(W\mid \tilde V)\,\pi(Y\mid W).
  \]
  Since $\pi(Y\mid W,V)$ does not depend on $V$, we have $\pi(Y\mid W,V) = \pi(Y\mid W)$; consequently, $Y$ and $V$ are conditionally independent given $W$, which yields
    $\pi(\tilde V\mid W,Y)=\pi(\tilde V\mid W)$.
  Moreover, the dependence of $\pi(W\mid \tilde V)$ on $W$ enters only through the latent relation involving $KW+\mu h$; hence $\pi(\tilde V\mid W)$ depends on $W$ only through $M=KW+\mu h$. Therefore, since both $\pi(\tilde V| W,V)$ and $\pi(\tilde V| M,V)$ are probability densities, we have that
  \[
    \pi(\tilde V\mid W,Y)=\pi(\tilde V\mid M,Y)\qquad \text{whenever } M=KW+\mu h.
  \]

  Let $\pi_W(\cdot\mid V,Y)$ denote the conditional density of $W$ given $(V,Y)$, and let $\pi_M(\cdot\mid V,Y)$ denote the conditional density of $M=T(W)$ given $(V,Y)$.
  By the change-of-variables formula,
  \[
    \pi_M(M\mid V,Y) = \pi_W\!\bigl(T^{-1}(M)\mid V,Y\bigr)\,|K|^{-1},
    \qquad dM = |K|\,dW.
  \]
  Applying this transformation to the kernel integral, we obtain
  \begin{align*}
    \int_{\RR^n} \pi(\tilde V\mid W,Y)\,\pi_W(W\mid V,Y)\,dW
     & =\int_{\RR^n} \pi(\tilde V\mid T(W),Y)\,\pi_W(W\mid V,Y)\,dW \\
     & =\int_{\RR^n} \pi(\tilde V\mid M,Y)\,\pi_M(M\mid V,Y)\,dM    \\
     & =\int_{\RR^n} \pi(\tilde V\mid M,Y)\,\pi(M\mid V,Y)\,dM,
  \end{align*}
  which establishes the claimed equality.
\end{proof}

\subsection{Main results on trace-class and geometric ergodicity}\label{sec:trace-class}

This section presents our main results concerning (i) the trace-class property of the $V$-marginal Markov operator $\Lambda$ defined in \eqref{eq:Lambda-operator}, and (ii) the geometric ergodicity of the Markov chain $(V^{(t)})$ induced by the Gibbs sampler. The full proofs are deferred to Section~\ref{chp:proof}.

\begin{theorem}[Trace-class property]\label{thm:trace-class}
  Consider the LLnGM with Gibbs sampler under either the centered or the non-centered parameterization.
  The Markov operator $\Lambda$ defined in \eqref{eq:Lambda-operator} is trace-class on $L_2(\pi)$ if either of the following conditions holds:
  \begin{enumerate}
    \item $a>0$, $b>0$, and $p\in\mathbb{R}$;
    \item $a>0$, $b=0$, and $p>\tfrac{1}{2}$.
  \end{enumerate}
\end{theorem}

For parameter regimes not covered by Theorem~\ref{thm:trace-class}, we will establish geometric ergodicity via a drift-minorization argument. To state this result, we first introduce the following assumption.


\begin{assumption}[Null-smallness assumption]
  \label{ass:null-smallness-A}
  Consider either the centered or the non-centered parameterization introduced in Section \ref{sec:model} with $K$ invertible.
  Let $r := \dim\Null(A)$. If $r > 0$, let $U_A\in\mathbb R^{n\times r}$ be an orthonormal basis of $\Null(A)$ and define
  \[
    G:=U_A^\top K^\top K\,U_A \in\mathbb R^{r\times r},\qquad
    z:=U_A^\top K^\top \mathbf 1 \in\mathbb R^{r}.
  \]
  We assume that either $r = 0$, or
  \begin{equation}\label{eq:null-smallness-A}
    \frac{|\mu|}{\sqrt{\sigma^2 a+\mu^2}}\;\sqrt{\,z^\top G^{-1}z\,}<1.
  \end{equation}
  Note that this condition is automatically satisfied when $\mu=0$.
\end{assumption}

\begin{remark}
  \label{rem:null-smallness-scalability}
  The null-smallness ratio in \eqref{eq:null-smallness-A} imposes a geometric constraint on the interaction between the drift direction $\mathbf{1}$, the prior structure matrix $K$, and the null space of the design matrix $A$.
  The term $\sqrt{z^\top G^{-1}z}$ quantifies the magnitude of the projection of the transformed vector $K^\top \mathbf{1}$ onto the subspace $\Null(A)$, under the geometry induced by the restricted Gram matrix $G=U_A^\top K^\top K\,U_A$ on $\Null(A)$.
  Specifically $z = U_A^\top (K^\top \mathbf 1)$ represents the coefficients of $K^\top \mathbf{1}$ when projected onto the basis of $\Null(A)$, and $G = U_A^\top K^\top K U_A$ represents the Gram matrix of the basis vectors transformed by $K$, capturing the curvature of the prior precision restricted to the null space.
  Therefore, any dependence on the dimension $n$ is driven by how strongly the vector $K^\top \mathbf{1}$ aligns with $\Null(A)$ relative to the conditioning of $K^\top K$ on that subspace.

  To guarantee scalability, specifically for the term $\sqrt{z^\top G^{-1}z}$ to remain bounded as ${n\to\infty}$, it suffices that the drift direction does not increasingly concentrate in the unidentifiable subspace. For instance, the following structural assumptions are sufficient for dimension-free behavior:
  \begin{enumerate}
    \item[(i)] \emph{Orthogonality of the constant mode:} $\mathbf 1 \perp \Null(A)$ (equivalently $\mathbf 1\in\Range(A^\top)$), which implies $U_A^\top\mathbf 1=0$.
    \item[(ii)] \emph{Structure preservation by $K$:} The matrix $K$ does not map the constant vector entirely into the null space. Specifically, assume $K^\top\mathbf 1 = c\,\mathbf 1 + r$, where $c$ is a scalar and the remainder $r$ has a bounded projection onto the null space: $\|U_A^\top r\|_2 = O(1)$.
    \item[(iii)] \emph{Uniform non-degeneracy:} The matrix $K$ remains well-conditioned when restricted to $\Null(A)$, i.e., $\lambda_{\min}(G) = \lambda_{\min}(U_A^\top K^\top K U_A) \ge c_0 > 0$ uniformly in $n$.
  \end{enumerate}
\end{remark}

Having established the null-smallness condition, we now state our main result on geometric ergodicity of the marginal $V$-chain under the drift-minorization framework. The proof is deferred to Section~\ref{sec:geometric-ergodicity-proof}.

\begin{proposition}[Geometric ergodicity via drift-minorization]\label{prop:geo-dm-cases}
  Consider the Gibbs sampler induced by \eqref{eq:gibbs-sampler-non-centered} on $V\in(0,\infty)^n$. Since the conditional density $\pi(M\mid V,Y)$ is everywhere positive and continuous in $V$ for the Gaussian model considered, the following results hold.
  Recall that $\tilde a:=a+\mu^2/\sigma^2$. 
  Then the marginal chain $(V^{(t)})_{t\ge0}$ is geometrically ergodic in each of the following parameter regimes:
  \begin{enumerate}
    \item[\textup{(I)}] $a=\mu=0$ (equivalently $\tilde a=0$). In this regime, necessarily $p<0$ and $b>0$,
          and $\tilde V_i\mid M_i$ is \textup{Inv-Gamma}. The conclusion holds unconditionally.

    \item[\textup{(II)}] $a=0$, $b>0$, $p<0$, and $\mu\neq0$ (so that $\tilde a=\mu^2/\sigma^2>0$),
          with Assumption~\ref{ass:null-smallness-A} satisfied. For $a=0$, Assumption~\ref{ass:null-smallness-A} reduces to
            $\sqrt{z^\top G^{-1}z}<1$.
          In this regime, $\tilde V_i\mid M_i$ is \textup{GIG} with parameters $(p-\frac12,\tilde a,b+(M_i/\sigma)^2)$.

    \item[\textup{(III)}] $a>0$, $b=0$, $p>0$ (so that $\tilde a=a+\mu^2/\sigma^2>0$),
          with Assumption~\ref{ass:null-smallness-A} satisfied.
          In this regime, $\tilde V_i\mid M_i$ is \textup{GIG} with parameters $(p-\frac12,\tilde a,(M_i/\sigma)^2)$.
  \end{enumerate}
\end{proposition}

The preceding results, Theorem~\ref{thm:trace-class} and Proposition~\ref{prop:geo-dm-cases}, together characterize geometric ergodicity over complementary regions of the parameter space.
We now synthesize these findings into a single comprehensive statement, which provides a complete map of geometric ergodicity across all parameter regimes (Table~\ref{tab:geom-ergodicity}).

\begin{theorem}[Geometric ergodicity over the full parameter space]\label{thm:geo-all-cases}
  Consider the Gibbs sampler induced by \eqref{eq:gibbs-sampler-non-centered} on $V\in(0,\infty)^n$ and assume that $\pi(M\mid V,Y)$ is everywhere positive and continuous in $V$.
  Then the marginal chain $(V^{(t)})_{t\ge0}$ is geometrically ergodic in each of the parameter regimes listed in Table~\ref{tab:geom-ergodicity}.
\end{theorem}

\begin{proof}
  The proof proceeds by case analysis corresponding to the regimes in Table~\ref{tab:geom-ergodicity}.
  
  For the cases $a>0, b>0, p\in\RR$ and $a>0, b=0, p>\tfrac12$, Theorem~\ref{thm:trace-class} establishes that the operator $\Lambda$ is trace-class, and geometric ergodicity then follows from Proposition~\ref{prop:tc-implies-geo}.

  For the remaining regimes, we appeal to the drift-minorization results of Proposition~\ref{prop:geo-dm-cases}.
  Specifically: case (I) covers $a=0, b>0, p<0, \mu=0$; case (II) covers $a=0, b>0, p<0, \mu\neq0$ under Assumption~\ref{ass:null-smallness-A}; and case (III) covers $a>0, b=0, p\le \tfrac12$ under Assumption~\ref{ass:null-smallness-A}.
\end{proof}

\begin{table}[t]
  \centering
  \caption{Geometric ergodicity and trace class property of the Gibbs sampler across parameter regimes}
  \label{tab:geom-ergodicity}
  \begin{tabular}[t]{lll}
    \toprule
    Parameter regime & Trace-class  & Geometric ergodicity                                     \\
    \midrule
    $a>0, b>0, p \in \RR$            & Yes        & Yes                                                      \\
    $a>0, b=0, p>1/2$                & Yes        & Yes                                                      \\
    $a>0, b=0, 0 < p\le 1/2$         & Unknown    & Yes$^*$ (requires Assumption~\ref{ass:null-smallness-A}) \\
    $a=0, b>0, p<0, \mu = 0$         & No         & Yes                                                      \\
    $a=0, b>0, p<0, \mu \neq 0$      & Unknown    & Yes$^*$ (requires Assumption~\ref{ass:null-smallness-A}) \\
    \bottomrule
  \end{tabular}
\end{table}

An important special case is the NIG model, which we state as a corollary.

\begin{corollary}[Geometric ergodicity of the NIG field]
  \label{cor:ergodic-nig}
  Suppose that the mixing variables follow an Inverse Gaussian distribution, $V_i \sim \text{IG}(a, b)$ for $a, b > 0$, corresponding to $p = -1/2$ in the GIG parameterization.
  Then the latent field $W$ follows a NIG distribution, the operator $\Lambda$ is trace-class, and the Gibbs sampler is geometrically ergodic in the $L_2(\pi)$-sense for all values of the drift parameter $\mu$.
\end{corollary}

\begin{proof}
  Since $V_i \sim \text{IG}(a, b)$ implies $a > 0$, $b > 0$, and $p = -1/2$, the result follows immediately from Theorem~\ref{thm:trace-class} (case 1) and Proposition~\ref{prop:tc-implies-geo}.
\end{proof}

\section{Stochastic gradient estimation}\label{chp:estimation}

This section explains how the Gibbs sampler developed in Section~\ref{sec:gibbs} can be used to construct stochastic gradient estimators for maximum-likelihood inference of LLnGMs.

\subsection{Likelihood gradient via Fisher's identity}\label{sec:fisher}

Let $\theta \in \Theta$ denote the vector of all model parameters and let $p_\theta(Y, Z)$ be the joint density of the observed data $Y$ and the latent variables $Z$. As discussed in Section~\ref{sec:gibbs}, we may use either the \emph{centered} parametrization with $Z=(W,V)$ or the \emph{non-centered} parametrization with $Z=(M,V)$. The marginal likelihood of the observed data is 
\[
  p_\theta(Y) = \int p_\theta(Y,Z) \, \lambda(dZ),
\]
where $\lambda$ is a dominating measure on the latent space. The corresponding log-likelihood is $\ell(\theta; Y) := \log p_\theta(Y)$.

Direct maximization of $\ell(\theta; Y)$ is challenging because the integral defining $p_\theta(Y)$ is generally intractable, precluding a closed-form expression for the log-likelihood. To perform maximum likelihood estimation using gradient-based optimization, we instead seek to compute the gradient $\nabla_\theta \ell(\theta; Y)$. While the log-likelihood itself is intractable, its gradient can be expressed as a posterior expectation via Fisher's identity.

\begin{proposition}[Fisher's identity {\cite[Appendix D.2]{douc2013nonlinear}}]\label{prp:fisher}
  Assume that $p_\theta(Y,Z)$ is differentiable with respect to $\theta$ and that the order of differentiation and integration can be interchanged (a standard dominated-convergence condition). Then,
  \[
    \nabla_\theta \ell(\theta;Y)
    = \int \nabla_\theta \log p_\theta(Y,Z) \, p_\theta(Z \mid Y) \, \lambda(dZ)
    = \mathbb{E}_\theta \bigl[ \nabla_\theta \log p_\theta(Y,Z) \mid Y \bigr].
  \]
\end{proposition}

The complete-data log-likelihood $\log p_\theta(Y,Z)$ is available in closed form, owing to the conditionally Gaussian structure of the observation model and latent field, and the parametric mixing distribution for $V$. The expectation in Proposition~\ref{prp:fisher}, however, remains analytically intractable. We therefore approximate it using Monte Carlo integration with samples drawn from the posterior distribution.

Specifically, let $\{(V^{(j)}, W^{(j)})\}_{j=1}^k$ be samples from the joint posterior $p_\theta(V, W \mid Y)$ generated by the Gibbs sampler described in Section~\ref{sec:gibbs}. We define the complete-data score function as
\[
  g_\theta(Y, V, W) := \nabla_\theta \log p_\theta(Y, V, W).
\]
A natural Monte Carlo estimator for $\nabla_\theta \ell(\theta; Y)$ is then given by
\begin{equation}\label{eq:mc-grad}
  \widehat{\nabla \ell}_k(\theta; Y)
  = \frac{1}{k} \sum_{j=1}^k g_\theta(Y, V^{(j)}, W^{(j)}).
\end{equation}
Under the non-centered parametrization, an analogous estimator is constructed using the Gibbs samples $(V^{(j)}, M^{(j)})$ and the corresponding complete-data score $\nabla_\theta \log p_\theta(Y, V, M)$.

\subsection{Rao--Blackwellized gradient estimator and Markov chain ergodicity}\label{sec:rb}

Geometric ergodicity derived in Section~\ref{sec:geometric-ergodicity-proof} applies to the $V$-marginal of the Gibbs sampler. 
\begin{color}{black}
It is important to emphasize that throughout each application of the Gibbs sampler,
the parameter vector $\theta$ is held fixed.
That is, for a given iterate $\theta^{(t)}$ in the stochastic gradient procedure,
the Gibbs sampler targets the posterior distribution $p_{\theta^{(t)}}(V,W \mid Y)$
(or $p_{\theta^{(t)}}(V,M \mid Y)$ in the non-centered parametrization).
The parameter $\theta$ is updated only between Gibbs sampling stages via the stochastic
gradient update, and not within the Markov chain itself.
Consequently, all ergodicity and convergence results apply conditionally on a fixed
value of $\theta$, as required by the Markov chain theory invoked below.
\end{color}
This theoretical guarantee ensures that the chain $\{V^{(j)}\}$ converges to the target marginal posterior $p_\theta(V \mid Y)$ at a geometric rate. However, the basic Monte Carlo estimator $\widehat{\nabla \ell}_k(\theta; Y)$ in \eqref{eq:mc-grad} involves the joint samples $\{(V^{(j)}, W^{(j)})\}$. To rigorously apply the ergodic limit theorems established for the $V$-chain, it is desirable to construct a gradient estimator that depends \emph{solely} on the sequence $\{V^{(j)}\}$. 
Fortunately, the conditional model structure allows us to integrate out the latent field $W$ (or $M$) analytically given $(V,Y)$, yielding a function of $V$ alone. This process, known as Rao--Blackwellization, not only aligns the estimator with our theoretical results for $V$ but also reduces the variance of the estimator.

By the law of total variance, conditioning on a sub-sigma-algebra reduces the variance of an estimator: if $\delta(X)$ is an integrable random variable, then
\[
  \mathrm{Var}\!\bigl(\mathbb{E}[\delta(X) \mid Y]\bigr) \le \mathrm{Var}(\delta(X)).
\]
In our setting, the latent field $W$ (or $M$ in the non-centered parametrization) is Gaussian conditional on $(V, Y)$. This structure allows us to integrate out the Gaussian layer analytically, defining the Rao--Blackwellized score as:
\[
  g^{\mathrm{RB}}_\theta(Y,V)
  :=
  \mathbb{E}_\theta\!\left[\nabla_\theta \log p_\theta(Y,V,W) \mid Y, V\right].
\]
By the tower property, this score remains an unbiased estimator of the gradient:
\[
  \mathbb{E}_\theta\!\left[g^{\mathrm{RB}}_\theta(Y,V) \mid Y\right]
  =
  \mathbb{E}_\theta\!\left[\nabla_\theta \log p_\theta(Y,V,W) \mid Y\right]
  =
  \nabla_\theta \ell(\theta; Y).
\]
Consequently, we can construct the Rao--Blackwellized gradient estimator using only the $V$-component of the Gibbs samples, $\{V^{(j)}\}_{j=1}^k$:
\begin{equation}\label{eq:rb-grad}
  \widehat{\nabla \ell}^{\mathrm{RB}}_k(\theta; Y)
  =
  \frac{1}{k} \sum_{j=1}^k g^{\mathrm{RB}}_\theta(Y, V^{(j)}).
\end{equation}
A key computational advantage is that $g^{\mathrm{RB}}_\theta(Y,V)$ typically admits a closed-form expression involving only the conditional mean and covariance of $W \mid (V, Y)$, eliminating the need to sample $W$ for gradient estimation (except as a step in the Gibbs sampling of $V$).

The convergence of this estimator is guaranteed by the ergodic properties of the Markov chain. Suppose the chain $\{V^{(j)}\}$ is (geometrically) ergodic with invariant distribution $p_\theta(V \mid Y)$ and that the score function satisfies the integrability condition
\begin{equation}\label{eq:L1-condition}
  \mathbb{E}_\theta\!\left[\bigl\|g^{\mathrm{RB}}_\theta(Y,V)\bigr\| \mid Y\right] < \infty.
\end{equation}
Then, the Markov chain ergodic theorem ensures almost sure convergence:
\[
  \widehat{\nabla \ell}^{\mathrm{RB}}_k(\theta; Y)
  \longrightarrow
  \mathbb{E}_\theta\!\left[g^{\mathrm{RB}}_\theta(Y,V) \mid Y\right]
  =
  \nabla_\theta \ell(\theta; Y),
  \qquad \text{as } k \to \infty.
\]
\begin{color}{black}
A subtle but important point is that, in practice, the gradient estimator
$g^{\mathrm{RB}}_\theta(Y,V)$ is evaluated using a finite-length Gibbs sampler targeting
$p_\theta(V\mid Y)$. Unless the Markov chain is initialized strictly in stationarity, the resulting
gradient estimate is not exactly unbiased for finite Gibbs runs, but rather
asymptotically unbiased. The residual bias arises from the deviation of the chain's finite-step 
marginal distribution from the target stationary distribution.

Nevertheless, convergence of the resulting stochastic gradient scheme can still be
guaranteed. Classical stochastic approximation theory allows for biased gradient estimators,
provided that the bias decays sufficiently fast relative to the step-size sequence.
Specifically, if the SGD update admits the decomposition
\[
\theta^{(t+1)} = \theta^{(t)} - \gamma_t \bigl(
\nabla f(\theta^{(t)}) + \xi_t + b_t
\bigr),
\]
where $\xi_t$ is a martingale-difference noise term and $b_t$ is the bias
induced by incomplete MCMC convergence, almost sure convergence is preserved as
long as the bias sequence satisfies the summability condition
\(
\sum_{t=1}^\infty \gamma_t \| b_t \| < \infty
\)
\cite[Theorem 2.1]{kushner2003stochastic}.

In the present context, we rely on the geometric ergodicity of the Gibbs sampler.
Under mild regularity assumptions, such as moment conditions ensuring sufficient control of the
Monte Carlo error, geometric ergodicity established via a drift condition is known to yield
geometric decay of the bias associated with finite-length MCMC runs; see, for example,
\cite[Fact~10]{Roberts2004} for representative conditions of this type.
In particular, the bias of the gradient estimator after $k$ Gibbs iterations satisfies
\[
\| b_t \| \le M(\theta^{(t)}, V_{\mathrm{init}})\, \rho^{k},
\]
where $\rho<1$ denotes a contraction rate and $M$ is a finite constant depending on the current
parameter value and the initialization of the Markov chain.
Consequently, under standard step-size schedules satisfying
$\sum_t \gamma_t = \infty$ and $\sum_t \gamma_t^2 < \infty$, the accumulated contribution of the
MCMC bias is asymptotically negligible and does not affect the limiting dynamics of the
stochastic gradient algorithm.

\end{color}

\subsection{Integrability of score functions and the role of parametrization}\label{sec:L1}

To ensure the almost sure convergence of the gradient estimator, the integrability condition~\eqref{eq:L1-condition} must be satisfied under the invariant distribution $p_\theta(V \mid Y)$. Although the marginal chain for $V$ is the same regardless of the parametrization used for the latent field, the functional form of the score $g^{\mathrm{RB}}_\theta(Y,V)$ differs significantly between the centered and non-centered representations. This choice matters because verifying integrability can be challenging if the score function exhibits singularities or effectively undefined regions in the domain of $V$.


Under the \emph{centered} parametrization, the complete-data score vector inherently contains terms
proportional to inverse powers of the mixing variables. In particular, the score with respect to
$\mu$ is
\begin{align*}
  \partial_\mu \log p_\theta(Y,V,W)
   & = \frac{1}{\sigma^2} \sum_{i=1}^n \left[
    (KW)_i - \mu V_i + 2\mu h_i - h_i\frac{ (KW)_i+ \mu h_i}{V_i} 
    \right],
\end{align*}
so verifying the regularity condition \eqref{eq:L1-condition} requires control of posterior
expectations of negative moments of the mixing variables (explicit expressions for all score
components under both parametrizations are given in Appendix~\ref{app:gradients}).
The following lemma provides a sharp condition for the existence of these posterior expectations, revealing an integrability issue when $\alpha\le 1/2$.

\begin{lemma}
\label{lem:L1_threshold_centered}
Fix $i \in \{1, \dots, n\}$ and assume that $K$ is invertible. Under the condition $0 < \pi(y) < \infty$, the posterior expectation 
\[
\mathbb{E} \left[ \left| \frac{(KW)_i + \mu h_i}{V_i} \right| \,\middle|\, Y=y \right]
\]
is finite if and only if $\alpha > 1/2$. In particular, this expectation diverges whenever $\alpha \le 1/2$.
\end{lemma}

The proof is given in Appendix~\ref{app:proofs-estimation}. The lemma shows that this obstruction is not merely a prior artefact: whenever the likelihood does
not suppress mass near $V_i\approx 0$ on events where the coefficient $|h_i\bigl((KW)_i+\mu h_i\bigr)|$ is bounded away
from zero, the posterior expectation of the centered score component becomes infinite. In the
centered formulation, both conditions arise naturally in many settings, since the likelihood need
not decay rapidly as $V_i\to 0$ and the coefficients $h_i\bigl((KW)_i+\mu h_i\bigr)$ are typically
non-degenerate on sets of positive probability. Consequently, the regularity condition \eqref{eq:L1-condition} is not
automatic and may fail without additional, model-specific assumptions controlling the behaviour of
the likelihood near the origin.

However, the \emph{non-centered} parametrization introduced in Section~\ref{sec:gibbs} 
can be used to circumvent these difficulties. A key advantage of this representation is that the resulting score functions typically have a much more benign algebraic structure with respect to $V$. In particular, the gradients do not involve singular negative powers of $V$; instead, they generally behave polynomially. For instance, the corresponding complete-data score in the non-centered parametrization is given by
\begin{align*}
  \partial_\mu \log p_\theta(Y,V,M)
  = \frac{1}{\sigma^2}\sum_{i=1}^n (M_i-\mu V_i)
  + \frac{1}{\sigma_\epsilon^2}(A K^{-1} h)^\top
  \left( Y - X \beta - A K^{-1} (M - \mu h) \right).
\end{align*}

Because standard drift conditions (such as those in Lemma~\ref{lem:drift-a0-mune0} and Lemma~\ref{lem:drift-nullsmall}) readily provide bounds on polynomial moments of $V$, verifying the integrability condition becomes straightforward. For this reason, we rely on the non-centered parametrization to theoretically justify the convergence of our stochastic gradient estimators.

We remark that Lemma~\ref{lem:L1_threshold_centered} pertains to the standard gradient estimator. The Rao-Blackwellized estimator is expected to behave better in both parametrizations due to its inherently lower variance. While this variance reduction can sometimes mitigate integrability issues, we favor the strategy of changing the parametrization, which resolves the problem structurally.


\subsection{Summary}\label{sec:sgd-summary}

We conclude by summarizing how the Gibbs sampler developed in
Section~\ref{sec:gibbs} enables principled SGD estimation for
maximum-likelihood inference of LLnGMs.

Fisher's identity (Section~\ref{sec:fisher}) provides the key link between
likelihood-based inference and posterior simulation by expressing the
log-likelihood gradient as a conditional expectation of the complete-data
score. Although the marginal likelihood is analytically intractable, its
gradient can be estimated using posterior samples of the latent variables,
forming the basis for gradient-based optimization and MCMC-based maximum
likelihood estimation.

Exploiting the hierarchical Gaussian structure further allows the construction
of Rao-Blackwellized gradient estimators (Section~\ref{sec:rb}) by analytically
integrating out the latent Gaussian field conditional on the mixing variables.
The resulting estimator depends only on $V$, reduces Monte Carlo variance, and
aligns naturally with the geometric ergodicity results established for the
$V$-marginal Gibbs chain. Under mild integrability conditions, the ergodic
theorem guarantees almost sure convergence of the gradient estimates.

Crucially, Section~\ref{sec:L1} shows that the validity of this convergence
depends on the integrability of the score function, which in turn is strongly
affected by the chosen parametrization. While the marginal distribution of $V$
is invariant, centered parametrizations typically yield score functions
involving inverse powers of $V$, making integrability difficult to verify, and in some cases not hold. In
contrast, the non-centered parametrization leads to polynomial dependence on
$V$, for which integrability follows directly from standard drift conditions.

Overall, the non-centered parametrization emerges not merely as a computational
device but as a key theoretical ingredient that enables rigorously justified
Rao–Blackwellized SGD estimation for scalable 
maximum-likelihood inference of LLnGMs.

\section{Proof of the main result} \label{chp:proof}

In what follows, for symmetric matrices $A$ and $B$, we write $B \succeq A$ (resp.\ $B \succ A$) to denote the Loewner (resp.\ strict Loewner) partial order, meaning that $B - A$ is positive semi-definite (resp.\ positive definite).

\subsection{Trace-class property proof}\label{sec:proof-trace-class}

We first state three technical lemmas used in the proof of Theorem~\ref{thm:trace-class}; their proofs are deferred to Appendix~\ref{app:proofs}.

\begin{lemma}\label{lem:quadratic_form_lemma}
  Let $A, B \in \mathbb{R}^{n\times n}$ be symmetric positive-definite matrices, and $u, v \in \mathbb{R}^n$. Then
  \begin{equation*}
    \begin{split}
      u^{\top}Au + v^{\top}Bv &- (u+v)^{\top}(A^{-1}+B^{-1})^{-1}(u+v) 
      = (Au - Bv)^{\top}(B+A)^{-1}(Au-Bv).
    \end{split}
  \end{equation*}
  In particular, 
    $u^{\top}Au + v^{\top}Bv \geq (u+v)^{\top}(A^{-1}+B^{-1})^{-1}(u+v)$.
\end{lemma}

\begin{lemma}\label{lem:K_bound}
  For any $p \in \mathbb{R}$, there exists a constant $C > 0$ such that for all $x > 0$,
    $K_{p}(x) > Cx^{-|p|}(1+x)^{-1/2+|p|}e^{-x}$.
\end{lemma}

\begin{lemma}\label{lem:normal_moment_bound}
  Let $p \ge 0$ and $\gamma \ge 0$.
  For any fixed vector $b \in \mathbb{R}^n$, and for 
  all symmetric positive definite matrices $A$,
  there exists a constant
  $C = C(p,n,\|b\|) > 0$, independent of $A$, such that 
  \begin{equation*}
    \begin{split}
      \int \Big(\prod_{i=1}^n (\gamma + x_i^2)^p\Big) & \exp\!\left\{-\tfrac12(x^\top A x - 2 b^\top x)\right\}\,dx                                                                                                      \\
                                                      & \le C\, |A|^{-1/2} \exp\!\left\{\frac12 b^\top A^{-1} b\right\} \prod_{i=1}^n \Big( \gamma^p + [A^{-1}]_{ii}^p \big(1 + \sum_{j=1}^n [A^{-1}]_{jj}^p\big) \Big).
    \end{split}
  \end{equation*}
\end{lemma}

We now proceed with the proof of Theorem~\ref{thm:trace-class}.

\begin{proof}[Proof of Theorem~\ref{thm:trace-class}]
  Throughout the proof, $C$ denotes a generic positive constant, independent of $V$ and $M$, whose value may change from line to line. Without loss of generality, we assume $\sigma=\sigma_\epsilon=1$.
  Let $\tilde p = p - \tfrac{1}{2}$ and $\tilde a = a + \mu^2 > 0$.
  Recalling that $B = AK^{-1}$, we define the  quantities
    $D_V=\mathrm{diag}(V_1,\dots,V_n)$,
    $Q=B^\top B+D_V^{-1}$, and 
    $m=B^\top x+\mu\mathbf 1$,
  where $x:=Y-X\beta+\mu Bh$ is a fixed vector. 
  Recall that the conditional distribution of $M$ is 
  \[
    \pi(M\mid V,Y)=C\,|Q|^{1/2}\exp\!\left\{-\tfrac12(M-Q^{-1}m)^\top Q(M-Q^{-1}m)\right\},
  \]
  and that, conditionally on $M$, the components $V_i$ are independent with densities
  \[
    \pi(V\mid M,Y)
    =
    C\prod_{i=1}^n
    \left[
    \frac{\bigl(\sqrt{\tilde a(b+M_i^2)}\bigr)^{-\tilde p}}{K_{\tilde p}\bigl(\sqrt{\tilde a(b+M_i^2)}\bigr)}
    \,V_i^{\tilde p-1}
    \right]
    \exp\!\left\{-\tfrac12\Bigl(\tilde a\,\mathbf 1^\top D_V\mathbf 1+\sum_{i=1}^n \frac{b+M_i^2}{V_i}\Bigr)\right\}.
  \]

  \paragraph*{Case 1: $a>0$, $b>0$, $p\in\mathbb R$}
  By Lemma~\ref{lem:K_bound}, for a fixed $\tilde p \in \mathbb{R}$, there exists a constant $C > 0$ such that for all $t > 0$,
    $K_{\tilde p}(t) \ge C t^{-|\tilde p|} (1+t)^{-1/2+|\tilde p|} e^{-t}$.
  Setting $t_i = \sqrt{\tilde a(b+M_i^2)}$, we have
  \[
    \frac{t_i^{-\tilde p}}{K_{\tilde p}(t_i)} \le C t_i^{|\tilde p|-\tilde p} (1+t_i)^{1/2-|\tilde p|} e^{t_i}.
  \]
  Define the exponents
    $\eta := (|\tilde p|-\tilde p)/2 \ge 0$ and $\bar\eta := \max \!\left(\nicefrac14-\nicefrac{|\tilde p|}{2}, \, 0\right) \ge 0$.
  Using the bounds $1+t_i \le C(1+b+M_i^2)^{1/2}$ and $t_i \le \sqrt{\tilde ab} + \sqrt{\tilde a}|M_i|$, and noting that $t_i^{|\tilde p|-\tilde p} = (\tilde a(b+M_i^2))^\eta$, we obtain
  \[
    \frac{(\sqrt{\tilde a(b+M_i^2)})^{-\tilde p}}{K_{\tilde p}(\sqrt{\tilde a(b+M_i^2)})} \le C (b+M_i^2)^\eta (1+b+M_i^2)^{\bar\eta} \exp\{\sqrt{\tilde a}|M_i|\}.
  \]
  Substituting this result into the expression for $\pi(V\mid M,Y)$ yields
  \begin{equation}\label{eq:VgivenM_bound_case1}
    \begin{split}
      \pi(V\mid M,Y)
       & \le
      C\Big(\prod_{i=1}^n V_i^{\tilde p-1}(b+M_i^2)^\eta(1+b+M_i^2)^{\bar\eta}\Big) \\
       & \times
      \exp\!\left\{-\tfrac12\Bigl(\tilde a\,\mathbf 1^\top D_V\mathbf 1+b\,\mathbf 1^\top D_V^{-1}\mathbf 1+M^\top D_V^{-1}M\Bigr)
      +\sqrt{\tilde a}\sum_{i=1}^n|M_i|\right\}.
    \end{split}
  \end{equation}

  Setting $\Phi:=\{\pm1\}^n$, we note that $\sum_i |M_i| = \max_{\phi \in \Phi} \phi^\top M$. Using the inequality $\exp(\max_j u_j) \le \sum_j \exp(u_j)$, we have
  \[
    \exp\Bigl(\sqrt{\tilde a}\sum_i|M_i|\Bigr) \le \sum_{\phi\in\Phi} \exp\{\sqrt{\tilde a}\,\phi^\top M\}.
  \]
  By combining this bound with \eqref{eq:VgivenM_bound_case1} and the expression for $\pi(M\mid V,Y)$, we identify the $M$-dependent exponent for each $\phi \in \Phi$ as
    $-\tfrac12 M^\top(Q+D_V^{-1})M + (m+\sqrt{\tilde a}\phi)^\top M$.
  This leads to the following bound on the joint density:
  \begin{equation}\label{eq:joint_bound_before_Mint}
    \begin{split}
      \pi(V\mid M,Y)\pi(M\mid V,Y)
       & \le
      C\,|Q|^{1/2}
      \Big(\prod_i V_i^{\tilde p-1}(b+M_i^2)^\eta(1+b+M_i^2)^{\bar\eta}\Big)                                                             \\
       & \quad\times
      \exp\!\left\{-\tfrac12\Bigl(\tilde a\,\mathbf 1^\top D_V\mathbf 1+b\,\mathbf 1^\top D_V^{-1}\mathbf 1+m^\top Q^{-1}m\Bigr)\right\} \\
       & \quad\times
      \sum_{\phi\in\Phi}
      \exp\!\left\{-\tfrac12\Bigl(M^\top\tilde Q\,M-2(m+\sqrt{\tilde a}\phi)^\top M\Bigr)\right\},
    \end{split}
  \end{equation}
  where $\tilde Q := Q + D_V^{-1} = B^\top B + 2D_V^{-1} \succ 0$. 

  Fix $\phi \in \Phi$ and let $d_\phi := m + \sqrt{\tilde a} \phi$. Defining $f(M) := \prod_{i=1}^n(b+M_i^2)^\eta$ and $g(M) := \prod_{i=1}^n(1+b+M_i^2)^{\bar\eta}$, we apply the Cauchy--Schwarz inequality to obtain:
  \begin{equation*}
    \begin{split}
      \int  f(M)g(M) e^{-\frac{1}{2}(M^\top\tilde Q M-2d_\phi^\top M)} dM 
      \le& \left(\int f(M)^2 e^{-\frac{1}{2}(M^\top\tilde Q M-2d_\phi^\top M)}dM \right)^{1/2} \\
      & \times \left(\int g(M)^2 e^{-\frac{1}{2}(M^\top\tilde Q M-2d_\phi^\top M)}dM \right)^{1/2}.
    \end{split}
  \end{equation*}
  We now apply Lemma~\ref{lem:normal_moment_bound} (with $\gamma=b$, $p=2\eta$ and $\gamma=1+b$, $p=2\bar\eta$) to each factor under the Gaussian kernel $\exp\{-\tfrac12(M^\top\tilde Q M-2d_\phi^\top M)\}$. Since $Q \succeq D_V^{-1}$ implies $Q^{-1} \preceq D_V$, and $\tilde Q \succeq 2D_V^{-1}$ implies $\tilde Q^{-1} \preceq \tfrac12 D_V$, it follows that $[\tilde Q^{-1}]_{ii} \le V_i$ for all $i$. This yields:
  \begin{equation}\label{eq:Mint_bound}
    \begin{split}
      \int \prod_i & (b+M_i^2)^\eta(1+b+M_i^2)^{\bar\eta} \exp\!\left\{-\tfrac12(M^\top\tilde Q M-2d_\phi^\top M)\right\}\,dM \\
       & \le C\,|\tilde Q|^{-1/2}\exp\!\left\{\tfrac12 d_\phi^\top \tilde Q^{-1}d_\phi\right\} 
       \times S_\eta^{n/2}S_{\bar\eta}^{n/2} \prod_i (b^\eta+V_i^\eta)\bigl((1+b)^{\bar\eta}+V_i^{\bar\eta}\bigr),
    \end{split}
  \end{equation}
  where $S_\eta := 1 + \sum_{k=1}^n V_k^{2\eta}$ and $S_{\bar\eta} := 1 + \sum_{k=1}^n V_k^{2\bar\eta}$.
  Integrating \eqref{eq:joint_bound_before_Mint} with respect to $M$, summing over $\phi \in \Phi$, and using the determinantal inequality
  \[
    |Q|^{1/2}|\tilde Q|^{-1/2} = \bigl|I+Q^{-1/2}D_V^{-1}Q^{-1/2}\bigr|^{-1/2} \le 1,
  \]
  we obtain the following bound on the $V$-marginal:
  \begin{equation}
    \begin{aligned}
      \int \pi(V\mid M,Y)\pi(M\mid V,Y)\,dM
       & \le C \Big(\prod_{i=1}^n V_i^{\tilde p-1}(b^\eta+V_i^\eta)\bigl((1+b)^{\bar\eta}+V_i^{\bar\eta}\bigr)\Big) S_\eta^{n/2} S_{\bar\eta}^{n/2} \\
       & \quad \times \exp\!\left\{-\tfrac{1}{2} b\,\mathbf{1}^\top D_V^{-1}\mathbf{1}\right\} \sum_{\phi\in\Phi} \exp\!\left\{-\tfrac{1}{2}\Delta_\phi(V)\right\} \\
       & \le C \max \left\{ \left(\prod_{i=1}^n V_i\right)^{m_1}, \left(\prod_{i=1}^n V_i^{-1}\right)^{m_2} \right\},
    \end{aligned}
    \label{eq:after_Mint}
  \end{equation}
  where $C>0$ is a constant, $m_1, m_2 \in \mathbb{N}$, and the \emph{exponent gap} $\Delta_\phi(V)$ is defined as
  \[
    \Delta_\phi(V) := \tilde a\,\mathbf{1}^\top D_V\mathbf{1} + m^\top Q^{-1}m - (m+\sqrt{\tilde a}\phi)^\top(Q+D_V^{-1})^{-1}(m+\sqrt{\tilde a}\phi).
  \]

  We show that for all $\phi\in\Phi$,
  \begin{equation}\label{eq:Delta_lower}
    \Delta_\phi(V)\ \ge\ \frac38(\sqrt{\tilde a}-|\mu|)^2\sum_{i=1}^n V_i\ -\ \frac12\|x\|_2^2.
  \end{equation}

  Applying Lemma~\ref{lem:quadratic_form_lemma} with $A=D_V$, $B=Q^{-1}$, $u=\sqrt{\tilde a}\phi$, and $v=m$, we obtain:
  \[
    \Delta_\phi(V) = \bigl(\sqrt{\tilde a}D_V\phi-Q^{-1}m\bigr)^\top(D_V+Q^{-1})^{-1}\bigl(\sqrt{\tilde a}D_V\phi-Q^{-1}m\bigr).
  \]
  Since $Q\succeq D_V^{-1}$,  $Q^{-1}\preceq D_V$. Hence, $D_V+Q^{-1}\preceq 2D_V$ and thus
  $(D_V+Q^{-1})^{-1}\succeq \tfrac12 D_V^{-1}$. Therefore,
  \begin{equation}\label{eq:Delta_ge_half}
    \Delta_\phi(V) \ge \tfrac{1}{2} \bigl(\sqrt{\tilde a}D_V\phi-Q^{-1}m\bigr)^\top D_V^{-1}\bigl(\sqrt{\tilde a}D_V\phi-Q^{-1}m\bigr).
  \end{equation}
  Substituting $m = B^\top x + \mu \mathbf{1}$ and defining $z := \sqrt{\tilde a}D_V\phi - \mu Q^{-1} \mathbf{1}$, we find that
  \[
    \sqrt{\tilde a}D_V\phi - Q^{-1}m = z - Q^{-1}B^\top x.
  \]
  Equation \eqref{eq:Delta_ge_half} then becomes:
  \begin{equation}\label{eq:Delta_ge_half2}
    \Delta_\phi(V) \ge \tfrac{1}{2} (z-Q^{-1}B^\top x)^\top D_V^{-1}(z-Q^{-1}B^\top x).
  \end{equation}
  Expanding this expression and applying the inequality $2 u^\top v \le u^\top u + v^\top v$ with $u=x$ and $v=BQ^{-1}D_V^{-1}z$, we obtain
  \begin{equation*}
    \begin{split}
      (z  - Q^{-1}B^\top x)^\top & D_V^{-1}(z-Q^{-1}B^\top x) \\
      &\ge z^\top\Bigl(D_V^{-1}-D_V^{-1}Q^{-1}D_V^{-1}+D_V^{-1}Q^{-1}D_V^{-1}Q^{-1}D_V^{-1}\Bigr)z -\|x\|_2^2,
    \end{split}
  \end{equation*}
  where we have used the relation $B^\top B = Q - D_V^{-1}$ and discarded the non-negative term $x^\top BQ^{-1}D_V^{-1}Q^{-1}B^\top x$. To simplify the notation, let $D := D_V$ and define the symmetric positive definite matrix
    $G := D^{1/2}QD^{1/2} \succ 0$.
  Then $G^{-1}=D^{-1/2}Q^{-1}D^{-1/2}$ and $G^{-2}=D^{-1/2}Q^{-1}D^{-1}Q^{-1}D^{-1/2}$.
  A direct factorization shows that the matrix in parentheses can be rewritten as
  \[
    D^{-1}-D^{-1}Q^{-1}D^{-1}+D^{-1}Q^{-1}D^{-1}Q^{-1}D^{-1}
    =
    D^{-1/2}\bigl(I-G^{-1}+G^{-2}\bigr)D^{-1/2}.
  \]
  Next, observe that for all $t>0$, the scalar inequality
  \[
    1-\frac1t+\frac1{t^2}\ \ge\ \frac34
  \]
  holds, which is equivalent to $\tfrac14 t + t^{-1} \ge 1$. By the spectral functional calculus, this yields the matrix inequality
    $I-G^{-1}+G^{-2}\ \succeq\ \nicefrac34 I$
    for all  $G \succ 0$.
  Consequently, we have:
  \begin{equation*}
    \begin{split}
      z^\top\Bigl(D^{-1}-D^{-1}Q^{-1}D^{-1} &+ D^{-1}Q^{-1}D^{-1}Q^{-1}D^{-1}\Bigr)z \\
      &= z^\top D^{-1/2} \bigl(I-G^{-1}+G^{-2}\bigr) D^{-1/2} z 
      \ge \tfrac34\,z^\top D^{-1}z.
    \end{split}
  \end{equation*}
  This leads to the bound:
  \[
    (z-Q^{-1}B^\top x)^\top D^{-1}(z-Q^{-1}B^\top x) \ge \tfrac34\,z^\top D^{-1}z-\|x\|_2^2.
  \]
  Combining \eqref{eq:Delta_ge_half2} with the above inequality, we obtain:
  \begin{equation}\label{eq:Delta_ge_3over8}
    \Delta_\phi(V) \ge \frac12\Big(\tfrac34\,z^\top D_V^{-1}z-\|x\|_2^2\Big) = \frac{3}{8}\,z^\top D_V^{-1}z - \frac{1}{2}\|x\|_2^2.
  \end{equation}

  Finally, we bound $z^\top D_V^{-1}z$. Since $\phi_i^2=1$, we have:
  \[
    z^\top D_V^{-1}z = \sum_{i=1}^n \frac{(\sqrt{\tilde a}V_i\phi_i-\mu(Q^{-1}\mathbf 1)_i)^2}{V_i}.
  \]
  Applying the weighted Cauchy--Schwarz inequality,
  \[
    \sum_{i=1}^n \frac{r_i^2}{V_i}\ge \frac{\big(\sum_{i=1}^n r_i\big)^2}{\sum_{i=1}^n V_i},
  \]
  with $r_i=\sqrt{\tilde a}V_i-\mu\,\phi_i(Q^{-1}\mathbf 1)_i$, we obtain:
  \[
    z^\top D_V^{-1}z \ge \frac{\bigl(\sqrt{\tilde a}\sum_i V_i-\mu\,\phi^\top Q^{-1}\mathbf 1\bigr)^2}{\sum_i V_i}.
  \]
  Furthermore, using $Q^{-1}\preceq D_V$ and the inequality $2|u^\top v|\le u^\top u+v^\top v$ with $u=Q^{-1/2}\phi$ and $v=Q^{-1/2}\mathbf 1$, it follows that:
  \[
    |\phi^\top Q^{-1}\mathbf 1| \le \tfrac12\phi^\top Q^{-1}\phi+\tfrac12\mathbf 1^\top Q^{-1}\mathbf 1 \le \tfrac12\phi^\top D_V\phi+\tfrac12\mathbf 1^\top D_V\mathbf 1 = \sum_{i=1}^n V_i.
  \]
  This implies that 
    $\bigl|\sqrt{\tilde a}\sum_i V_i-\mu\,\phi^\top Q^{-1}\mathbf 1\bigr| \ge (\sqrt{\tilde a}-|\mu|)\sum_i V_i$,
  and we thus have that 
    $z^\top D_V^{-1}z \ge (\sqrt{\tilde a}-|\mu|)^2\sum_{i=1}^n V_i$.
  Substituting this result into \eqref{eq:Delta_ge_3over8} proves \eqref{eq:Delta_lower}.

  From \eqref{eq:Delta_lower}, we have uniformly in $\phi \in \Phi$:
  \[
    \exp\!\left\{-\tfrac{1}{2}\Delta_\phi(V)\right\} \le \exp\!\left\{\tfrac{1}{4}\|x\|_2^2\right\} \exp\!\left\{-c\sum_{i=1}^n V_i\right\}, \qquad c:=\tfrac{3}{16}(\sqrt{\tilde a}-|\mu|)^2>0.
  \]
  Noting that $\exp\{-\tfrac{1}{2} b\,\mathbf{1}^\top D_V^{-1}\mathbf{1}\} = \exp\{-\tfrac{b}{2}\sum_i V_i^{-1}\}$ and 
  absorbing $\exp\{\|x\|_2^2/4\}$ and $|\Phi|=2^n$ into $C$, we obtain from \eqref{eq:after_Mint}:
  \begin{equation*}
    \begin{split}
      \int \pi(V\mid M,Y)\pi(M\mid V,Y)\,dM 
      &\le C \max \left\{ \left(\prod_{i=1}^n V_i\right)^{m_1}, \left(\prod_{i=1}^n V_i^{-1}\right)^{m_2} \right\} \\
      &\quad \times \exp\!\left\{-c\sum_{i=1}^n V_i-\frac{b}{2}\sum_{i=1}^n V_i^{-1}\right\}.
    \end{split}
  \end{equation*}
  Since $c>0$, the factor $\exp\{-c\sum_i V_i\}$ decays exponentially as any $V_i \to \infty$, dominating the at most polynomial growth of $(\prod_i V_i)^{m_1}$. Similarly, since $b>0$, the factor $\exp\{-\tfrac{b}{2}\sum_i V_i^{-1}\}$ decays exponentially as any $V_i \to 0$, dominating the at most polynomial growth of $(\prod_i V_i^{-1})^{m_2}$. Hence the right-hand side is integrable on $(0,\infty)^n$, and we conclude
  \[
    \int_{(0,\infty)^n}\int_{\mathbb{R}^n} \pi(V\mid M,Y)\pi(M\mid V,Y)\,dM\,dV < \infty,
  \]
  which establishes the trace-class property in Case 1.

  \paragraph*{Case 2: $a>0$, $b=0$, $p>\tfrac12$}

  The argument proceeds identically up to \eqref{eq:after_Mint}, except that the exponential term $\exp\{-\tfrac{b}{2}\sum_i V_i^{-1}\}$ is absent (since $b=0$). The exponential decay as $V_i \to \infty$ remains because $c>0$ still holds (indeed, $a>0$ implies $\sqrt{\tilde a}>|\mu|$). Therefore, we only need to verify integrability near $V_i \to 0$.

  For $p>\tfrac{1}{2}$, we have $\tilde p = p - \tfrac{1}{2} > 0$, and hence $\eta = \tfrac{1}{2}(|\tilde p| - \tilde p) = 0$. Consequently, the leading small-$V$ behavior of the prefactor in \eqref{eq:after_Mint} is
  \[
    \prod_{i=1}^n V_i^{\tilde p-1} = \prod_{i=1}^n V_i^{p-\frac{3}{2}},
  \]
  up to multiplicative factors that remain bounded as $V_i \to 0$ (since $S_\eta \equiv 1$ when $\eta=0$, and $\bar\eta \ge 0$). The $V$-integral near $0$ is therefore finite if and only if $p - \tfrac{3}{2} > -1$, which is equivalent to $p > \tfrac{1}{2}$. This matches the stated hypothesis, and hence the target integral is finite in Case 2.

  This completes the proof of the trace-class property in cases (1)--(2).
\end{proof}

\subsection{Geometric ergodicity proof}\label{sec:geometric-ergodicity-proof}

We establish geometric ergodicity of the marginal Gibbs chain $(V^{(t)})_{t\ge 0}$
by verifying a drift condition and a one-step minorization condition, and then
invoking Rosenthal's drift--minorization theorem \cite[Theorem~12]{rosenthal1995}.
For completeness, we restate it below.

\begin{theorem}[Rosenthal drift--minorization bound {\cite[Theorem~12]{rosenthal1995}}]
  \label{thm:drift-minorization}
  Let $P(x,\cdot)$ be the transition kernel of a Markov chain $(X^{(k)})_{k\ge 0}$ on
  a state space $\mathcal X$ with stationary distribution $\pi$.
  Assume there exist a measurable function $V:\mathcal X\to \mathbb R_{\ge 0}$,
  constants $\lambda\in(0,1)$ and $b<\infty$, such that the drift condition holds:
  \[
    \mathbb E\!\left[V(X^{(1)})\mid X^{(0)}=x\right]\ \le\ \lambda V(x)+b,
    \qquad \forall x\in\mathcal X.
  \]
  Assume further that there exist $\varepsilon>0$, a probability measure $Q(\cdot)$ on
  $\mathcal X$, and $d>2b/(1-\lambda)$ such that the minorization condition holds:
  \[
    P(x,\cdot)\ \ge\ \varepsilon\,Q(\cdot),
    \qquad \forall x\in\mathcal X\ \text{with}\ V(x)\le d.
  \]
  Define
  \[
    \alpha^{-1} := \frac{1+2b+\lambda d}{1+d}\in(0,1),
    \qquad
    A := 1+2(\lambda d+b).
  \]
  Then for any $0<r<1$ and any initial distribution $\nu$,
  \[
    \bigl\| \mathcal L(X^{(k)})-\pi \bigr\|_{\mathrm{TV}}
    \ \le\
    (1-\varepsilon)^{rk}
    \ +\
    \Bigl(\alpha^{-(1-r)}A^{\,r}\Bigr)^k\,
    \Bigl(1+\frac{b}{1-\lambda}+\mathbb E_\nu[V(X^{(0)})]\Bigr).
  \]
  In particular, the chain is geometrically ergodic.
\end{theorem}

This theorem provides a sufficient condition for geometric ergodicity of the Gibbs sampler. We now establish the drift and minorization conditions for our specific model.

For the model in \eqref{eq:non-centered-combined}, we denote $\rho = \sigma^2 \sigma_{\epsilon}^{-2}$, $\bar{Q} = \sigma^2 Q = \rho B^\top B + D_{V}^{-1}$, and $\bar{m} = \sigma^2 m = \rho B^\top (Y - X \beta + \mu B h ) + \mu \mathbf{1}$. The Gibbs sampler has the following transition distributions:
\begin{subequations}
  \label{eq:gibbs-sampler-non-centered}
  \begin{align}
    M|V, Y            & \sim N(\bar{Q}^{-1} \bar{m}, \sigma^2 \bar{Q}^{-1}) \label{eq:gibbs-m}                                                                                    \\
    \tilde V_i |  M_i, Y & \overset{iid}{\sim} \text{GIG}\left(p-\frac 1 2, a+\left(\frac{\mu}{\sigma}\right)^2, b+\left(\frac{M_i}{\sigma}\right)^2\right) \label{eq:gibbs-v-tilde}
  \end{align}
\end{subequations}

We establish the drift conditions by considering the following three cases:
\begin{description}
  \item[Case 1:] $a=0, p<0, b>0, \mu=0$;
  \item[Case 2:] $a=0, p<0, b>0, \mu\neq0$;
  \item[Case 3:] $a>0, b=0, p>0$.
\end{description}

\subsubsection{Case 1}
We first consider the special case $a = \mu = 0$, which implies $b>0$ and $p<0$. In this case, the transition distributions in \eqref{eq:gibbs-sampler-non-centered} reduce to
\begin{subequations}
  \label{eq:gibbs-sampler-non-centered-a0-mu0}
  \begin{align}
    M|V, Y            & \sim N(\bar{Q}^{-1} \bar{m}, \sigma^2 \bar{Q}^{-1})                                                    \\
    \tilde V_i |  M_i, Y & \overset{iid}{\sim} \text{Inv-Gamma}\left(-p + \frac 1 2, \frac{b}{2} + \frac{M_i^2}{2\sigma^2}\right)
  \end{align}
\end{subequations}
Recall that the density of $\text{Inv-Gamma}(\alpha, \beta)$ with shape parameter $\alpha > 0$ and scale parameter $\beta > 0$ is
\[
  f_X(x) = \frac{\beta^{\alpha}}{\Gamma(\alpha)} x^{-(\alpha+1)} \exp\left(-\frac{\beta}{x}\right), \quad x > 0.
\]

In this case, one can show that the Markov operator is not trace class in general. We state this formally in the following lemma (proved in Appendix~\ref{app:proofs}).

\begin{lemma}
\label{lem:not-trace-class}
  Assume $a=\mu=0$ (implying $p<0$ and $b>0$), and $A=0$. Then, the Markov operator associated with the Gibbs sampler is not trace class.
\end{lemma}

\begin{remark}
  The case $\mu=0$ satisfies the null-smallness assumption (Assumption~\ref{ass:null-smallness-A}). Lemma~\ref{lem:not-trace-class} thus demonstrates that this assumption alone is not sufficient to ensure the trace-class property. However, as we show below, it is sufficient to guarantee geometric ergodicity.
\end{remark}

Despite the failure of the trace-class property, we can still establish geometric ergodicity by verifying the drift and minorization conditions. We begin by introducing some technical lemmas (the proof of the first is deferred to Appendix~\ref{app:proofs}).

\begin{lemma}[Uniform boundedness of $\eta(V)$]
  \label{lem:eta-range-bound}
  Let $\bar m \in \mathrm{Range}(B^\top)$. There exists a constant $C < \infty$ such that
  \[
    \left\| \left(\rho B^\top B + D_V^{-1}\right)^{-1} \bar m \right\|_2 \le C,
  \]
  for all $V_1, \ldots, V_n \in (0,\infty)$, where $D_V = \mathrm{diag}(V_1,\ldots,V_n)$.
\end{lemma}

\begin{lemma}\label{lem:baseline-contraction}
  Let $\alpha>0$ and $\delta\in(0,\alpha)$. Then
  \begin{equation}\label{eq:baseline-contraction}
    \frac{\Gamma(\alpha+\frac{1}{2}-\delta)}{\Gamma(\alpha+\frac{1}{2})} \cdot
    \frac{\Gamma(\delta+\frac{1}{2})}{\sqrt{\pi}} < 1.
  \end{equation}
\end{lemma}

\begin{proof}
  Define
  \[
    f(\delta)
    :=\log\Gamma\!\left(\alpha+\tfrac12-\delta\right)
    +\log\Gamma\!\left(\delta+\tfrac12\right)
    -\log\Gamma\!\left(\alpha+\tfrac12\right)
    -\log\Gamma\!\left(\tfrac12\right).
  \]
  Then
    $f''(\delta)=\psi_1\!\left(\alpha+\tfrac12-\delta\right)+\psi_1\!\left(\delta+\tfrac12\right)>0$,
  so $f$ is strictly convex on $(0,\alpha)$.
  Since $f(0)=0$ and $f(\alpha)=0$, strict convexity implies $f(\delta)<0$
  for all $\delta\in(0,\alpha)$.
  Exponentiating yields \eqref{eq:baseline-contraction}.
\end{proof}

To establish the drift condition for the case $a=\mu=0$, we introduce the following notation. Recall that under the assumption $a=\mu=0$ (with $b>0$ and $p<0$), the distributions simplify to the Inverse-Gamma family as in \eqref{eq:gibbs-sampler-non-centered-a0-mu0}. Set $\alpha:=-p>0$ and define
  $\delta:=\min\left\{\nicefrac{\alpha}{2},\nicefrac12\right\}\in\left(0,\nicefrac12\right]$.
We use the Lyapunov function
\[
  G(V):=1+\sum_{i=1}^n\Big(V_i^{\delta}+V_i^{-1/4}\Big).
\]
Recall that $\rho:=\sigma^2\sigma_\epsilon^{-2}$, $\bar Q(V):=\rho B^\top B + D_V^{-1}$, and $\bar m:=\rho B^\top (Y-X\boldsymbol\beta)\in\mathrm{Range}(B^\top)$. The conditional distribution of $M$ given $(V,Y)$ is
\[
  M\mid(V,Y)\sim N\!\bigl(\eta(V),\Sigma(V)\bigr),
  \qquad
  \eta(V):=\bar Q(V)^{-1}\bar m,\quad \Sigma(V):=\sigma^2\bar Q(V)^{-1},
\]
and the conditional distribution of $\tilde V_i$ given $M_i=x$ is
\[
  \tilde V_i\mid (M_i=x)\sim \text{Inv-Gamma}\Bigl(\alpha+\tfrac12,\ \frac{b}{2}+\frac{x^2}{2\sigma^2}\Bigr).
\]
Define the constants
\[
  C_\delta:=\frac{\Gamma(\alpha+\frac12-\delta)}{\Gamma(\alpha+\frac12)},
  \qquad
  K_\delta := \frac{\Gamma(\delta+\frac12)}{\sqrt{\pi}},
  \qquad
  \gamma^{(\delta)}:=C_\delta\,K_\delta,
\]
and
\[
  \phi_{\mathrm{IG}}^{(\delta)}
  :=
  1+nC_\delta\Bigl(\frac{b}{2}\Bigr)^\delta
  +n\,\frac{\Gamma(\alpha+\frac34)}{\Gamma(\alpha+\frac12)}\,\Bigl(\frac{b}{2}\Bigr)^{-1/4}.
\]
Let $C_\eta$ denote the constant from Lemma~\ref{lem:eta-range-bound} such that $\sup_{V\in(0,\infty)^n}\|\eta(V)\|_2 \le C_\eta$, and set
\[
  \psi^{(\delta)}:=\frac{C_\delta}{(2\sigma^2)^\delta}\,n^{1-\delta} C_\eta^{2\delta}.
\]

\begin{lemma}[Drift condition for $a=\mu=0$ (Inv-Gamma)]
  \label{lem:drift-a0mu0}
  Under the assumptions and notation above, for all $V\in(0,\infty)^n$,
  \[
    \mathbb E\!\left[G(\tilde V)\mid V\right]
    \le
    \gamma^{(\delta)}\,G(V)+\bigl(\phi_{\mathrm{IG}}^{(\delta)}+\psi^{(\delta)}\bigr).
  \]
  Moreover, $\gamma^{(\delta)}<1$ by Lemma~\ref{lem:baseline-contraction}, since $\delta\in(0,\alpha)$ with $\alpha=-p$.
\end{lemma}

\begin{proof}
  For the Inverse-Gamma distribution $\text{Inv-Gamma}(\alpha+\tfrac12, \tfrac{b}{2}+\tfrac{x^2}{2\sigma^2})$, the fractional moment formula holds for any $r < \alpha+\tfrac12$ (possibly negative):
  \[
    \mathbb E\big(\tilde V_i^{\,r}\mid M_i=x, Y\big)
    =
    \frac{\Gamma(\alpha+\frac12-r)}{\Gamma(\alpha+\frac12)}
    \Big(\frac{b}{2}+\frac{x^2}{2\sigma^2}\Big)^{r}.
  \]
  Setting $r=\delta$ and using the inequality $(u+v)^\delta\le u^\delta+v^\delta$ for $\delta\in(0,1]$ and $u,v\ge 0$, we obtain
  \begin{equation}\label{eq:IG-frac-moment}
    \mathbb E\big(\tilde V_i^{\,\delta}\mid M_i=x, Y\big)
    \le
    C_\delta\,\Bigl(\frac{b}{2}\Bigr)^\delta
    +
    C_\delta\,\frac{|x|^{2\delta}}{(2\sigma^2)^{\delta}}.
  \end{equation}
  For the negative moment, we apply the same formula with $r=-\tfrac14$:
  \begin{equation}\label{eq:IG-negquarter}
    \mathbb E\big(\tilde V_i^{-1/4}\mid M_i=x, Y\big)
    =
    \frac{\Gamma(\alpha+\frac34)}{\Gamma(\alpha+\frac12)}\,\Bigl(\frac{b}{2}+\frac{x^2}{2\sigma^2}\Bigr)^{-1/4}
    \le
    \frac{\Gamma(\alpha+\frac34)}{\Gamma(\alpha+\frac12)}\,\Bigl(\frac{b}{2}\Bigr)^{-1/4}.
  \end{equation}

  By the definition of the transition kernel, the conditional distribution of $\tilde V$ given $(M,V,Y)$ coincides with the conditional distribution of $\tilde V$ given $(M,Y)$. In other words, $\tilde V$ and $V$ are conditionally independent given $(M,Y)$. Therefore, for any measurable function $h$, we have that 
    $\mathbb E\!\big[h(\tilde V)\mid M,Y\big] = \mathbb E\!\big[h(\tilde V)\mid M,V,Y\big]$.
  By the tower property,
  \[
    \mathbb E\!\Big[\mathbb E\!\big[h(\tilde V)\mid M,Y\big]\Big|\,V\Big]
    =
    \mathbb E\!\Big[\mathbb E\!\big[h(\tilde V)\mid M,V,Y\big]\Big|\,V\Big]
    =
    \mathbb E\!\big[h(\tilde V)\mid V\big].
  \]
  Applying this identity with $h(\tilde V)=\sum_{i=1}^n \tilde V_i^{\delta}$ and $h(\tilde V)=\sum_{i=1}^n \tilde V_i^{-1/4}$, and summing the bounds \eqref{eq:IG-frac-moment} and \eqref{eq:IG-negquarter} over $i$, we obtain
  \begin{align}
    \mathbb E\!\left[\sum_{i=1}^n \tilde V_i^{\delta}\mid V\right]
     & \le
    nC_\delta\Bigl(\frac{b}{2}\Bigr)^\delta
    +\frac{C_\delta}{(2\sigma^2)^\delta}\,
    \mathbb E\!\left[\sum_{i=1}^n |M_i|^{2\delta}\mid V\right],
    \label{eq:sumVdelta-step} \\
    \mathbb E\!\left[\sum_{i=1}^n \tilde V_i^{-1/4}\mid V\right]
     & \le
    n\,\frac{\Gamma(\alpha+\frac34)}{\Gamma(\alpha+\frac12)}\,\Bigl(\frac{b}{2}\Bigr)^{-1/4}.
    \label{eq:sumVinverse-step}
  \end{align}

  Conditionally on $V$, write $M_i=\eta_i(V)+Z_i$ where $Z_i\sim N(0,\Sigma_{ii}(V))$.
  Since $\delta\le \tfrac12$ we have $2\delta\le 1$ and hence
  $|a+b|^{2\delta}\le |a|^{2\delta}+|b|^{2\delta}$. Therefore
  \[
    \mathbb E\big(|M_i|^{2\delta}\mid V\big)\le |\eta_i(V)|^{2\delta} + \mathbb E|Z_i|^{2\delta}.
  \]
  Let $m_{2\delta}:=\mathbb E|Z|^{2\delta}=2^\delta\Gamma(\delta+\tfrac12)/\sqrt{\pi}$ for $Z\sim N(0,1)$.
  Since $Z_i=\Sigma_{ii}(V)^{1/2}Z$, we have
  $\mathbb E|Z_i|^{2\delta}=m_{2\delta}\,\Sigma_{ii}(V)^{\delta}$.
  Summing over $i$ gives
  \begin{equation}\label{eq:sumMi-2delta}
    \mathbb E\!\left[\sum_{i=1}^n |M_i|^{2\delta}\mid V\right]
    \le
    \sum_{i=1}^n |\eta_i(V)|^{2\delta}
    +
    m_{2\delta}\sum_{i=1}^n \Sigma_{ii}(V)^{\delta}.
  \end{equation}

  We now handle the two terms on the right-hand side of \eqref{eq:sumMi-2delta} separately. First, consider the variance term. Since $\bar Q(V)=\rho B^\top B + D_V^{-1}\succeq D_V^{-1}$, we have
  $\bar Q(V)^{-1}\preceq D_V$.
  Thus $(\bar Q(V)^{-1})_{ii}\le V_i$ and hence
  \[
    \Sigma_{ii}(V)=\sigma^2(\bar Q(V)^{-1})_{ii}\le \sigma^2 V_i,
    \qquad
    \sum_{i=1}^n \Sigma_{ii}(V)^\delta \le \sigma^{2\delta}\sum_{i=1}^n V_i^\delta.
  \]
  Next, we bound the mean term. By Lemma~\ref{lem:eta-range-bound}, $\sup_V \|\eta(V)\|_2 \le C_\eta$. Hence
  \[
    \sum_{i=1}^n |\eta_i(V)|^{2\delta} = \|\eta(V)\|_{2\delta}^{2\delta}
    \le \bigl(n^{\frac{1}{2\delta}-\frac12}\|\eta(V)\|_2\bigr)^{2\delta}
    = n^{1-\delta}\|\eta(V)\|_2^{2\delta}
    \le n^{1-\delta}C_\eta^{2\delta}.
  \]
  Plugging these two bounds into \eqref{eq:sumMi-2delta} yields
  \[
    \mathbb E\!\left[\sum_{i=1}^n |M_i|^{2\delta}\mid V\right]
    \le
    n^{1-\delta}C_\eta^{2\delta}
    +
    m_{2\delta}\sigma^{2\delta}\sum_{i=1}^n V_i^\delta.
  \]
  Insert into \eqref{eq:sumVdelta-step}:
  \[
    \mathbb E\!\left[\sum_{i=1}^n \tilde V_i^{\delta}\mid V\right]
    \le
    nC_\delta\Bigl(\frac{b}{2}\Bigr)^\delta
    +
    \frac{C_\delta}{(2\sigma^2)^\delta}n^{1-\delta}C_\eta^{2\delta}
    +
    \frac{C_\delta}{(2\sigma^2)^\delta}m_{2\delta}\sigma^{2\delta}\sum_{i=1}^n V_i^\delta.
  \]
  Note that $m_{2\delta}/2^\delta = \Gamma(\delta+\tfrac12)/\sqrt{\pi}=K_\delta$, hence
  the last coefficient equals $C_\delta K_\delta=\gamma^{(\delta)}$. Therefore
  \[
    \mathbb E\!\left[\sum_{i=1}^n \tilde V_i^{\delta}\mid V\right]
    \le
    nC_\delta\Bigl(\frac{b}{2}\Bigr)^\delta
    +
    \psi^{(\delta)}
    +
    \gamma^{(\delta)}\sum_{i=1}^n V_i^\delta.
  \]

  Finally, we combine with \eqref{eq:sumVinverse-step} and use $\sum_i V_i^\delta\le G(V)$ to obtain
  \[
    \mathbb E[G(\tilde V)\mid V]\le \gamma^{(\delta)}\,G(V)+\bigl(\phi_{\mathrm{IG}}^{(\delta)}+\psi^{(\delta)}\bigr).
  \]
  By Lemma~\ref{lem:baseline-contraction}, $\gamma^{(\delta)}<1$.
\end{proof}

\subsubsection{Case 2}
For this case, where $\tilde a:=\mu^2/\sigma^2>0$, we require Assumption~\ref{ass:null-smallness-A}. To establish the drift condition, we need several technical lemmas. The first lemma provides a computational formula for the projection of the all-ones vector onto the null space of $B$, which is crucial for bounding the component of $\bar m$ in the null space.

\begin{lemma}
  \label{lem:nullproj-1-via-A}
  Assume $B=AK^{-1}$ with $K\in\mathbb R^{n\times n}$ invertible.
  Let $U_A\in\mathbb R^{n\times r}$ be an orthonormal basis for $\Null(A)$,
  where $r=\dim\Null(A)$. Define
    $G:=U_A^\top K^\top K\,U_A\in\mathbb R^{r\times r}$,
    and
    $z:=U_A^\top K^\top \mathbf 1\in\mathbb R^r$.
  Let $P_{\Null(B)}$ denote the Euclidean orthogonal projector onto $\Null(B)$.
  Then
  \begin{equation}\label{eq:nullproj-identity}
    \|P_{\Null(B)}\mathbf 1\|_2^2
    =\mathbf 1^\top P_{\Null(B)}\mathbf 1
    =z^\top G^{-1}z.
  \end{equation}
  In particular, if $\bar m=\rho B^\top(Y-X\beta+\mu Bh)+\mu\mathbf 1$ and $\bar m_0:=P_{\Null(B)}\bar m$, then we have $\bar m_0=\mu P_{\Null(B)}\mathbf 1$, and consequently
  \begin{equation}\label{eq:m0-norm-identity}
    \|\bar m_0\|_2=|\mu|\,\|P_{\Null(B)}\mathbf 1\|_2
    =|\mu|\sqrt{z^\top G^{-1}z}.
  \end{equation}
\end{lemma}

The second lemma establishes uniform bounds on negative moments of the GIG distribution, which are needed to control the behavior of $\tilde V_i$ when the parameters vary.

\begin{lemma}
\label{lem:GIG-negmoment-unif}
  Assume $p<0$, $\tilde a>0$, and $b>0$. Fix any $q\in(0,\tfrac12]$.
  Then there exists a constant $L_-(p,\tilde a,b,q)<\infty$ such that for all $M_i\in\mathbb R$,
    $\mathbb E\!\left[\tilde V_i^{-q}\mid M_i\right]\le L_-(p,\tilde a,b,q)$.
  Consequently, for all $V\in(0,\infty)^n$,
  \[
    \sum_{i=1}^n \mathbb E\!\left[\tilde V_i^{-q}\mid V\right]\le n\,L_-(p,\tilde a,b,q).
  \]
\end{lemma}

Finally, the third lemma provides a linear bound on the first moment of the GIG distribution that holds uniformly for all values of the shape parameter $p$. This is essential for bounding terms involving $\mathbb E[\tilde V_i\mid M_i]$ in the drift analysis.

\begin{lemma}
\label{lem:EV-linear-allp}
  Let $\tilde V_i\mid M_i\sim \mathrm{GIG}(p-\tfrac12,\tilde a,\ b+(M_i/\sigma)^2)$ with $\tilde a>0$ and $b\ge 0$.
  Then there exists a constant $C_{\mathrm{lin}}(p,\tilde a)<\infty$ such that for all $M_i\in\mathbb R$,
  \[
    \mathbb E[\tilde V_i\mid M_i]
    \le
    \frac{1}{\sqrt{\tilde a}}\Bigl|\frac{M_i}{\sigma}\Bigr|
    +\frac{\sqrt b}{\sqrt{\tilde a}}
    +\frac{C_{\mathrm{lin}}(p,\tilde a)}{\tilde a}.
  \]
\end{lemma}

The proofs of Lemmas~\ref{lem:nullproj-1-via-A}, \ref{lem:GIG-negmoment-unif}, and~\ref{lem:EV-linear-allp} are given in Appendix~\ref{app:proofs}. With these auxiliary lemmas in hand, we are now ready to establish the main drift condition for Case~2.

\begin{lemma}[Drift condition for $a=0, \mu\neq 0$]
  \label{lem:drift-a0-mune0}
  Assume $a=0$, $\mu\neq 0$, $b>0$, and $p<0$. Suppose Assumption~\ref{ass:null-smallness-A} holds.
  Fix any $q\in(0,\tfrac12]$ and define the Lyapunov function
  \[
    G_q(V):=1+\sum_{i=1}^n\Big(V_i+V_i^{-q}\Big).
  \]
  Then there exist $\gamma\in(0,1)$ and $L<\infty$ such that for all $V\in(0,\infty)^n$,
  \[
    \mathbb E\!\left[G_q(\tilde V)\mid V\right]\ \le\ \gamma\,G_q(V)+L.
  \]
  More precisely, for any $\varepsilon>0$ satisfying
    $\gamma_+ := \sqrt{z^\top G^{-1} z} + \varepsilon < 1$,
  one may take $\gamma=\gamma_+$.
\end{lemma}

\begin{proof}
  Let $\eta(V):=\bar Q(V)^{-1}\bar m$ and recall $\bar Q(V)=\rho B^\top B+D_V^{-1}$.
  By linearity of conditional expectation,
  \[
    \mathbb E\!\left[G_q(\tilde V)\mid V\right]
    = 1 + \sum_{i=1}^n\mathbb E[\tilde V_i\mid V] + \sum_{i=1}^n\mathbb E[\tilde V_i^{-q}\mid V].
  \]
  By Lemma~\ref{lem:GIG-negmoment-unif}, the negative moments are uniformly bounded:
  \[
    \sum_{i=1}^n \mathbb E[\tilde V_i^{-q}\mid V]\le n\,L_-(p,\tilde a,b,q)=:L_-.
  \]
  By Lemma~\ref{lem:EV-linear-allp},
  \[
    \sum_{i=1}^n\mathbb E[\tilde V_i\mid V]
    \le
    \frac{1}{\sqrt{\tilde a}}\,
    \mathbb E\Bigl[\Bigl\|\frac{M}{\sigma}\Bigr\|_1\Bigm|V\Bigr]
    +n\Bigl(\frac{\sqrt b}{\sqrt{\tilde a}} + \frac{C_{\mathrm{lin}}(p,\tilde a)}{\tilde a}\Bigr).
  \]
  Recall that $M/\sigma\mid V\sim N(\eta(V)/\sigma,\ \bar Q(V)^{-1})$. By using this conditional distribution together with the triangle inequality, we obtain
  \[
    \mathbb E\left[\Bigl\|\frac{M}{\sigma}\Bigr\|_1\Bigm|V\right]
    \le
    \Bigl\|\frac{\eta(V)}{\sigma}\Bigr\|_1+\mathbb E\|Z\|_1,
    \qquad Z\sim N(0,\bar Q(V)^{-1}).
  \]
  Now, decompose $\bar m=\bar m_0+\bar m_\parallel$ with $\bar m_0:=P_{\Null(B)}\bar m$ and $\bar m_\parallel\in\Range(B^\top)$.
  Thus $\eta(V)=\bar Q(V)^{-1}\bar m_0+\bar Q(V)^{-1}\bar m_\parallel$.
  For the null component, apply Lemma~\ref{lem:l1-duality-nosqrtn} and $\bar Q(V)^{-1}\preceq D_V$ to obtain
  \[
    \|\bar Q(V)^{-1}\bar m_0\|_1
    \le \sqrt{\sum_{i=1}^n V_i}\,\sqrt{\bar m_0^\top \bar Q(V)^{-1}\bar m_0}
    \le \sqrt{\sum_{i=1}^n V_i}\,\sqrt{\bar m_0^\top D_V \bar m_0}
    \le \|\bar m_0\|_2\sum_{i=1}^n V_i.
  \]
  We now apply Lemma~\ref{lem:nullproj-1-via-A}. Since  $\bar m=\rho B^\top(Y-X\beta+\mu Bh)+\mu\mathbf 1$, its projection onto the null space is purely determined by $\mu$:
    $\bar m_0 = \mu P_{\Null(B)}\mathbf 1$.
  By Lemma~\ref{lem:nullproj-1-via-A} (Eq.~\eqref{eq:m0-norm-identity}), we have
    $\|\bar m_0\|_2 = |\mu| \sqrt{z^\top G^{-1} z}$.
  Substituting this into the bound for the null component, the coefficient of $\sum V_i$ becomes:
  \[
    \frac{1}{\sigma\sqrt{\tilde a}}\|\bar m_0\|_2
    = \frac{1}{\sigma (\mu/\sigma)} |\mu| \sqrt{z^\top G^{-1} z}
    = \sqrt{z^\top G^{-1} z}.
  \]
  For the range component, as established in Lemma~\ref{lem:drift-nullsmall} (Step 2(A)), there exists $c_\parallel<\infty$ such that
  \[
    \frac{1}{\sqrt{\tilde a}}\Bigl\|\frac{\bar Q(V)^{-1}\bar m_\parallel}{\sigma}\Bigr\|_1
    \le c_\parallel \sqrt{\sum V_i}
    \le \frac{\varepsilon}{2}\sum V_i+\frac{c_\parallel^2}{2\varepsilon}.
  \]
  Similarly, the fluctuation term satisfies (with $c_0 := \sqrt{2/\pi}/\sqrt{\tilde a}$):
  \[
    \frac{1}{\sqrt{\tilde a}}\mathbb E\|Z\|_1
    \le c_0 \sum_{i=1}^n \sqrt{V_i}
    \le \frac{\varepsilon}{2}\sum_{i=1}^n V_i + \frac{n c_0^2}{2\varepsilon}.
  \]
  Combining these bounds, there exists $L_+<\infty$ such that
  \[
    \sum_{i=1}^n\mathbb E[\tilde V_i\mid V]
    \le
    \Bigl(\sqrt{z^\top G^{-1} z}+\varepsilon\Bigr)\sum_{i=1}^n V_i + L_+.
  \]
  By Assumption~\ref{ass:null-smallness-A} (in the case $a=0$), we have $\sqrt{z^\top G^{-1} z} < 1$. Thus, we can choose $\varepsilon>0$ sufficiently small so that
    $\gamma_+:=\sqrt{z^\top G^{-1} z}+\varepsilon<1$.
  Let $\gamma:=\gamma_+$ and $L:=1+L_-+L_+$.
  Then for all $V\in(0,\infty)^n$,
  \[
    \mathbb E\!\left[G_q(\tilde V)\mid V\right]
    \le
    \gamma \sum_{i=1}^n V_i + L
    \le
    \gamma\,G_q(V)+L,
  \]
  which proves the drift condition.
\end{proof}

\subsubsection{Case 3: \(a>0, b=0, p>0\)}

For this case, the proof requires several technical lemmas. The first lemma provides a lower bound on diagonal elements of matrix inverses and an upper bound for negative moments of Gaussian random variables.

\begin{lemma}\label{lem:two-bounds-nosqrtn}
  Let $A\succ0$ be symmetric. For each $i$,
    $(A^{-1})_{ii}\ \ge\ \nicefrac{1}{A_{ii}}$.
  Moreover, for any $0<\delta<1$ and $Z\sim N(\eta,s^2)$,
  \[
    \mathbb E|Z|^{-\delta}\ \le\ \kappa(\delta)\,s^{-\delta},
    \qquad
    \kappa(\delta):=\mathbb E|N(0,1)|^{-\delta}
    =\frac{1}{\sqrt{\pi}}\,2^{-\delta/2}\Gamma\!\Big(\frac{1-\delta}{2}\Big).
  \]
\end{lemma}

\begin{proof}
  The first inequality follows from the Schur complement formula or, equivalently, from the variational representation of the diagonal elements.
  For the Gaussian bound, note that $x \mapsto |x|^{-\delta}$ is symmetric and decreasing in $|x|$. By Anderson's inequality,
    $\mathbb E|N(\eta,s^2)|^{-\delta}\le \mathbb E|N(0,s^2)|^{-\delta}=s^{-\delta}\kappa(\delta)$.
\end{proof}

The second lemma establishes a key bound for negative fractional moments of the GIG distribution when $p>0$. The bound involves a specific exponent $\delta(p)$ chosen to ensure that the drift coefficient is contractive.

\begin{lemma}\label{lem:E1V-delta-app}
  Given $\sigma>0$ and
    $V_i \mid M_i \sim \mathrm{GIG}\!\left(p-\nicefrac12,\ \tilde a,\ \nicefrac{M_i^2}{\sigma^2}\right)$,
  where $\tilde a = a+\mu^2/\sigma^2 > 0$ and $p>0$. Define
  \[
    \delta(p) :=
    \begin{cases}
      p                      & \text{if } 0 < p \le \tfrac12, \\
      \min\{\tfrac12, 2p-1\} & \text{if } p > \tfrac12.
    \end{cases}
  \]
  Then $0<\delta(p)\le \frac12$. There exist constants
  $C_1(p)>0$ (depending only on $p$) and $C_2(p,\tilde a)\in\mathbb R$ such that
  \[
    \mathbb E\!\left[V_i^{-\frac{\delta(p)}{2}}\mid M_i\right]
    \le C_1(p)\,\Bigl|\frac{M_i}{\sigma}\Bigr|^{-\delta(p)}+C_2(p,\tilde a).
  \]
  Moreover, with $\kappa(\cdot)$ as in Lemma~\ref{lem:two-bounds-nosqrtn}, we have
    $C_1(p)\,\kappa\bigl(\delta(p)\bigr) < 1$ for all $p>0$.
\end{lemma}

The third lemma provides an $\ell_1$-duality bound for controlling the norm of $\bar Q(V)^{-1}u$.

\begin{lemma}
\label{lem:l1-duality-nosqrtn}
  Given $\bar{Q}(V) = \rho B^\top B + D_{V}^{-1}$, let $u\in\mathbb{R}^n$ and $V\in(0,\infty)^n$. Then
  \begin{equation}\label{eq:l1-duality-main}
    \|\bar{Q}(V)^{-1}u\|_1
    \le
    \sqrt{\sum_{i=1}^n V_i}\,\sqrt{u^\top \bar{Q}(V)^{-1}u}.
  \end{equation}
\end{lemma}

The proofs of Lemmas~\ref{lem:E1V-delta-app} and~\ref{lem:l1-duality-nosqrtn} are given in Appendix~\ref{app:proofs}. With these technical lemmas established, we now present the main drift condition for Case~3.

\begin{lemma}[Drift condition for $a>0, b=0$]
  \label{lem:drift-nullsmall}
  Assume $a>0$, $b=0$, $p>0$, and Assumption~\ref{ass:null-smallness-A}.
  Let $\delta=\delta(p)\in(0,\tfrac12]$ be as in Lemma~\ref{lem:E1V-delta-app}, and define the Lyapunov function
  \[
    G_\delta(V):=1+\sum_{i=1}^n\Big(V_i+V_i^{-\delta/2}\Big).
  \]
  Then there exist $\gamma\in(0,1)$ and $L<\infty$ (possibly depending on the dimension $n$) such that for all $V\in(0,\infty)^n$,
    $\mathbb E\!\left[G_\delta(\tilde V)\mid V\right]\ \le\ \gamma\,G_\delta(V)+L$.
  Explicitly, for any $\varepsilon>0$ satisfying
  \[
    \gamma_+:=\frac{|\mu|}{\sqrt{\sigma^2 a+\mu^2}}\sqrt{\,z^\top G^{-1}z\,}+\varepsilon<1,
  \]
  one may take $\gamma=\max\{\gamma_-,\ \gamma_+\} < 1$, where $\gamma_-:=C_1(p)\kappa(\delta(p))$.
\end{lemma}

\begin{proof}
  Let $\mathcal S:=\Range(B^\top)\subseteq\mathbb R^n$ and let $U\in\mathbb R^{n\times d}$ be any orthonormal basis matrix of $\mathcal S$ (where $d=\dim\mathcal S$).
  Define
    $\lambda_+ := \lambda_{\min}\!\big(U^\top B^\top B\,U\big) \;>\; 0$,
  that is, the smallest strictly positive eigenvalue of $B^\top B$ (equivalently, the smallest eigenvalue of $B^\top B$ restricted to $\Range(B^\top)$).
  Let $\eta(V):=\bar Q(V)^{-1}\bar m$.
  By linearity of conditional expectation,
  \[
    \mathbb E\!\left[G_\delta(\tilde V)\mid V\right]
    = 1 + \sum_{i=1}^n\mathbb E[\tilde V_i\mid V] + \sum_{i=1}^n\mathbb E[\tilde V_i^{-\delta/2}\mid V].
  \]

  We first bound the negative-moment term. By Lemma~\ref{lem:E1V-delta-app}, for each $i$,
  \[
    \mathbb E\!\left[\tilde V_i^{-\delta/2}\mid M_i\right]
    \le C_1(p)\Bigl|\frac{M_i}{\sigma}\Bigr|^{-\delta} + C_2(p,\tilde a).
  \]
  Conditioning on $V$, we have $Z_i:=M_i/\sigma\sim N(\eta_i(V)/\sigma,\ (\bar Q(V)^{-1})_{ii})$.
  By Lemma~\ref{lem:two-bounds-nosqrtn},
  \[
    \mathbb E\!\left[\bigl|Z_i\bigr|^{-\delta}\mid V\right]
    \le \kappa(\delta)\,(\bar Q(V)^{-1})_{ii}^{-\delta/2}.
  \]
  Because $\bar Q(V)_{ii} = \rho(B^\top B)_{ii} + V_i^{-1}$ and $(\bar Q(V)^{-1})_{ii} \ge 1/\bar Q(V)_{ii}$, we have the inequality 
    $(\bar Q(V)^{-1})_{ii}^{-\delta/2} \le \bar Q(V)_{ii}^{\delta/2}$.
  Using $(x+y)^q \le x^q + y^q$ for $q=\delta/2 \in (0, 1/4]$, we obtain
  \[
    \bar Q(V)_{ii}^{\delta/2}
    \le \bigl(\rho(B^\top B)_{ii}\bigr)^{\delta/2} + V_i^{-\delta/2}.
  \]
  Combining these estimates gives
    $\mathbb E[\tilde V_i^{-\delta/2}\mid V]
    \le C_1(p)\kappa(\delta)\,V_i^{-\delta/2} + L_{-,i}$,
  where $L_{-,i}$ denotes the sum of all terms independent of $V$. Summing over $i$ and letting $L_- := \sum_{i=1}^n L_{-,i}$, we obtain
  \[
    \sum_{i=1}^n\mathbb E[\tilde V_i^{-\delta/2}\mid V]
    \le \gamma_- \sum_{i=1}^n V_i^{-\delta/2} + L_-,
    \qquad
    \gamma_- := C_1(p)\kappa(\delta).
  \]
  By Lemma~\ref{lem:E1V-delta-app}, $\gamma_-<1$.
  Next, we bound the linear term. By Lemma~\ref{lem:EV-linear-allp},
  \[
    \sum_{i=1}^n\mathbb E[\tilde V_i\mid V]
    \le \frac{1}{\sqrt{\tilde a}}\,
    \mathbb E\Bigl[\Bigl\|\frac{M}{\sigma}\Bigr\|_1\Bigm|V\Bigr]
    + n\Bigl(\frac{\sqrt b}{\sqrt{\tilde a}} + \frac{C_{\mathrm{lin}}(p,\tilde a)}{\tilde a}\Bigr).
  \]
  Since $b=0$ in this case and $M/\sigma \mid V \sim N(\eta(V)/\sigma,\ \bar Q(V)^{-1})$, we can combine this conditional distribution with the triangle inequality to obtain
  \[
    \mathbb E\Bigl\|\frac{M}{\sigma}\Bigr\|_1\Bigm|V
    \le
    \Bigl\|\frac{\eta(V)}{\sigma}\Bigr\|_1
    + \mathbb E\|Z\|_1,
    \qquad \text{where } Z\sim N(0,\bar Q(V)^{-1}).
  \]

  We now decompose $\bar m = \bar m_0 + \bar m_\parallel$, where $\bar m_0 := P_{\Null(B)}\bar m$ and $\bar m_\parallel \in \Range(B^\top)$, so that
    $\eta(V)=\bar Q(V)^{-1}\bar m_0 + \bar Q(V)^{-1}\bar m_\parallel$.
  By Lemma~\ref{lem:l1-duality-nosqrtn} and the bound $\bar Q(V)^{-1} \preceq D_V$,
  \[
    \Bigl\|\bar Q(V)^{-1}\bar m_0\Bigr\|_1
    \le \sqrt{\sum_{i=1}^n V_i}\,\sqrt{\bar m_0^\top D_V \bar m_0}
    \le \|\bar m_0\|_2 \sum_{i=1}^n V_i.
  \]
  Lemma~\ref{lem:nullproj-1-via-A} gives $\|\bar m_0\|_2 = |\mu|\sqrt{z^\top G^{-1}z}$, hence
  \[
    \frac{1}{\sigma\sqrt{\tilde a}}\Bigl\|\bar Q(V)^{-1}\bar m_0\Bigr\|_1
    \le
    \frac{|\mu|}{\sqrt{\sigma^2 a+\mu^2}}\sqrt{z^\top G^{-1}z}\;\sum_{i=1}^n V_i.
  \]

  For the component in $\Range(B^\top)$, Lemma~\ref{lem:l1-duality-nosqrtn} also gives
  \[
    \Bigl\|\bar Q(V)^{-1}\bar m_\parallel\Bigr\|_1
    \le \sqrt{\sum_{i=1}^n V_i}\,\sqrt{\bar m_\parallel^\top \bar Q(V)^{-1}\bar m_\parallel}.
  \]
  Since $\bar Q(V)\succeq \rho B^\top B$, the restriction of $\bar Q(V)^{-1}$ to $\mathcal S=\Range(B^\top)$ satisfies $\bar Q(V)^{-1}\preceq \frac{1}{\rho}(B^\top B)^\dagger$, and therefore
  \[
    \bar m_\parallel^\top \bar Q(V)^{-1}\bar m_\parallel
    \le \frac{1}{\rho\lambda_+}\|\bar m_\parallel\|_2^2.
  \]
  Let $S:=\sum_{i=1}^n V_i$ and $c_\parallel := \frac{\|\bar m_\parallel\|_2}{\sigma\sqrt{\tilde a}\sqrt{\rho\lambda_+}}$. Then
  \[
    \frac{1}{\sqrt{\tilde a}}\Bigl\|\frac{\bar Q(V)^{-1}\bar m_\parallel}{\sigma}\Bigr\|_1
    \le c_\parallel \sqrt{S}.
  \]
  For any $\varepsilon>0$, we have $c_\parallel\sqrt{S} \le \frac{\varepsilon}{2}S + \frac{c_\parallel^2}{2\varepsilon}$.

  Next, we bound the Gaussian fluctuation term. Since $(\bar Q(V)^{-1})_{ii} \le V_i$, we have
  \[
    \mathbb E|Z_i| = \sqrt{\frac{2}{\pi}}\sqrt{(\bar Q(V)^{-1})_{ii}} \le \sqrt{\frac{2}{\pi}}\sqrt{V_i}.
  \]
  Thus
  \[
    \frac{1}{\sqrt{\tilde a}}\mathbb E\|Z\|_1
    \le c_0 \sum_{i=1}^n \sqrt{V_i},
    \qquad c_0 := \sqrt{\frac{2}{\pi}}\frac{1}{\sqrt{\tilde a}}.
  \]
  Using $c_0\sqrt{V_i} \le \frac{\varepsilon}{2}V_i + \frac{c_0^2}{2\varepsilon}$ and summing over $i$, we obtain
  \[
    \frac{1}{\sqrt{\tilde a}}\mathbb E\|Z\|_1
    \le \frac{\varepsilon}{2}\sum_{i=1}^n V_i + \frac{n c_0^2}{2\varepsilon}.
  \]

  Combining the preceding bounds, there exists $L_+<\infty$ such that
  \[
    \sum_{i=1}^n\mathbb E[\tilde V_i\mid V]
    \le
    \Biggl(\frac{|\mu|}{\sqrt{\sigma^2 a+\mu^2}}\sqrt{z^\top G^{-1}z}+\varepsilon\Biggr)\sum_{i=1}^n V_i + L_+.
  \]
  By Assumption~\ref{ass:null-smallness-A}, we may choose $\varepsilon>0$ small enough so that
  \[
    \gamma_+:=\frac{|\mu|}{\sqrt{\sigma^2 a+\mu^2}}\sqrt{z^\top G^{-1}z}+\varepsilon<1.
  \]

  Finally, set $\gamma := \max\{\gamma_-,\gamma_+\}<1$ and $L:=1+L_-+L_+$. Then for all $V\in(0,\infty)^n$,
    $\mathbb E\!\left[G_\delta(\tilde V)\mid V\right] \le \gamma\,G_\delta(V) + L$,
  completing the proof of the drift inequality.
\end{proof}

Having established the drift condition, we now turn to proving the minorization condition, which completes the verification of the conditions for geometric ergodicity.

\begin{lemma}[Minorization on $\{G(V)\le d\}$]\label{lem:minorization}
  Fix $\alpha,\beta>0$ and define
  \[
    G(V):=\sum_{i=1}^n\bigl(V_i^\alpha+V_i^{-\beta}\bigr),
    \qquad
    C_d:=\{V\in(0,\infty)^n:\ G(V)\le d\}.
  \]
  Assume that for each $V\in(0,\infty)^n$, the conditional density $\pi(M\mid V,Y)$ exists,
  is strictly positive for all $M\in\mathbb R^n$, and the map
  $V\mapsto \pi(M\mid V,Y)$ is continuous on $(0,\infty)^n$ for each fixed $M$.
  (For the Gaussian update $M\mid(V,Y)\sim N(Q(V)^{-1}m,Q(V)^{-1})$, these conditions hold).
  Let
  \[
    k(V,\tilde{V}):=\int_{\mathbb R^n}\pi(\tilde{V}\mid M)\,\pi(M\mid V,Y)\,dM,
    \qquad V\in(0,\infty)^n,\ \tilde{V}\in(0,\infty)^n.
  \]
  Then, for every $d>0$, there exist $\epsilon_d\in(0,1]$ and a density $f_d(\tilde{V})$ on $(0,\infty)^n$
  such that
  \[
    k(V,\tilde{V})\ \ge\ \epsilon_d\, f_d(\tilde{V}),
    \qquad \forall\,V\in C_d,\ \forall\,\tilde{V}\in(0,\infty)^n.
  \]
\end{lemma}

\begin{proof}
  First, observe that $C_d$ is compact. Indeed, for any $V \in C_d$, we have
  $V_i^\alpha \le G(V) \le d$ and $V_i^{-\beta} \le d$ for each $i$, implying that
  $d^{-1/\beta} \le V_i \le d^{1/\alpha}$. Thus,
  $C_d$ is a closed and bounded subset of $[d^{-1/\beta}, \, d^{1/\alpha}]^n \subset (0, \infty)^n$.

  Fix $M \in \mathbb R^n$. Since the map $V \mapsto \pi(M \mid V,Y)$ is continuous on $(0, \infty)^n$,
  its restriction to the compact set $C_d$ attains its minimum by the Weierstrass theorem. Define
  \[
    g_d(M) := \min_{V \in C_d} \pi(M \mid V,Y).
  \]
  Since $\pi(M \mid V,Y) > 0$ for all $V \in (0, \infty)^n$, it follows that $g_d(M) > 0$ for all $M$.
  Moreover, choosing an arbitrary $V_0 \in C_d$, we have $0 < g_d(M) \le \pi(M \mid V_0,Y)$ for all $M$,
  showing that $g_d$ is integrable and satisfies
  \[
    0 < \epsilon_d := \int_{\mathbb R^n} g_d(M) \, dM \le \int_{\mathbb R^n} \pi(M \mid V_0,Y) \, dM = 1.
  \]
  Set $\tilde g_d(M) := g_d(M) / \epsilon_d$, which is a probability density on $\mathbb R^n$ and define
  \[
    f_d(\tilde V) := \int_{\mathbb R^n} \pi(\tilde V \mid M) \, \tilde g_d(M) \, dM.
  \]
  By Tonelli's theorem, and because $\int_{(0, \infty)^n} \pi(\tilde V \mid M) \, d\tilde V = 1$, we obtain
  \[
    \int_{(0, \infty)^n} f_d(\tilde V) \, d\tilde V
    = \int_{\mathbb R^n} \tilde g_d(M) \, dM = 1,
  \]
  demonstrating that $f_d$ is a probability density on $(0, \infty)^n$.

  Finally, for any $V \in C_d$, we have $\pi(M \mid V,Y) \ge g_d(M)$ for all $M$; hence,
  \[
    k(V, \tilde V)
    = \int_{\mathbb R^n} \pi(\tilde V \mid M) \pi(M \mid V,Y) \, dM
    \ge \int_{\mathbb R^n} \pi(\tilde V \mid M) \, g_d(M) \, dM
    = \epsilon_d \, f_d(\tilde V),
  \]
  which is the desired minorization.
\end{proof}

With these preliminary results established, we are now ready to prove Proposition~\ref{prop:geo-dm-cases}.

\begin{proof}[Proof of Proposition~\ref{prop:geo-dm-cases}]
  Recall that Proposition~\ref{prop:geo-dm-cases} establishes geometric ergodicity in three regimes based on the drift and minorization conditions.
  \begin{description}
    \item[Case \textup{(I)}] follows from Lemmas~\ref{lem:drift-a0mu0} and \ref{lem:minorization} via Theorem~\ref{thm:drift-minorization}.

    \item[Case \textup{(II)}] follows from Lemmas~\ref{lem:drift-a0-mune0} and \ref{lem:minorization} under Assumption~\ref{ass:null-smallness-A}, again via Theorem~\ref{thm:drift-minorization}.

    \item[Case \textup{(III)}] follows from Lemmas~\ref{lem:drift-nullsmall} and \ref{lem:minorization} under Assumption~\ref{ass:null-smallness-A}, via Theorem~\ref{thm:drift-minorization}.
  \end{description}
\end{proof}

\section{Simulation study}\label{chp:simulation-study}

\subsection{Overview and common setup}\label{sec:sim-common}

This section reports two simulation studies (S1 and S2) that numerically illustrate and stress-test the theoretical mixing results established in Section~\ref{chp:proof}. Both studies focus on the $V$-marginal Gibbs chain induced by the non-centered parametrization, since the key theoretical mechanisms (trace-class properties and drift--minorization) are formulated for the marginal transition operator on $V$. Mixing behavior is compared using standard MCMC efficiency diagnostics.

Simulation S1 examines how mixing varies across the parameter regimes in Table~\ref{tab:geom-ergodicity} by varying the GIG parameters $(p,a,b,\mu)$. The goal is to contrast trace-class settings (Theorem~\ref{thm:trace-class}) against boundary/heavy-tail regimes where geometric ergodicity is established via drift--minorization (Proposition~\ref{prop:geo-dm-cases}). Simulation S2 investigates the practical impact of the null-smallness mechanism by fixing a boundary-sensitive distributional regime and keeping the geometry fixed (i.e., fix $A$ and hence $B=A K(\phi)^{-1}$), while scanning the drift parameter $\mu$ so that the null-smallness constant changes. Mixing is assessed using boundary-sensitive and null-sensitive summaries.

Both studies use the non-centered parametrization in \eqref{eq:non-centered-combined}. Let $V\in(0,\infty)^n$ denote the mixing variables and $M\in\mathbb R^n$ the latent field. Given fixed matrices $A$ and $K(\phi)$ (with $K(\phi)$ invertible), define $B:=A K(\phi)^{-1}$. The model is
\begin{align*}
  Y \mid M & \sim N\!\bigl( X\beta + B(M-\mu h),\ \sigma_\epsilon^2 I \bigr),      \\
  M \mid V & \sim N\!\ bigl(\mu V,\ \mathrm{diag}(\sigma^2 V)\bigr),                \\
  V_i      & \overset{\mathrm{ind.}}{\sim} \mathrm{GIG}(p,a,b),\qquad i=1,\dots,n,
\end{align*}
where we use the GIG parametrization with density $f(v;p,a,b)\ \propto\ v^{p-1}\exp\!\bigl\{-\frac12(av+b/v)\bigr\}$ for $v>0$. Under this parametrization, the non-centered Gibbs sampler alternates between the Gaussian update
\[
  M \mid V, Y \sim N\!\bigl(\bar Q(V)^{-1}\bar m,\ \sigma^2 \bar Q(V)^{-1}\bigr),
  \qquad \bar Q(V):=\rho\,B^\top B + D_V^{-1},\quad \rho:=\sigma^2\sigma_\epsilon^{-2},
\]
and independent coordinate-wise GIG updates
\[
  \tilde V_i \mid M_i \sim \mathrm{GIG}\!\left(p-\tfrac12,\ a+\left(\tfrac{\mu}{\sigma}\right)^2,\
  b+\left(\tfrac{M_i}{\sigma}\right)^2\right),
  \qquad i=1,\dots,n,
\]
where $D_V=\mathrm{diag}(V_1,\dots,V_n)$ and $\bar m:=\rho\,B^\top\!\bigl(Y-X\beta+\mu Bh\bigr)+\mu\mathbf 1$.

Unless otherwise stated, we fix the problem size to $n=m=300$ and set $h=\mathbf 1$, $X\beta=\mathbf 0$, $\sigma=1$, $\sigma_\epsilon=1$, and $Y=\mathbf 0$. The choice $Y=\mathbf 0$ is a fixed benchmark observation (not prior sampling), introduced to remove randomness from data realizations while preserving the dependence structure through $\rho B^\top B$. For $K(\phi)$ we use an AR(1)-type correlation matrix with parameter $\phi=0.5$ and compute $K(\phi)^{-1}$ numerically. The design matrix $A$ is specified within each study (S1: $A=I_n$; S2: $A=I_n-uu^\top$).

For each chain and each monitored scalar statistic, we compute the integrated autocorrelation time (IACT) and the effective sample size per second (ESS/sec), where wall-clock time is measured per chain. For a retained length $N$, $\mathrm{ESS}=N/\mathrm{IACT}$. To assess agreement across chains we compute split-$\widehat R$ and report maxima over the monitored statistics, as specified in each study.


\subsection{Simulation study (S1): mixing comparison across parameter regimes}
\label{sec:sim-S1}

This simulation study numerically illustrates how the mixing behavior of the $V$-marginal Gibbs chain varies across the parameter regimes summarized in Table~\ref{tab:geom-ergodicity}. In particular, we contrast regimes where the marginal Markov operator $\Lambda$ is trace-class against bound ary/heavy-tail regimes where trace-classness is unavailable and geometric ergodicity is established via drift--minorization.

We adopt the common model, sampler, and benchmark platform described in Section~\ref{sec:sim-common}. In S1 we set $A=I_n$, so that $B=K(\phi)^{-1}$, and vary only the GIG parameters $(p,a,b,\mu)$ over the representative regime grid specified below. We evaluate six representative parameter points that cover all regimes in Table~\ref{tab:geom-ergodicity}. Each point specifies $(p,a,b,\mu)$, while keeping $(n,A,K(\phi),\sigma,\sigma_\epsilon,Y)$ fixed as in Section~\ref{sec:sim-common}:

\begin{table}[t]
  \centering
  \caption{Representative parameter points for Simulation S1.}
  \label{tab:sim-S1-grid}
  \begin{tabular}{ccccccl}
    \toprule
    Point & Regime & $p$    & $a$ & $b$ & $\mu$ & Key tag                                            \\
    \midrule
    A     & TC-1   & $-0.5$ & $1$ & $1$ & $1$   & IG mixing ($p=-\tfrac12$); NIG benchmark           \\
    B     & TC-1   & $1.0$  & $1$ & $1$ & $1$   & trace-class interior ($a>0,b>0$)                   \\
    C     & TC-2   & $0.6$  & $1$ & $0$ & $1$   & $b=0$ boundary; trace-class ($p>\tfrac12$)         \\
    D     & DM-III & $0.3$  & $1$ & $0$ & $1$   & $b=0$ boundary; non-trace-class ($0<p\le\tfrac12$) \\
    E     & DM-I   & $-1.5$ & $0$ & $2$ & $0$   & Inv-Gamma full conditional ($a=\mu=0$)             \\
    F     & DM-II  & $-1.5$ & $0$ & $2$ & $1$   & drifted heavy-tail ($a=0,p<0,\mu\neq0$)            \\
    \bottomrule
  \end{tabular}
\end{table}
Here, TC-1 and TC-2 fall into the trace-class regimes of Theorem~\ref{thm:trace-class}, while DM-I--DM-III correspond to boundary/heavy-tail regimes covered by the drift--minorization argument in Proposition~\ref{prop:geo-dm-cases}.

For each parameter point (A--F), we run $4$ overdispersed chains with $T=50{,}000$ iterations, $T_{\mathrm{burn}}=5{,}000$, and thinning${}=1$. The four initial states are chosen as $V^{(0)}= \mathbf 1$, $V^{(0)}= 0.1\,\mathbf 1$, $V^{(0)}= 10\,\mathbf 1$, and $V^{(0)}\sim \mathrm{GIG}(1,1,1)$. No thinning is applied; all reported efficiency measures are based on the retained post-burn-in draws. To connect empirical diagnostics to the drift functions used in the theoretical analysis, we monitor three scalar summaries of the $V$-chain:
\[
  S_+(t):=\frac{1}{n}\sum_{i=1}^n V_i^{(t)},\qquad
  S_-(t):=\frac{1}{n}\sum_{i=1}^n \bigl(V_i^{(t)}\bigr)^{-q},\qquad
  S_{\log}(t):=\frac{1}{n}\sum_{i=1}^n \log V_i^{(t)},
\]
with $q=0.25$. The statistic $S_+$ is sensitive to right-tail excursions, while $S_-$ and $S_{\log}$ are sensitive to boundary behavior near $V\downarrow 0$, which is particularly relevant in the $b=0$ regimes. For each statistic we report IACT, ESS/sec, and split-$\widehat R$, and summarize agreement by $\max\{\widehat R(S_+),\widehat R(S_-),\widehat R(S_{\log})\}$. Based on our theoretical results, we expect: (i) trace-class regimes (TC-1, TC-2) to typically exhibit higher sampling efficiency, and (ii) boundary regimes with $b=0$ and $0<p\le 1/2$ (DM-III) to show the most pronounced slowdown, particularly for $S_-$ and $S_{\log}$.

\begin{table}[t]
  \centering
  \small
  \caption{S1 simulation summary (IACT and ESS/sec).}
  \label{tab:sim-S1-summary}
  \begin{tabular}{llrrrrrr}
    \toprule
    \multirow{2}{*}{Regime} & \multirow{2}{*}{Point} & \multicolumn{2}{c}{$S_+$} & \multicolumn{2}{c}{$S_-$} & \multicolumn{2}{c}{$S_{\log}$}                            \\
    \cmidrule(lr){3-4} \cmidrule(lr){5-6} \cmidrule(lr){7-8}
                            &                        & IACT                      & ESS/sec                   & IACT                           & ESS/sec & IACT & ESS/sec \\
    \midrule
    TC-1                    & A                      & 1.31                      & 90.91                     & 1.33                           & 90.02   & 1.34 & 89.07   \\
    TC-1                    & B                      & 1.13                      & 105.69                    & 1.15                           & 104.03  & 1.15 & 103.78  \\
    TC-2                    & C                      & 1.48                      & 80.71                     & 1.80                           & 67.89   & 2.28 & 52.53   \\
    DM-III                  & D                      & 1.97                      & 60.79                     & 4.71                           & 29.49   & 6.75 & 17.74   \\
    DM-I                    & E                      & 1.02                      & 119.83                    & 1.17                           & 104.37  & 1.17 & 104.33  \\
    DM-II                   & F                      & 1.27                      & 94.15                     & 1.26                           & 94.64   & 1.27 & 93.70   \\
    \bottomrule
  \end{tabular}
\end{table}

Table~\ref{tab:sim-S1-summary} reports efficiency diagnostics for the six representative parameter points (A--F). Across all regimes and monitored summaries, split-$\widehat R$ values are extremely close to one (with $\max \widehat R \approx 1.0001$), indicating consistent behavior across chains for the scalar summaries considered.
A clear qualitative pattern is that the $b=0$ boundary regimes induce a pronounced slowdown, especially for summaries sensitive to $V\downarrow 0$. In the trace-class interior regime (TC-1, points A--B), the IACTs remain close to one and ESS/sec is high across $S_+,S_-,S_{\log}$. Moving to $b=0$ with $p>1/2$ (TC-2, point C), mixing degrades moderately, most noticeably for $S_{\log}$. The strongest degradation occurs at $b=0$ with $0<p\le 1/2$ (DM-III, point D), where $S_-$ and $S_{\log}$ exhibit substantially larger IACT and the lowest ESS/sec. Comparing DM-I and DM-II, introducing a nonzero drift parameter $\mu$ increases dependence in the chain: point F ($\mu\neq 0$) exhibits systematically larger IACTs and smaller ESS/sec than point E ($\mu=0$), although both remain well-behaved in this experiment.

These empirical findings align with the mechanisms underlying Table~\ref{tab:geom-ergodicity}. Trace-class regimes typically exhibit stronger $L_2(\pi)$ spectral behavior, which often translates into higher sampling efficiency. When $b=0$ and $p$ is small (DM-III), the target distribution exhibits more pronounced boundary behavior near $V\downarrow 0$; correspondingly, the chain mixes slowly for near-boundary functionals such as negative moments and $\log V$, consistent with Lyapunov functions involving negative powers of $V$ in the drift--minorization analysis.


\subsection{Simulation study (S2): null-smallness scan by varying $\mu$}
\label{sec:sim-S2}

This simulation study probes how the null-smallness constant affects mixing of the $V$-marginal Gibbs chain in a controlled setting. We keep the geometry fixed (i.e., fix $A$ and hence $B=A K(\phi)^{-1}$) and scan the drift parameter $\mu$. Varying $\mu$ changes both the combined quantity $\tilde a(\mu)=a+\bigl(\mu/\sigma\bigr)^2$ and the forcing term $\bar m(\mu)$ in the Gaussian update, yielding a one-dimensional stress test in which the null-smallness constant varies while the global coupling induced by $B^\top B$ is held fixed.
We adopt the common model, non-centered Gibbs sampler, and benchmark platform described in Section~\ref{sec:sim-common}. In S2 we fix the design matrix to have a one-dimensional null space by setting $A  :=  I_n - u u^\top$ with $u:=\mathbf 1/\|\mathbf 1\|_2$, so that $\Null(A)=\mathrm{span}\{u\}$ and $B := A K(\phi)^{-1}$ is fixed throughout the experiment (with $K(\phi)$ the AR(1)-type correlation matrix at $\phi=0.5$). We work under the benchmark $Y=\mathbf 0$ and $X\beta=\mathbf 0$, so that the dependence on $\mu$ enters through the drift term only.

We scan the drift parameter over a fixed grid $\mu \in \{\mu_1,\dots,\mu_L\}$ and for each $\mu$ compute a null-smallness constant associated with the fixed operator $B$. Let $ u_0$ be a unit vector spanning $\Null(B)$. For $B=A K(\phi)^{-1}$ with $A=I-uu^\top$, we have $\Null(B)=\mathrm{span}\{K(\phi)u\}$; hence we take $u_0  :=  K(\phi)u/\|K(\phi)u\|_2$. Recall $\tilde a(\mu)  :=  a+\bigl(\mu/\sigma\bigr)^2$, $\bar m(\mu)  :=  \rho\,B^\top\!\bigl(Y-X\beta+\mu Bh\bigr)+\mu\mathbf 1$, and $\rho:=\sigma^2\sigma_\epsilon^{-2}$. We define
\begin{equation}\label{eq:gamma-ns-mu}
  \gamma_{\mathrm{ns}}(\mu)
  \ :=\
  \frac{\|P_{\Null(B)}\bar m(\mu)\|_2}{\sigma\,\sqrt{\tilde a(\mu)}}
  \ =\
  \frac{\bigl|\langle u_0,\bar m(\mu)\rangle\bigr|}{\sigma\,\sqrt{\tilde a(\mu)}},
\end{equation}
where the second identity uses that $\Null(B)$ is one-dimensional. We fix a boundary-sensitive regime and vary only $\mu$. In our implementation we focus on the $b=0$ boundary regime (DM-III), taking $(p,a,b)=(0.5,\,1,\,0)$ and $h=\mathbf 1$. We monitor the three scalar summaries $S_+$, $S_-$, and $S_{\log}$ introduced in the S1 study, along with the null-direction statistic
\begin{equation}\label{eq:T-null-mu-scan}
  T_{\mathrm{null}}(t)
  \ :=\
  \bigl\langle K(\phi)^{-1}u_0,\ M^{(t)}-\mu h\bigr\rangle.
\end{equation}
For each statistic and each chain, we compute IACT together with split-$\widehat R$.

For each $\mu$ on the scan grid, we run $4$ overdispersed chains with $T=50{,}000$ iterations, $T_{\mathrm{burn}}=10{,}000$, and thinning${}=1$, initialized at $V^{(0)}=\mathbf 1$, $V^{(0)}=0.1\,\mathbf 1$, $V^{(0)}=10\,\mathbf 1$, and $V^{(0)}\sim \mathrm{GIG}(1,1,1)$. All efficiency measures are computed from the retained post-burn-in draws, with wall-clock time recorded per chain.

\begin{figure}[t]
  \centering
  \includegraphics[width=0.8\textwidth]{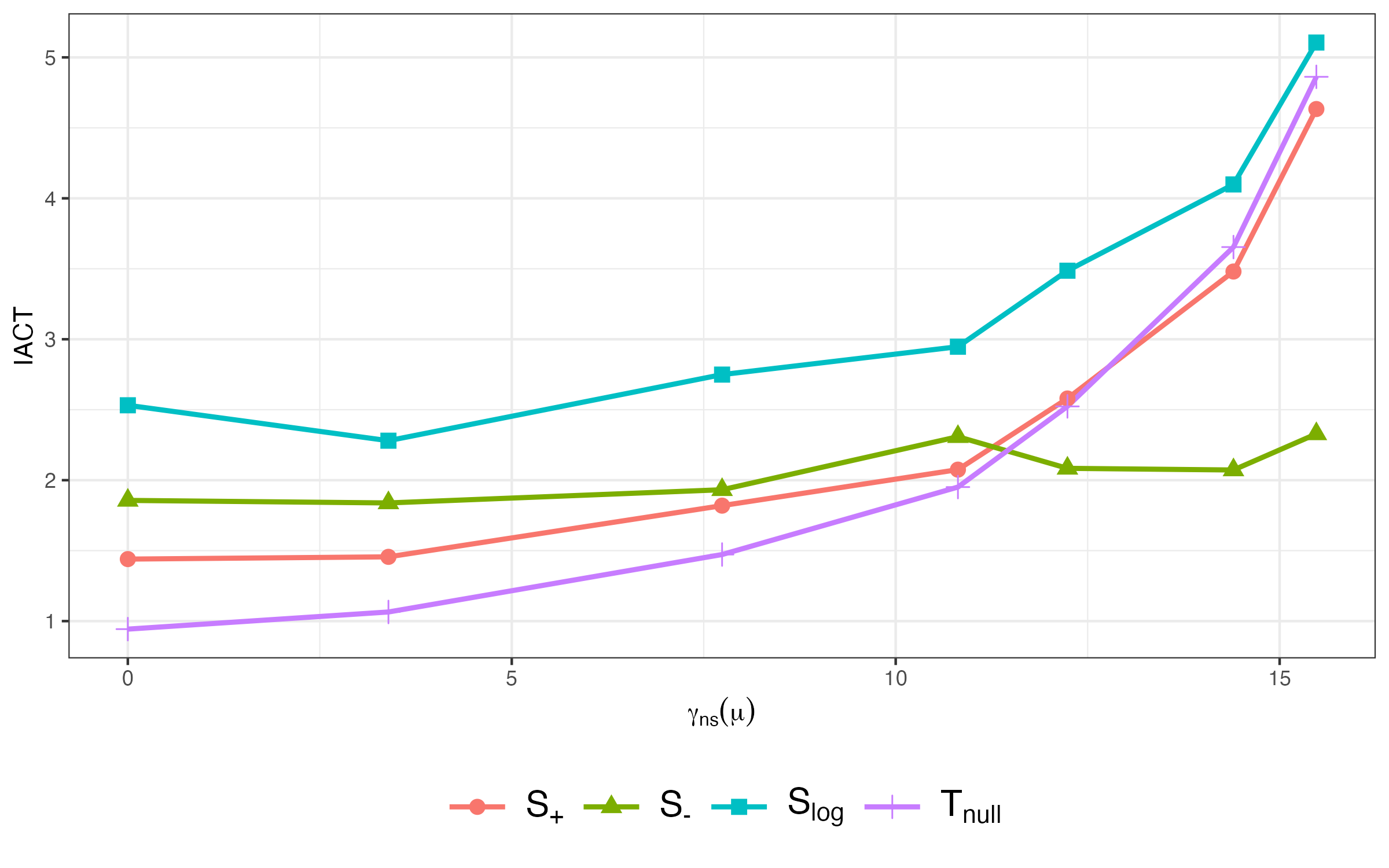}
  \caption{Simulation S2: mixing efficiency along the null-smallness scan induced by varying $\mu$ while keeping $A$ fixed. The figure reports IACT for four monitored summaries: $S_+$ (right-tail), $S_-$ and $S_{\log}$ (boundary-sensitive), and the null-direction statistic $T_{\mathrm{null}}$.}
  \label{fig:sim-S2-ess-iact}
\end{figure}

Across all scan points, split-$\widehat R$ values remain essentially one for all monitored summaries, indicating consistent behavior across the overdispersed chains. Increasing $\gamma_{\mathrm{ns}}(\mu)$ is accompanied by a systematic loss of sampling efficiency, visible as increasing IACT across multiple summaries, despite the fixed global coupling through $B^\top B$. In this scan, the null-direction statistic $T_{\mathrm{null}}(t)$ and the log-summary $S_{\log}(t)$ exhibit the strongest sensitivity to $\gamma_{\mathrm{ns}}(\mu)$, with markedly increasing IACT over the scan range. The negative-moment summary $S_-(t)$ typically shows a milder change, suggesting that the dominant slowdown mechanism here is associated with exploration along the null-sensitive direction encoded by \eqref{eq:T-null-mu-scan}, rather than purely with near-boundary excursions of $V$.

The observed deterioration is consistent with the drift-based mechanism underlying the null-smallness condition: scanning $\mu$ changes both $\bar m(\mu)$ and $\tilde a(\mu)$, thereby altering the normalized projection of $\bar m(\mu)$ onto $\Null(B)$ captured by $\gamma_{\mathrm{ns}}(\mu)$. These results provide numerical evidence that larger values of this constant can be associated with slower mixing for null-sensitive functionals, complementing the conditions established in Section~\ref{chp:proof}.

\section{Discussion}
\label{sec:discussion}

This work develops a coherent theoretical and computational foundation for inference in linear latent non-Gaussian models (LLnGMs), where non-Gaussianity is introduced through variance-mixture augmentations while retaining conditional Gaussian structure. The main contribution is a set of verifiable conditions under which the Markov chain admits strong stability properties (geometric ergodicity), thereby supporting both Monte Carlo estimation and stochastic-gradient descent likelihood optimization.

A central message of the paper is that geometric ergodicity can be established via two complementary mechanisms, each illuminating different aspects of the chain. The trace-class result in Theorem~\ref{thm:trace-class} provides an operator-theoretic certificate of rapid mixing for the marginal chain on the mixing variables. When the associated integral operator is trace-class, one obtains strong spectral regularity and, in particular, geometric ergodicity for the Markov chain induced by the Gibbs updates. Conceptually, this route is global: it relies on compactness/summability properties of the transition operator and requires relatively mild, distribution-level assumptions in regimes where the conditional densities possess sufficient regularization. The drift and minorization approach complements the trace-class argument by targeting regimes not covered by Theorem~\ref{thm:trace-class}. In these cases, geometric ergodicity follows from the construction of a Lyapunov function together with a small-set condition on an appropriate region of state space. This route is local and modular: it makes explicit how different components of the model contribute to stability, and it can be tuned to handle boundary parameter regimes where the mixing distribution is more singular near $0$ or exhibits heavier tails. Practically, the drift argument also provides quantitative insight into which moments must be controlled to ensure contraction. Taken together, these two routes yield a clear proof strategy: apply the trace-class result whenever available for a concise conclusion; otherwise, use drift--minorization to close the remaining boundary cases. This division avoids overly technical assumptions in well-behaved regimes while preserving coverage of practically relevant parameterizations.

Some conditions in the ergodicity analysis are most naturally interpreted through the geometry of the latent operator and the observation operator. In particular, null-space interactions (e.g., $\Null(B)$ and related projectors) determine whether certain directions of the latent field are weakly informed by the data and hence can lead to near-degeneracies in the mixing updates. The null-smallness conditions appearing in our drift analysis formalize the requirement that the observation mechanism injects sufficient information into the latent directions that would otherwise be weakly controlled by the prior operator. In this sense, the assumptions are not merely technical: they encode a concrete identifiability principle. Moreover, the projection-based constants entering the bounds (such as $\|P_{\Null(B)}\mathbf 1\|_2$ and its computable surrogates) provide a way to diagnose problematic regimes in specific discretizations.

More broadly, the proof strategy developed in this paper is not tied to the specific choice of generalized inverse Gaussian (GIG) mixing distributions. Rather, it applies to a wider class of normal mean–variance mixture models, provided that the mixing variables admit suitable control of their asymptotic behavior near the boundary and in the tails. From a technical standpoint, the ergodicity arguments rely only on the existence of upper and lower envelope bounds for the conditional density of the mixing variables that are comparable to linear combinations of GIG-type kernels. In particular, the drift and minorization constructions require explicit control of the leading-order behavior of the mixing density as $V \to 0$ and $V \to \infty$, together with integrability of a finite number of inverse and positive moments.
Consequently, the trace-class and drift-based arguments can be viewed as structural results for Gibbs samplers in conditionally Gaussian latent models with heavy-tailed or skewed marginals, rather than as distribution-specific proofs. This observation suggests that the theoretical framework developed here may be extended to alternative mixture constructions by verifying a small set of asymptotic conditions on the mixing distribution, without revisiting the full operator or drift analysis from first principles. In this sense, the present work provides a reusable blueprint for stability analysis in a much wider class of latent non-Gaussian models.

Geometric ergodicity is not only a qualitative statement about convergence; it underpins the validity and reliability of Monte Carlo estimators used inside likelihood optimization. In particular, the stochastic gradient estimators developed via Fisher's identity and Rao--Blackwellization require integrability under the invariant law together with sufficiently fast mixing of the Gibbs chain. The results of this paper provide these stability guarantees in the parameter regimes of interest, thereby justifying the use of ergodic averages for gradient estimation and supporting the convergence claims of stochastic optimization algorithms built on top of the sampler. From an implementation viewpoint, these results motivate the use of parameterizations and update blocks that preserve the conditional Gaussian structure while improving mixing.
Although centered and non-centered Gibbs kernels target the same posterior distribution, their mixing behavior can differ substantially. The non-centered parameterization can reduce posterior coupling in regimes where the latent field is weakly identified by the data, while the centered parameterization may be preferable when the likelihood is informative and the latent field is tightly constrained. The analysis in this paper emphasizes that the stability properties needed for inference can be established for either formulation, and it clarifies which distributional regimes are naturally handled by operator-theoretic arguments versus drift constructions. This perspective suggests practical guidelines: use the parameterization that empirically reduces autocorrelation, while ensuring that the resulting chain remains within the theoretically covered regime.

The theoretical results are tailored to LLnGMs that preserve conditional Gaussianity through variance-mixture augmentation. Extensions that break conjugacy (e.g., nonlinear observation models or non-Gaussian likelihoods without an auxiliary-variable representation) may require different Markov chain constructions and new stability arguments. In addition, while the drift--minorization approach provides broad coverage, it can yield conservative constants and may not directly translate into sharp non-asymptotic bounds on Monte Carlo error.

\section*{Acknowledgments}

This publication is based upon work supported by King Abdullah University of Science and Technology (KAUST) under Award No. ORFS-CRG11-2022-5015.

\appendix

\section{Modified Bessel functions of the second kind}\label{facts-about-the-modified-bessel-function-of-the-second-kind}

We summarize several useful properties regarding the modified Bessel function of the second kind, $K_{\nu}(x)$. The following asymptotic properties can be found in \citeproc{abramowitz1965}{Abramowitz and Stegun 1965}, page 375:
\begin{align}
    K_{\nu}(x) & = K_{-\nu}(x) \label{eq:bessel-symmetry}   \\
    K_{\nu}(x) & \sim
    \sqrt{\frac{\pi}{2x}} e^{-x}
    \text{  as } x \rightarrow \infty
    \label{eq:bessel-asymp-inf}                             \\
    K_{\nu}(x) & \sim
    \frac{\Gamma(\nu)}{2} \left( \frac{x}{2} \right)^{-\nu}
    \text{ for } \nu > 0 \text{ as } x \rightarrow 0
    \label{eq:bessel-asymp}                                 \\
    K_{0}(x)   & \sim - \log(x) \text{ as } x \rightarrow 0.
    \label{eq:bessel-asymp-nu-0}
\end{align}
In addition, our analysis relies on several ratio inequalities established in \cite{pal2014}.

\begin{proposition}[Modified Bessel function inequality {\cite[Prop.~A2]{pal2014}}]\label{prp:bessel-inequalities-A2}
  Let \(\nu_1, \nu_2 > 0\). Then for arbitrary \(\epsilon^* > 0\) there exists \(\epsilon > 0\) such that
  \begin{align*}
    \frac{K_{\nu_1}(x)}{K_{\nu_2}(x)} \leq (1+\epsilon^*) \frac{\Gamma(\nu_1) 2^{\nu_1 - \nu_2}}{\Gamma(\nu_2) x^{\nu_1 - \nu_2}}
  \end{align*}
  for all \(0 < x < \epsilon\), where \(\epsilon\) depends on \(\epsilon^*\), \(\nu_1\), and \(\nu_2\).
\end{proposition}

\begin{proposition}[Modified Bessel function inequality {\cite[Prop.~A3]{pal2014}}]\label{prp:bessel-inequalities-A3}
  Let \(\nu > 0\). Then there exists \(\epsilon' > 0\) such that
  \begin{align*}
    \frac{K_{\nu}(x)}{K_{0}(x)} \leq \frac{2 ^{\frac{1}{4}} \sqrt{\pi}}{  2 \Gamma(\frac{1}{4}) x^{\nu}}
  \end{align*}
  for all \(0 < x < \epsilon'\), where \(\epsilon'\) depends on \(\nu\).
\end{proposition}

\begin{proposition}[Modified Bessel function inequality {\cite[Prop.~A4]{pal2014}}]\label{prp:bessel-inequalities-A4}
  Let \(\nu \geq 0\), \(\delta \in (0, 1)\) and \(x > 0\). Then
  \begin{align*}
    \frac{K_{\nu+\delta}(x)}{K_{\nu}(x)} \leq \frac{(2 \nu + 1)^{\frac{\delta}{2}}}{x^{\frac{\delta}{2}}} + 1.
  \end{align*}
\end{proposition}

\section{Proofs}\label{app:proofs}

\subsection{Proofs for Section \ref{chp:gibbs}}
\label{app:proofs-model}

\begin{proof}[Proof of Proposition~\ref{prop:conditional-distributions}]
The derivations for both parameterizations rely on the conjugacy of the Gaussian likelihood with respect to the Gaussian process prior, and the fact that the GIG distribution is conjugate for the variance components in a Gaussian variance-mixture model.
Starting with the centered parameterization for $W\mid V,Y$, the joint distribution of the observations and the latent process is given by $Y\mid W,V \sim N(X\beta+AW, \sigma_\epsilon^2 I)$ and $KW\mid V \sim N(\mu(V-h), \sigma^2 D_V)$. The conditional density $p(W\mid V,Y)$ is proportional to the product of their respective kernels:
\begin{align*}
p(W\mid V,Y) &\propto \exp\biggl\{ -\frac{1}{2\sigma_\epsilon^2}\|Y-X\beta-AW\|^2 \\
&\qquad - \frac{1}{2\sigma^2}(KW-\mu(V-h))^\top D_V^{-1}(KW-\mu(V-h)) \biggr\}.
\end{align*}
Expanding the quadratic forms and collecting terms involving $W$ yields a Gaussian kernel $\exp\{-\frac{1}{2}(W^\top Q W - 2 W^\top r)\}$, where the precision matrix $Q$ and the potential vector $r$ are defined as
\[
Q = \sigma^{-2}K^\top D_V^{-1}K + \sigma_\epsilon^{-2}A^\top A, \quad r = \sigma^{-2}K^\top D_V^{-1}\mu(V-h) + \sigma_\epsilon^{-2}A^\top(Y-X\beta).
\]
Completing the square leads to the conditional distribution $W\mid V,Y \sim N(Q^{-1}r, Q^{-1})$, confirming \eqref{eq:centered-conditional}.

For the non-centered parameterization $M\mid V,Y$, we consider the transformation $W = K^{-1}(M-\mu h)$. Defining $B = AK^{-1}$, the likelihood for $Y$ simplifies to $N(X\beta - \mu Bh + BM, \sigma_\epsilon^2 I)$, while the prior for $M$ is $N(\mu V, \sigma^2 D_V)$. The conditional density $p(M\mid V,Y)$ thus follows:
\begin{align*}
p(M\mid V,Y) &\propto \exp\biggl\{ -\frac{1}{2\sigma_\epsilon^2}\|Y-X\beta+\mu Bh-BM\|^2 \\
&\qquad - \frac{1}{2\sigma^2}(M-\mu V)^\top D_V^{-1}(M-\mu V) \biggr\}.
\end{align*}
Grouping the terms in $M$ again results in a Gaussian kernel with
\[
Q = \sigma_\epsilon^{-2}B^\top B + \sigma^{-2}D_V^{-1}, \quad r = \sigma_\epsilon^{-2}B^\top(Y-X\beta+\mu Bh) + \sigma^{-2}D_V^{-1}\mu V.
\]
Utilizing the identity $D_V^{-1}V = \mathbf{1}$, the linear term simplifies to $r = \sigma_\epsilon^{-2}B^\top(Y-X\beta+\mu Bh) + (\mu/\sigma^2)\mathbf{1}$, which matches \eqref{eq:non-centered-conditional}.

Regarding the update for the variance components $V$, note that in both cases $Y$ is conditionally independent of $V$ once the latent process is given. For the centered case $V\mid W$, the independence of $(V_i)_i$ allows us to factorize the density as $p(V\mid W) \propto \prod_i p(V_i) p(z_i \mid V_i)$, where $z = KW$. For each $i$, the kernel is
\begin{align*}
p(V_i \mid W) &\propto V_i^{p-1} e^{-\frac{1}{2}(aV_i + b/V_i)} \cdot V_i^{-1/2} \exp\left\{-\frac{(z_i + \mu h_i - \mu V_i)^2}{2\sigma^2 V_i}\right\}.
\end{align*}
Expanding the exponent and retaining only the terms dependent on $V_i$ yields
\[
p(V_i \mid W) \propto V_i^{p-1/2-1} \exp\left\{ -\frac{1}{2} \left[ \left(a + \frac{\mu^2}{\sigma^2}\right)V_i + \left(b + \frac{(z_i + \mu h_i)^2}{\sigma^2}\right)\frac{1}{V_i} \right] \right\},
\]
which is the kernel of the $\mathrm{GIG}(p-1/2, a + \mu^2/\sigma^2, b + (z_i + \mu h_i)^2/\sigma^2)$ distribution. 

Finally, for the non-centered case $V\mid M$, a similar expansion of the prior $M_i \mid V_i \sim N(\mu V_i, \sigma^2 V_i)$ combined with the GIG prior on $V_i$ gives
\begin{align*}
p(V_i \mid M) \propto V_i^{p-1/2-1} \exp\left\{ -\frac{1}{2} \left[ \left(a + \frac{\mu^2}{\sigma^2}\right)V_i + \left(b + \frac{M_i^2}{\sigma^2}\right)\frac{1}{V_i} \right] \right\},
\end{align*}
confirming the expression in \eqref{eq:V_given_M} and completing the proof.
\end{proof}

\subsection{Proofs for Section \ref{chp:estimation}}
\label{app:proofs-estimation}

\begin{proof}[Proof of Lemma~\ref{lem:L1_threshold_centered}]
Let $Z = KW$ and $S_i = V_i^{-1} h_i (Z_i + \mu h_i)$. To simplify the exposition, we fix the index $i$ and drop it where possible, writing $v, h, z, w$ for $V_i, h_i, Z_i, W_i$ respectively. Throughout the proof, we use lower-case letters to denote the variables of integration. By Bayes' theorem, the posterior expectation of $|S_i|$ is given by
\[
\mathbb{E}[|S_i| \mid y] = \frac{1}{\pi(y)} \int \frac{|h| \, |z + \mu h|}{v} p(y \mid w) p(z, v) \, dz \, dv.
\]
\smallskip\noindent
We first show that for  $\alpha>\tfrac12$ the integral is finite.
Since the Gaussian likelihood is globally bounded,
$
p(y\mid w)\le L_{\max}:=(2\pi\sigma_\epsilon^2)^{-n/2},
$
we obtain
$
\mathbb E[|S_i|\mid y]\le \frac{L_{\max}}{\pi(y)}\,\mathbb E\!\left[\frac{|h|\,|z+\mu h|}{V}\right].
$
Define $T:=z_i+\mu h$, we have $T\mid v\sim \mathcal N(\mu v,\sigma^2 v)$, hence
$\mathbb E(|T|\mid v)\le c_1\sqrt v + |\mu| v$ with $c_1:=\sigma\sqrt{2/\pi}$, and therefore
\[
\mathbb E\!\left[\frac{|h|\,|T|}{v}\right]
\le |h|\int_0^\infty (c_1 v^{-1/2}+|\mu|)\,p_V(v)\,dv,
\quad p_V(v)\propto v^{\alpha-1}e^{-\beta v}.
\]
Near $0$ the integrand behaves like $v^{-1/2}v^{\alpha-1}=v^{\alpha-3/2}$, so this integral is finite iff
$\alpha>\tfrac12$.

Now we show that the expectation is divergent for $\alpha\le\tfrac12$.
Write  $\tilde z:=z_{-i}$. We want to find a set where we can bound the various terms of the integral from below  $\E(|S_i|\mid y)$. 
Let $z^\star:=-\mu h\mathbf 1$ and choose $\eta>0$. Define the cube
\[
Q:=\{z\in\R^n:\ \|z-z^\star\|_\infty\le \eta\},
\qquad
Q_{-i}:=\{\tilde z\in\R^{n-1}:\ \|\tilde z-\tilde z^\star\|_\infty\le \eta\}.
\]
Set $B:=K^{-1}Q$. Since $w\mapsto p(y\mid w)$ is continuous and strictly positive, there exists a
$
L_{\min}:=\inf_{w\in B}p(y\mid w)>0.
$
Thus, for all $z\in Q$ we have $p(y\mid K^{-1}z)\ge L_{\min}$, and hence
\begin{align*}
\E[|S_i|\mid y]
&=\frac{1}{\pi(y)}\int \frac{|h|\,|z_i+\mu h|}{v}\,p(y\mid K^{-1}z)\,p(z,v)\,dz\,dv\\
&\ge \frac{L_{\min}}{\pi(y)} \int_Q \int \frac{|h|\,|z_i+\mu h|}{v}\,p(z,v)\,dz\,dv.
\end{align*}

Fix $b>0$ and define
$
I(v):=[\,b\sigma\sqrt v,\ 2b\sigma\sqrt v\,].
$
 Choose $\varepsilon=b/4$ and pick $\delta>0$ such that
$
2b\sigma\sqrt{\delta}\le \eta$ and 
$\frac{|\mu|}{\sigma}\sqrt{\delta}\le \varepsilon$ (to be used later).
Then for $v\in(0,\delta)$, $t\in I(v)$ and $\tilde z\in Q_{-i}$ we have
$(z_i^\star+t,\tilde z)\in Q$.
We thus obtain
\begin{align*}
 \int_Q \int \frac{|h|\,|z_i+\mu h|}{v}&\,p(z,v)\,dz\,dv \\
&\ge \int_0^\delta\int_{Q_{-i}}\int_{t\in I(v)}
\frac{|h|\,t}{v}\,p_{Z_i\mid V_i}(z_i^\star+t\mid v)\,p_V(v)\,
\prod_{j\neq i} p_{Z_j}(z_j)\,dt\,d\tilde z\,dv\\
&= c_{-i}\int_0^\delta
\left[\int_{t\in I(v)}\frac{|h|\,t}{v}\,p_{Z_i\mid V_i}(z_i^\star+t\mid v)\,dt\right]p_V(v)\,dv,
\end{align*}
where
$
c_{-i}:=\int_{Q_{-i}}\prod_{j\neq i} p_{Z_j}(z_j)\,d\tilde z\ >0.
$

For $t\in I(v)$ we have $t/v\ge b\sigma v^{-1/2}$. Hence for all small $v$,
\[
\begin{split}
\int_0^\delta
&\left[\int_{t\in I(v)}\frac{|h|\,t}{v}\,p_{Z_i\mid V_i}(z_i^\star+t\mid v)\,dt\right]p_V(v)\,dv 
\\ &\ge\ 
\int_0^\delta b \sigma |h|\,v^{-1/2}
\left[\int_{t\in I(v)} p_{Z_i\mid V_i}(z_i^\star+t\mid v)\,dt\right]p_V(v)\,dv.
\end{split}
\]

Let $s(v):=\frac{\mu}{\sigma}\sqrt v$, since $\frac{|\mu|}{\sigma}\sqrt{\delta}\le \varepsilon$ we have
 $|s(v)|\le \varepsilon$ for all $v\in(0,\delta)$.
Let $f(s):=\Phi(2b-s)-\Phi(b-s)$, and define
$
p_0:=\min_{|s|\le \varepsilon} f(s)>0
$
which gives
\begin{align*}
\int_{t\in I(v)} p_{Z_i\mid V_i}(z_i^\star+t\mid v)\,dt
&=\int_{b\sigma\sqrt v}^{2b\sigma\sqrt v} 
\frac{1}{\sqrt{2\pi\sigma^2 v}}
\exp\!\left(-\frac{(t-\mu v)^2}{2\sigma^2 v}\right)\,dt \\
&= f(s(v))\ \ge\ p_0, \forall v\in(0,\delta).
\end{align*}
Therefore
\[
\E[|S_i|\mid y]\ \ge\ c\int_0^\delta v^{-1/2}\,p_V(v)\,dv
\ \asymp\ \int_0^\delta v^{-1/2}v^{\alpha-1}\,dv
=\int_0^\delta v^{\alpha-3/2}\,dv=\infty\qquad(\alpha\le\tfrac12),
\]
as claimed.
\end{proof}

\subsection{Proofs for Section \ref{sec:proof-trace-class}}

\begin{proof}[Proof of Lemma~\ref{lem:quadratic_form_lemma}]
  First, we establish a useful matrix identity. Since $A$ and $B$ are positive definite, they are invertible, and their sum $A+B$ and the sum of their inverses $A^{-1}+B^{-1}$ are also positive definite (and therefore invertible). We can write
  \[
    A^{-1} + B^{-1} = A^{-1}(A+B)B^{-1}.
  \]
  Taking the inverse of both sides yields the identity
  \begin{equation}\label{eq:matrix_identity}
    (A^{-1} + B^{-1})^{-1} = B(A+B)^{-1}A,
  \end{equation}
  which by symmetry is also equal to $A(A+B)^{-1}B$.
  Using the representation $I = (A+B)(A+B)^{-1}$, we can rewrite the term involving $A$ as
  \begin{align*}
    A - (A^{-1} + B^{-1})^{-1} & = A - A(A+B)^{-1}B                 \\
    & = A(A+B)^{-1}(A+B) - A(A+B)^{-1}B 
 = A(A+B)^{-1}A.
  \end{align*}
  Similarly, for the term involving $B$, we have $B - (A^{-1} + B^{-1})^{-1} = B(A+B)^{-1}B$.
  Now, let us expand the expression on the left-hand side of the lemma. Let $H = (A^{-1}+B^{-1})^{-1}$. We have
  \begin{equation*}
    \begin{split}
        u^{\top}Au + v^{\top}Bv - (u+v)^{\top}H(u+v)                 
       & = u^{\top}Au + v^{\top}Bv - (u^{\top}Hu + v^{\top}Hv + 2u^{\top}Hv) \\
       & = u^{\top}(A-H)u + v^{\top}(B-H)v - 2u^{\top}Hv = (I).
    \end{split}
  \end{equation*}
  Substituting the identities derived above into this expression:
  \begin{equation*}
    \begin{split}
       (I) & = u^{\top}\big(A(A+B)^{-1}A\big)u + v^{\top}\big(B(A+B)^{-1}B\big)v - 2u^{\top}\big(A(A+B)^{-1}B\big)v \\
       & = (Au)^{\top}(A+B)^{-1}(Au) + (Bv)^{\top}(A+B)^{-1}(Bv) - 2(Au)^{\top}(A+B)^{-1}(Bv)                   \\
       & = (Au - Bv)^{\top}(A+B)^{-1}(Au - Bv).
    \end{split}
  \end{equation*}
  This proves the equality. Since $A+B$ is positive definite, the quadratic form on the last line is non-negative, and the desired inequality follows.
\end{proof}

\begin{proof}[Proof of Lemma~\ref{lem:K_bound}]
  Define
  \[
    f(x):=K_p(x)\,x^{|p|}(1+x)^{\frac12-|p|}e^{x},\qquad x>0.
  \]
  Since $K_p(x)>0$ for $x>0$, we have $f(x)>0$ for all $x>0$.
  Moreover, $f$ is continuous on $(0,\infty)$ because $K_p$ is continuous on $(0,\infty)$.
  Using the standard asymptotics of $K_p$ (see Appendix~\ref{facts-about-the-modified-bessel-function-of-the-second-kind}):
  \begin{itemize}
    \item As $x\to\infty$, $K_p(x)\sim \sqrt{\frac{\pi}{2x}}e^{-x}$, hence
          \[
            f(x)\to \sqrt{\frac{\pi}{2}}\cdot \lim_{x\to\infty}
            \frac{x^{|p|}(1+x)^{\frac12-|p|}}{\sqrt{x}}
            =\sqrt{\frac{\pi}{2}}>0.
          \]
    \item As $x\to0^+$, if $p\neq 0$ then $K_p(x)\sim 2^{|p|-1}\Gamma(|p|)\,x^{-|p|}$, hence
            $f(x)\to 2^{|p|-1}\Gamma(|p|)>0$,
          while if $p=0$ then $K_0(x)\to\infty$, so $f(x)\to\infty$.
  \end{itemize}

  Therefore, there exist $\varepsilon\in(0,1)$ and $M>1$ such that
  \[
    f(x)\ge c_0>0\quad \text{for all }x\in(0,\varepsilon]\cup[M,\infty).
  \]
  On the remaining compact interval $[\varepsilon,M]$, continuity implies that $f$ attains its minimum:
  \[
    m:=\min_{x\in[\varepsilon,M]} f(x)>0,
  \]
  since $f>0$ everywhere. Hence
  \[
    \inf_{x>0} f(x)\ \ge\ \min\{c_0,m\}\ >\ 0.
  \]
  Let $C:=\min\{c_0,m\}$ (or any $0<C<\inf_{x>0}f(x)$ if you insist on a strict inequality).
  Then for all $x>0$, $f(x)\ge C$, which is equivalent to
    $K_p(x)\ \ge\ C\,x^{-|p|}(1+x)^{-\frac12+|p|}e^{-x}$.
\end{proof}

\begin{proof}[Proof of Lemma~\ref{lem:normal_moment_bound}]
  Let $\varphi$ denote the density of the distribution
  $N(A^{-1}b, A^{-1})$.
  Then the claim is equivalent to showing that
  \begin{equation}\label{eq:moment_bound_phi}
    \int \prod_{i=1}^n (\gamma + x_i^2)^p \, \varphi(x)\,dx
    \le
    C \prod_{i=1}^n
    \Big(
    \gamma^p + [A^{-1}]_{ii}^p
    \big(1 + \sum_{j=1}^n [A^{-1}]_{jj}^p\big)
    \Big),
  \end{equation}
  where $C = C(p,n,\|b\|)$.
  Let $k > np$ be an integer. By Hölder's inequality,
  \[
    \int \prod_{i=1}^n (\gamma + x_i^2)^p \varphi(x)\,dx
    \le
    \prod_{i=1}^n
    \left(
    \int (\gamma + x_i^2)^{np} \varphi(x)\,dx
    \right)^{1/n}.
  \]
  Since $k > np$, the monotonicity of $L^r$ norms yields
  \[
    \int (\gamma + x_i^2)^{np} \varphi(x)\,dx
    \le
    \left(
    \int (\gamma + x_i^2)^k \varphi(x)\,dx
    \right)^{np/k}.
  \]
  Combining the two inequalities,
  \[
    \int \prod_{i=1}^n (\gamma + x_i^2)^p \varphi(x)\,dx
    \le
    \prod_{i=1}^n
    \left(
    \int (\gamma + x_i^2)^k \varphi(x)\,dx
    \right)^{p/k}.
  \]
  Let $m = A^{-1}b$.
  Using the inequality
    $x_i^2 \le 2 m_i^2 + 2(x_i - m_i)^2$,
  we obtain
  \[
    (\gamma + x_i^2)^k
    \le
    (\gamma + 2m_i^2 + 2(x_i-m_i)^2)^k
    \le
    3^{k-1}\big(
    \gamma^k + 2^k m_i^{2k} + 2^k (x_i-m_i)^{2k}
    \big).
  \]
  Taking expectations with respect to $\varphi$ and absorbing constants into $C$,
  \[
    \int (\gamma + x_i^2)^k \varphi(x)\,dx
    \le
    C\big(
    \gamma^k + |m_i|^{2k} + [A^{-1}]_{ii}^k
    \big),
  \]
  since $\mathbb{E}_\varphi[(x_i-m_i)^{2k}] = C_k [A^{-1}]_{ii}^k$.
  Raising both sides to the power $p/k$ and using that $p/k \le 1$, we obtain
  \[
    \left(
    \int (\gamma + x_i^2)^k \varphi(x)\,dx
    \right)^{p/k}
    \le
    C\big(
    \gamma^p + |m_i|^{2p} + [A^{-1}]_{ii}^p
    \big).
  \]
  Finally, since $m = A^{-1}b$ and $b$ is fixed, we have
  \[
    |m_i|
    \le
    \|b\| \sum_{j=1}^n |[A^{-1}]_{ij}|
    \le
    \|b\| [A^{-1}]_{ii}^{1/2}
    \sum_{j=1}^n [A^{-1}]_{jj}^{1/2},
  \]
  which implies
  \[
    |m_i|^{2p}
    \le
    C(\|b\|)\,[A^{-1}]_{ii}^p
    \sum_{j=1}^n [A^{-1}]_{jj}^p.
  \]
  Substituting this bound into \eqref{eq:moment_bound_phi} yields the desired result.
\end{proof}

\subsection{Proofs for Section \ref{sec:geometric-ergodicity-proof}}

\begin{proof}[Proof of Lemma~\ref{lem:not-trace-class}]
In this setting, the marginal data augmentation algorithm corresponds to the operator $\Lambda$ defined by the joint density $\pi(V, M)$. A necessary condition for $\Lambda$ to be trace class is the finiteness of the integral
  \[
    T := \int \int \pi(V_i \mid M_i) \pi(M_i \mid V_i) \, dV_i \, dM_i.
  \]
  We show that this integral diverges.
  
  Under the assumption $a=\mu=0$, we have $p<0$. Let $\alpha := -p > 0$ and $\beta := b/2 > 0$. The prior on $V_i$ is $V_i \sim \text{Inv-Gamma}(\alpha, \beta)$.
  When $B=0$, the conditional distributions reduce to:
  \begin{align*}
    M_i \mid V_i & \sim N(0, V_i),                                                                  \\
    V_i \mid M_i & \sim \text{Inv-Gamma}\left(\alpha + \frac{1}{2}, \beta + \frac{M_i^2}{2}\right).
  \end{align*}
  The corresponding density functions are
  \begin{align*}
    \pi(M_i \mid V_i) &= \frac{1}{\sqrt{2\pi V_i}} \exp\left(-\frac{M_i^2}{2V_i}\right),\\
    \pi(V_i \mid M_i) &= \frac{(\beta + M_i^2/2)^{\alpha+1/2}}{\Gamma(\alpha+1/2)} V_i^{-\alpha-3/2} \exp\left(-\frac{\beta + M_i^2/2}{V_i}\right).
  \end{align*}
  Substituting these into the expression for $T$ and applying Fubini's theorem to integrate with respect to $V_i$ first:
  \begin{align*}
    T & = \int_{-\infty}^\infty \frac{(\beta + M_i^2/2)^{\alpha+1/2}}{\sqrt{2\pi}\,\Gamma(\alpha+1/2)} \left[ \int_0^\infty V_i^{-\alpha-2} \exp\left(-\frac{\beta + M_i^2/2}{V_i}\right) \, dV_i \right] dM_i.
  \end{align*}
  The inner integral is related to the unnormalized Inverse-Gamma density. Using the identity $\int_0^\infty x^{-(k+1)} e^{-C/x} dx = \Gamma(k) C^{-k}$ with $k = \alpha+1$ and $C = \beta + M_i^2/2$, we obtain:
  \[
    \int_0^\infty V_i^{-(\alpha+1)-1} \exp\left(-\frac{\beta + M_i^2/2}{V_i}\right) \, dV_i
    = \Gamma(\alpha+1) \left(\beta + \frac{M_i^2}{2}\right)^{-(\alpha+1)}.
  \]
  Substituting this back into the expression for $T$:
  \begin{align*}
    T & = \frac{\Gamma(\alpha+1)}{\sqrt{2\pi}\,\Gamma(\alpha+1/2)} \int_{-\infty}^\infty \left(\beta + \frac{M_i^2}{2}\right)^{\alpha+1/2} \left(\beta + \frac{M_i^2}{2}\right)^{-(\alpha+1)} \, dM_i \\
      & = \frac{\Gamma(\alpha+1)}{\sqrt{2\pi}\,\Gamma(\alpha+1/2)} \int_{-\infty}^\infty \left(\beta + \frac{M_i^2}{2}\right)^{-1/2} \, dM_i.
  \end{align*}
  As $|M_i| \to \infty$, the integrand behaves like $(\frac{M_i^2}{2})^{-1/2} \propto \frac{1}{|M_i|}$.
  Since $\int_{1}^\infty \frac{1}{x} dx$ diverges logarithmically, the integral $T$ is infinite. Thus, the operator is not trace class.
\end{proof}

The following two lemmas provide auxiliary results needed for the proof of Lemma~\ref{lem:eta-range-bound}. To state them, we need a few definitions. Let $A \in \mathbb{R}^{n \times n}$. The index of a row is its order, that is the row $[a_{i,j}]_{j \in \{1,\ldots,n\}}$ has index $i$; the index is defined similarly for columns. Let $S_r \subseteq \{1, 2, \ldots, n\}^r$ be the set of strictly increasing sequences of length $r$ with elements taking value in the set $\{1, \ldots, n\}$. For $r \in \{0, \ldots, n\}$, $\omega, \varepsilon \in S_r$, and a matrix $A \in \mathbb{R}^{n \times n}$, define $A(\omega|\varepsilon)$ to be the matrix constructed by removing the rows whose index appear in the sequence $\omega$ and removing columns whose index appear in $\varepsilon$. Define $A[\omega|\varepsilon]$ to be the matrix obtained by removing the rows whose indices do not appear in $\omega$ and columns whose indices do not appear in $\varepsilon$.

\begin{lemma}[Cauchy-Binet formula, Equation (1) in \cite{marcus1990}]
  \label{lem:cauchy-binet}
  Let $A, B \in \mathbb{R}^{n,n}$. We have
  \[
    \det(A + B) = \sum_{r=0}^{n} \sum_{\omega,\varepsilon \in S_r} (-1)^{s(\omega)+s(\varepsilon)} \det(A[\omega|\varepsilon]) \det(B(\omega|\varepsilon)).
  \]
\end{lemma}

\begin{lemma}[Cauchy-Schwarz inequality for minors]
  \label{lem:minor-inequality}
  Let $A \in \mathbb{R}^{n \times n}$ be a symmetric non-negative definite matrix and let $i, j$ be distinct integers in $\{1, \ldots, n\}$. We have
  \[
    \det(A(\omega \cup \{i\}|\omega \cup \{i\})) \det(A(\omega \cup \{j\}|\omega \cup \{j\}))
    \ge \det(A(\omega \cup \{i\}|\omega \cup \{j\}))^2.
  \]
\end{lemma}

\begin{proof}
  Because $A$ is symmetric non-negative definite, so is $A(\omega|\omega)$. Now, the desired inequality is a consequence of the fact that the matrix
  \[
    \mathrm{adj}(A(\omega|\omega)) = \lim_{\theta \to 0^+} \det(A(\omega|\omega) + \theta I)(A(\omega|\omega) + \theta I)^{-1}
  \]
  is a symmetric non-negative definite matrix.
\end{proof}

\begin{proof}[Proof of Lemma~\ref{lem:eta-range-bound}]
  Notice that $\mathrm{Ker}(B^\top B) = \mathrm{Ker}(B)$, and consequently
  \[
    \mathrm{Range}(B^\top B) = \mathrm{Ker}(B^\top B)^\perp = \mathrm{Ker}(B)^\perp = \mathrm{Range}(B^\top).
  \]
  Denote by $\mathcal{B}$ the matrix $B^\top B$ and let $B_1, \ldots, B_n$ denote the columns of the matrix $\mathcal{B}$.
  It is enough to prove the claim for $\bar m = B_j$ for every $j \in \{1, \ldots, n\}$.
  Our goal is to prove that $[(\mathcal{B} + D_V^{-1})^{-1} B_j]_i$ is bounded.

  By Cramer's formula, we have
  \begin{equation}\label{eq:cramer}
    [(\mathcal{B} + D_V^{-1})^{-1} B_j]_i
    = \frac{\det(E_i + D_{V,i}^{-1})}{\det(\mathcal{B} + D_V^{-1})},
  \end{equation}
  where the matrix $E_i$ is defined by replacing the $i$-th column of $\mathcal{B}$ by $B_j$ and $D_{V,i}$ is defined by replacing the $i$-th diagonal element in $D_V$ by $0$.

  By subtracting the $j$-th column of $E_i + D_{V,i}^{-1}$ from its $i$-th column we have
  \[
    \det(E_i + D_{V,i}^{-1}) = \det(F_i + D_{V,i,j}^{-1}),
  \]
  where $F_i$ is the matrix $\mathcal{B}$ with the $i$-th column replaced with $0$ and $D_{V,i,j}$ is the matrix $D_V$ with the $i$-th column replaced with negative the $j$-th column of $D_V$.

  By Lemma~\ref{lem:cauchy-binet}, we have
  \begin{align*}
    \det(F_i + D_{V,i,j}^{-1})
     & = \sum_{r} \sum_{\omega,\varepsilon \in S_r} (-1)^{s(\omega)+s(\varepsilon)} \det(F_i[\omega|\varepsilon]) \det(D_{V,i,j}^{-1}(\omega|\varepsilon)) \\
     & = \sum_{r} \sum_{\substack{\omega \in S_{r-1}                                                                                                       \\ \omega \not\ni i,j}} \det(F_i(\omega \cup \{j\}|\omega \cup \{i\})) \vartheta_\omega \prod_{k \in \omega \cup \{j\}} V_k^{-1} \\
     & = \sum_{r} \sum_{\substack{\omega \in S_{r-1}                                                                                                       \\ \omega \not\ni i,j}} \det(\mathcal{B}(\omega \cup \{j\}|\omega \cup \{i\})) \vartheta_\omega \prod_{k \in \omega \cup \{j\}} V_k^{-1},
  \end{align*}
  where $\vartheta_\omega = \pm 1$. The previous equality follows by studying the cases of whether $i, j$ belong to $\omega, \varepsilon \in S_r$.

  By Lemma~\ref{lem:minor-inequality} we have
  \begin{equation}\label{eq:numerator-bound}
    \begin{split}
      |\det(F_i + D_{V,i,j}^{-1})|
      & \le \max_{\varepsilon \not\ni i} \sqrt{\det(\mathcal{B}(\varepsilon \cup \{i\}|\varepsilon \cup \{i\}))} \\
      & \quad \times \sum_{r} \sum_{\substack{\omega \in S_{r-1} \\ \omega \not\ni i,j}}
      \sqrt{\det(\mathcal{B}(\omega \cup \{j\}|\omega \cup \{j\}))} \prod_{k \in \omega \cup \{j\}} V_k^{-1}.
    \end{split}
  \end{equation}

  On the other hand by Lemma~\ref{lem:cauchy-binet} again we have
  \begin{equation}\label{eq:denominator}
    \det(\mathcal{B} + D_V^{-1}) = \sum_r \sum_{\omega \in S_r} \det(\mathcal{B}(\omega|\omega)) \prod_{k \in \omega} V_k^{-1}.
  \end{equation}

  To conclude, observe that both the numerator in \eqref{eq:numerator-bound} and the denominator in \eqref{eq:denominator} are polynomials in $\{V_k^{-1}\}_k$ with the same degree. Moreover, all polynomial terms appearing in \eqref{eq:numerator-bound} also appear in \eqref{eq:denominator} with nonzero coefficients. Further, since $\mathcal{B} = B^\top B$ is positive semidefinite, the coefficients in \eqref{eq:denominator} are determinants of principal minors of $\mathcal{B}$ and are therefore nonnegative. Consequently, the denominator has no positive roots (it is strictly positive for all $V \in (0,\infty)^n$). Thus, the quantity \eqref{eq:cramer}, being the ratio of two such polynomials, is bounded uniformly over $V \in (0,\infty)^n$.
\end{proof}

\begin{proof}[Proof of Lemma~\ref{lem:nullproj-1-via-A}]
  Since $K$ is invertible, 
  $$\Null(B)=\{x:AK^{-1}x=0\}=K\,\Null(A)=\Range(KU_A).$$
  Let $W:=K U_A$. Then $W$ has full column rank and $\Range(W)=\Null(B)$.
  The Euclidean orthogonal projector onto $\Range(W)$ is
    $P_{\Range(W)}=W(W^\top W)^{-1}W^\top$.
  Therefore
  \[
    P_{\Null(B)}=P_{\Range(W)}
    =K U_A\,(U_A^\top K^\top K U_A)^{-1}U_A^\top K^\top
    =K U_A\,G^{-1}U_A^\top K^\top.
  \]
  Taking quadratic forms with $\mathbf 1$ yields
  \[
    \mathbf 1^\top P_{\Null(B)}\mathbf 1
    =\mathbf 1^\top K U_A\,G^{-1}U_A^\top K^\top \mathbf 1
    =z^\top G^{-1}z,
  \]
  which proves \eqref{eq:nullproj-identity}. If $\bar m=\rho B^\top(Y-X\beta+\mu Bh)+\mu\mathbf 1$, then
  $P_{\Null(B)}B^\top=0$ implies $\bar m_0=P_{\Null(B)}\bar m=\mu P_{\Null(B)}\mathbf 1$, and
  \eqref{eq:m0-norm-identity} follows immediately.
\end{proof}

\begin{proof}[Proof of Lemma~\ref{lem:GIG-negmoment-unif}]
  Fix $q\in(0,\tfrac12]$.
  Write $\tilde V:=\tilde V_i$ and $M:=M_i$ for simplicity.
  Conditioned on $M$, we have
  \[
    \tilde V\mid M \sim \mathrm{GIG}\!\left(\lambda,\psi,\chi\right),
    \qquad
    \lambda:=p-\tfrac12,\ \psi:=\tilde a>0,\ \chi:=b+\left(\tfrac{M}{\sigma}\right)^2\ge b>0 .
  \]
  Let $x:=\sqrt{\psi\chi}$.
  A standard moment identity for the GIG distribution (valid for all real $r$ when $\psi>0$ and $\chi>0$) is
  \begin{equation}\label{eq:GIG-moment}
    \mathbb E\!\left[\tilde V^{\,r}\mid M\right]
    =\left(\frac{\chi}{\psi}\right)^{r/2}\frac{K_{\lambda+r}(x)}{K_{\lambda}(x)},
    \qquad r\in\mathbb R,
  \end{equation}
  where $K_\nu(\cdot)$ is the modified Bessel function of the second kind.

  Taking $r=-q$ in \eqref{eq:GIG-moment} gives
  \[
    \mathbb E\!\left[\tilde V^{-q}\mid M\right]
    =\left(\frac{\chi}{\psi}\right)^{-q/2}\frac{K_{\lambda-q}(x)}{K_{\lambda}(x)}.
  \]

  Step 1: control the prefactor uniformly in $M$.
  Since $\chi\ge b$ and $q>0$, the map $\chi\mapsto \chi^{-q/2}$ is decreasing, hence
  \[
    \left(\frac{\chi}{\psi}\right)^{-q/2}
    \le \left(\frac{b}{\psi}\right)^{-q/2}.
  \]

  Step 2: control the Bessel ratio uniformly in $M$.
  Note that $x=\sqrt{\psi\chi}\ge \sqrt{\psi b}=:x_0>0$.
  Define the continuous positive function
  \[
    R(x):=\frac{K_{\lambda-q}(x)}{K_{\lambda}(x)},\qquad x>0.
  \]
  Because $K_\nu(x)>0$ for all $x>0$ and any $\nu\in\mathbb R$, $R$ is well-defined and continuous on $(0,\infty)$.
  Moreover, using the standard large-$x$ asymptotic $K_\nu(x)\sim \sqrt{\pi/(2x)}e^{-x}$,
  we have $R(x)\to 1$ as $x\to\infty$.
  Therefore, $R$ is bounded on $[x_0,\infty)$:
  indeed, choose $X>x_0$ such that $R(x)\le 2$ for all $x\ge X$,
  and set
  \[
    C_K:=\max\Big\{\sup_{x\in[x_0,X]}R(x),\ 2\Big\}<\infty.
  \]
  Hence, for all $M\in\mathbb R$,
  \[
    \frac{K_{\lambda-q}(x)}{K_{\lambda}(x)} \le C_K .
  \]

  Combining Steps 1--2 yields
  \[
    \mathbb E\!\left[\tilde V^{-q}\mid M\right]
    \le
    \left(\frac{b}{\psi}\right)^{-q/2} C_K
    {:}=L_-(p,\tilde a,b,q)<\infty,
  \]
  uniformly over $M\in\mathbb R$.
  The statement for $\sum_{i=1}^n\mathbb E[\tilde V_i^{-q}\mid V]$ follows by summing the bound.
\end{proof}

\begin{proof}[Proof of Lemma~\ref{lem:EV-linear-allp}]
  Write $\tilde V:=\tilde V_i$ and $M:=M_i$.
  As in the previous proof,
  \[
    \tilde V\mid M \sim \mathrm{GIG}(\lambda,\psi,\chi),
    \quad
    \lambda:=p-\tfrac12,\ \psi:=\tilde a>0,\ \chi:=b+\left(\tfrac{M}{\sigma}\right)^2,\
    x:=\sqrt{\psi\chi}.
  \]
  Applying the GIG moment formula \eqref{eq:GIG-moment} with $r=1$ gives
  \[
    \mathbb E[\tilde V\mid M]
    =\left(\frac{\chi}{\psi}\right)^{1/2}\frac{K_{\lambda+1}(x)}{K_{\lambda}(x)}
    =\frac{\sqrt{\chi}}{\sqrt{\psi}}\cdot \frac{K_{\lambda+1}(x)}{K_{\lambda}(x)}.
  \]

  We use a standard inequality for modified Bessel functions (a consequence of continued-fraction / Riccati analysis of the ratio;
  see, e.g., Amos-type bounds in the special-function literature):
  there exists a constant $C_R(\lambda)<\infty$ such that for all $x>0$,
  \begin{equation}\label{eq:Kratio-linear}
    \frac{K_{\lambda+1}(x)}{K_{\lambda}(x)} \le 1 + \frac{C_R(\lambda)}{x}.
  \end{equation}
  (One may take, for instance, $C_R(\lambda)=2|\lambda|+1$.)

  Using \eqref{eq:Kratio-linear} and $x=\sqrt{\psi\chi}$, we obtain
  \[
    \mathbb E[\tilde V\mid M]
    \le \frac{\sqrt{\chi}}{\sqrt{\psi}}
    +\frac{\sqrt{\chi}}{\sqrt{\psi}}\cdot \frac{C_R(\lambda)}{\sqrt{\psi\chi}}
    =
    \frac{\sqrt{\chi}}{\sqrt{\psi}} + \frac{C_R(\lambda)}{\psi}.
  \]

  Since $\chi=b+(M/\sigma)^2$,
  \[
    \sqrt{\chi}\le \sqrt b + \left|\frac{M}{\sigma}\right|.
  \]
  Therefore
  \[
    \mathbb E[\tilde V\mid M]
    \le
    \frac{1}{\sqrt{\psi}}\left|\frac{M}{\sigma}\right|
    +\frac{\sqrt b}{\sqrt{\psi}}
    +\frac{C_R(\lambda)}{\psi}.
  \]
  Recalling $\psi=\tilde a$ and $\lambda=p-\tfrac12$, this is exactly
  \[
    \mathbb E[\tilde V_i\mid M_i]
    \le
    \frac{1}{\sqrt{\tilde a}}\Bigl|\frac{M_i}{\sigma}\Bigr|
    +\frac{\sqrt b}{\sqrt{\tilde a}}
    +\frac{C_{\mathrm{lin}}(p,\tilde a)}{\tilde a},
  \]
  with $C_{\mathrm{lin}}(p,\tilde a):=C_R(p-\tfrac12)$.
\end{proof}

\begin{proof}[Proof of Lemma~\ref{lem:E1V-delta-app}]
  The proof follows similar arguments as in \cite{pal2014}.
  Let
  \[
    \tilde a:=a+\frac{\mu^2}{\sigma^2},\qquad \tilde b:=\frac{M_i^2}{\sigma^2},
    \qquad x:=\sqrt{\tilde a\tilde b}=\frac{\sqrt{\tilde a}\,|M_i|}{\sigma}.
  \]
  Below we assume $\tilde b>0$ (i.e. $M_i\neq 0$); the bound then holds $\pi(M_i|V,Y)-a.s.$.
  Using the standard moment identity for the GIG distribution,
  \begin{align*}
    E\!\left[V_i^{-\frac{\delta(p)}{2}}\mid M_i\right]
     & =\int_0^\infty v^{-\frac{\delta(p)}{2}}
    \left(\frac{\sqrt{\tilde a}}{\sqrt{\tilde b}}\right)^{p-\frac12}
    v^{\left(p-\frac12\right)-1} \\
     & \quad \times \exp\!\left\{-\frac12\left(\tilde a v+\frac{\tilde b}{v}\right)\right\}
    \frac{1}{2K_{p-\frac12}\!\left(\sqrt{\tilde a\tilde b}\right)}\,dv             \\
     & =\left(\frac{\sqrt{\tilde a}}{\sqrt{\tilde b}}\right)^{\frac{\delta(p)}{2}}
    \frac{K_{p-\frac12-\frac{\delta(p)}{2}}\!\left(\sqrt{\tilde a\tilde b}\right)}
    {K_{p-\frac12}\!\left(\sqrt{\tilde a\tilde b}\right)}.
  \end{align*}

  We consider three cases: $p<\frac12$, $p>\frac12$, and $p=\frac12$.

  \paragraph*{Case 1: $p<\frac12$}
  By $K_\nu(x)=K_{-\nu}(x)$,
  \[
    E\!\left[V_i^{-\frac{\delta(p)}{2}}\mid M_i\right]
    =\left(\frac{\sqrt{\tilde a}}{\sqrt{\tilde b}}\right)^{\frac{\delta(p)}{2}}
    \frac{K_{\frac12-p+\frac{\delta(p)}{2}}\!\left(\sqrt{\tilde a\tilde b}\right)}
    {K_{\frac12-p}\!\left(\sqrt{\tilde a\tilde b}\right)}.
  \]
  Apply Proposition~\ref{prp:bessel-inequalities-A2} with
  $\nu_1=\frac12+\frac{\delta(p)}{2}-p$, $\nu_2=\frac12-p$ and $x=\sqrt{\tilde a\tilde b}$:
  for any $\epsilon^*>0$ there exists $\epsilon=\epsilon(\epsilon^*,p)$ such that, for $x<\epsilon$,
  \[
    \frac{K_{\frac12-p+\frac{\delta(p)}{2}}(x)}{K_{\frac12-p}(x)}
    \le (1+\epsilon^*)\,
    \frac{\Gamma\!\left(\frac12+\frac{\delta(p)}{2}-p\right)2^{\frac{\delta(p)}{2}}}
    {\Gamma\!\left(\frac12-p\right)\,x^{\frac{\delta(p)}{2}}}.
  \]
  Apply Proposition~\ref{prp:bessel-inequalities-A4} (same $\nu=\frac12-p$) to bound the ratio when $x\ge\epsilon$:
  \[
    \frac{K_{\frac12-p+\frac{\delta(p)}{2}}(x)}{K_{\frac12-p}(x)}
    \le \frac{(2-2p)^{\frac{\delta(p)}{2}}}{x^{\frac{\delta(p)}{2}}}+1.
  \]
  Combining and using $x^{-\frac{\delta(p)}{2}}=(\tilde a\tilde b)^{-\frac{\delta(p)}{4}}$, we obtain
  \begin{align*}
    E\!\left[V_i^{-\frac{\delta(p)}{2}}\mid M_i\right]
     & \le (1+\epsilon^*)\,
    \frac{\Gamma\!\left(\frac12+\frac{\delta(p)}{2}-p\right)2^{\frac{\delta(p)}{2}}}
    {\Gamma\!\left(\frac12-p\right)}\,\tilde b^{-\frac{\delta(p)}{2}}
    + \left(\frac{\tilde a}{\epsilon}\right)^{\frac{\delta(p)}{2}}
    \left(\frac{(2-2p)^{\frac{\delta(p)}{2}}}{\epsilon^{\frac{\delta(p)}{2}}}+1\right) \\
     & = (1+\epsilon^*)\,
    \frac{\Gamma\!\left(\frac12+\frac{\delta(p)}{2}-p\right)2^{\frac{\delta(p)}{2}}}
    {\Gamma\!\left(\frac12-p\right)}\,\Bigl|\frac{M_i}{\sigma}\Bigr|^{-\delta(p)}
    + C_{2,1}\!\left(p,\tilde a\right),
  \end{align*}
  for an explicit finite $C_{2,1}(p,\tilde a)$.

  Moreover,
  \begin{align*}
    (1+\epsilon^*)\frac{\Gamma(\frac12+\frac{\delta(p)}{2}-p)2^{\frac{\delta(p)}{2}}}{\Gamma(\frac12-p)}\,\kappa(p)
     & =(1+\epsilon^*)\frac{\bigl[\Gamma(\frac{1-p}{2})\bigr]^2}{\Gamma(\frac12-p)\Gamma(\frac12)}
    \le (1+\epsilon^*)\frac{1-2p}{1-2p+\frac{p^2}{2}},
  \end{align*}
  where the last inequality follows from \cite{gurland1956}. Hence choose
  $\epsilon^*<\frac{p^2}{2(1-2p)}$, so that $C_1(p)\kappa(p)<1$ in this case.

  \paragraph*{Case 2: $p>\frac12$}
  Since $\delta(p)=\min\{\frac12,\,2p-1\}$, we have $\delta(p)\le 2p-1$, hence
  \[
    p-\frac12-\frac{\delta(p)}{2}\ge p-\frac12-\frac{2p-1}{2}=0,
    \qquad
    p-\frac12-\frac{\delta(p)}{2}\le p-\frac12.
  \]
  Thus $0\le p-\frac12-\frac{\delta(p)}{2}\le p-\frac12$.
  Using $K_{-\nu}(x)=K_\nu(x)$ and that $\nu\mapsto K_\nu(x)$ is increasing on $\nu\ge0$ for each $x>0$,
  \[
    \frac{K_{p-\frac12-\frac{\delta(p)}{2}}(x)}{K_{p-\frac12}(x)}\le 1.
  \]
  Therefore,
  \begin{align*}
    E\!\left[V_i^{-\frac{\delta(p)}{2}}\mid M_i\right]
     & \le \left(\frac{\sqrt{\tilde a}}{\sqrt{\tilde b}}\right)^{\frac{\delta(p)}{2}}
    \le \frac{1}{2\kappa(p)}\,\tilde b^{-\frac{\delta(p)}{2}}+\frac{\kappa(p)}{2}\,\tilde a^{\frac{\delta(p)}{2}}
    \quad\text{(Young)}                                                               \\
     & =\frac{1}{2\kappa(p)}\Bigl|\frac{M_i}{\sigma}\Bigr|^{-\delta(p)}
    +\frac{\kappa(p)}{2}\,\tilde a^{\frac{\delta(p)}{2}},
  \end{align*}
  so here $C_1(p)=1/(2\kappa(p))$ and $C_1(p)\kappa(p)=1/2<1$.

  \paragraph*{Case 3: $p=\frac12$}
  Here $\delta(p)=\frac12$. By Propositions~\ref{prp:bessel-inequalities-A3} and \ref{prp:bessel-inequalities-A4},
  there exists $\epsilon'>0$ such that, for $x=\sqrt{\tilde a\tilde b}$,
  \[
    \frac{K_{\frac{\delta(p)}{2}}(x)}{K_0(x)}
    \le \frac{2^{\frac14}\sqrt{\pi}}{2\Gamma(\frac14)\,x^{\frac{\delta(p)}{2}}}\,I_{(0<x<\epsilon')}
    +\left(\frac{1}{x^{\frac{\delta(p)}{2}}}+1\right)I_{(x\ge\epsilon')}.
  \]
  Hence
  \begin{align*}
    E\!\left[V_i^{-\frac{\delta(p)}{2}}\mid M_i\right]
     & =\left(\frac{\sqrt{\tilde a}}{\sqrt{\tilde b}}\right)^{\frac{\delta(p)}{2}}
    \frac{K_{\frac{\delta(p)}{2}}(x)}{K_0(x)}                                              \\
     & \le \frac{2^{\frac14}\sqrt{\pi}}{2\Gamma(\frac14)}\,\tilde b^{-\frac{\delta(p)}{2}}
    + C_{2,3}\!\left(p,\tilde a\right)
    = \frac{2^{\frac14}\sqrt{\pi}}{2\Gamma(\frac14)}\Bigl|\frac{M_i}{\sigma}\Bigr|^{-\delta(p)}+C_{2,3}(p,\tilde a),
  \end{align*}
  and one finds
  \[
    \frac{2^{\frac14}\sqrt{\pi}}{2\Gamma(\frac14)}\,\kappa\!\left(\tfrac12\right)=\frac12<1.
  \]

  \paragraph*{Conclusion}
  Combining the three cases yields
  \[
    E\!\left[V_i^{-\frac{\delta(p)}{2}}\mid M_i\right]
    \le C_1(p)\Bigl|\frac{M_i}{\sigma}\Bigr|^{-\delta(p)}+C_2(p,\tilde a),
  \]
  with
  \[
    C_1(p)= (1+\epsilon^*)\frac{\Gamma(\frac12+\frac{\delta(p)}{2}-p)2^{\frac{\delta(p)}{2}}}{\Gamma(\frac12-p)}\,I_{(p<\frac12)}
    +\frac{1}{2\kappa(p)}\,I_{(p\ge\frac12)},
  \]
  and a finite $C_2(p,\tilde a)$ collecting the corresponding constants in each case.
  Moreover, $C_1(p)\kappa(p)<1$ for all $p>0$ by the arguments above.
\end{proof}

\begin{proof}[Proof of Lemma~\ref{lem:l1-duality-nosqrtn}]
  By $\ell_\infty$-duality, we have
  \[
    \|\bar{Q}(V)^{-1}u\|_1 = \sup_{\|s\|_\infty\le 1} s^\top \bar{Q}(V)^{-1}u.
  \]
  For any such $s$, applying the Cauchy--Schwarz inequality in the $\bar{Q}(V)^{-1}$-inner product yields
  \[
    s^\top \bar{Q}(V)^{-1}u
    \le \sqrt{s^\top \bar{Q}(V)^{-1}s}\,\sqrt{u^\top \bar{Q}(V)^{-1}u}.
  \]
  Since $\rho B^\top B \succeq 0$, we have $\bar{Q}(V) \succeq D_V^{-1}$. This implies $\bar{Q}(V)^{-1} \preceq D_V$. Consequently,
  \[
    s^\top \bar{Q}(V)^{-1}s \le s^\top D_V s = \sum_{i=1}^n V_i s_i^2 \le \sum_{i=1}^n V_i,
  \]
  where the last inequality holds because $|s_i|\le 1$. Taking the supremum over $\|s\|_\infty\le 1$ yields \eqref{eq:l1-duality-main}.
\end{proof}

\section{Explicit Gradient Derivations}\label{app:gradients}

In this appendix, we provide detailed derivations of the gradient estimators for the joint log-likelihood function of the LLnGM model described in Section~\ref{sec:model}. We explicitly compare the formulations under the centered and non-centered parameterizations to highlight the integrability issues inherent in the former.

\subsection{Parameterization and Distributions}

Following \cite{wallin2015} and \cite{podgorski2016}, ensuring that the variance process belongs to a class of distributions closed under convolution limits us to two special cases of the Generalized Hyperbolic (GH) distribution: the Normal-Inverse Gaussian (NIG) and the Generalized Asymmetric Laplace (GAL) distributions. To avoid over-parameterization of $B$ in \eqref{eq:mixture}, we adopt the parameterization from \cite{asar2020}. Specifically, let $\nu > 0$ be the variance component parameter: For NIG noise, $V_i \sim \mathrm{IG}(\nu, \nu h_i^2)$, and for GAL noise $V_i \sim \mathrm{Gamma}(h_i \nu, \nu)$.

\subsection{Centered Parameterization}

Let $Y = (y_1, \dots, y_m)^\top$, $W = (w_1, \dots, w_n)^\top$, and $\zeta$ denote the vector of kernel parameters, and let $\theta$ denote the full parameter vector (including $\zeta$, $\mu$, $\sigma$, etc.). The joint log-likelihood is given by
\[
  \log p_\theta(Y, W, V) = \log p_\theta(Y \mid W) + \log p_\theta(W \mid V) + \log p_\theta(V),
\]
where the components are:
\begin{align*}
  \log p_\theta(Y \mid W) & = -\frac{m}{2} \log(2\pi) - m \log \sigma_{\epsilon} - \frac{1}{2\sigma_{\epsilon}^2} \|Y - X \beta - A W \|^2, \\
  \log p_\theta(W \mid V)
                          & = -\frac{n}{2}\log(2\pi) - n\log\sigma
  + \log|K|
  -\frac12\sum_{i=1}^n\log V_i                                                                                                              \\
                          & \quad -\frac{1}{2\sigma^2}(KW-\mu(V-h))^\top D_V^{-1}(KW-\mu(V-h)),
\end{align*}
with $D_V = \mathrm{diag}(V)$. The term $\log p_\theta(V)$ depends on the mixing distribution:
\begin{itemize}
  \item For $V_i \sim \mathrm{IG}(\nu, \nu h_i^2)$:
        \[
          \log p_\theta(V) = \sum_{i=1}^n \left( \frac{1}{2} \log \left( \frac{\nu h_i^2}{2 \pi V_i^3} \right) - \frac{\nu V_i^2 + \nu h_i^2}{2 V_i} + \nu h_i \right).
        \]
  \item For $V_i \sim \mathrm{Gamma}(h_i \nu, \nu)$:
        \[
          \log p_\theta(V) = \sum_{i=1}^n \left( h_i \nu \log(\nu) - \log \Gamma(h_i \nu) + (h_i \nu - 1)\log V_i - \nu V_i \right).
        \]
\end{itemize}

We now derive the gradients of the joint log-likelihood $\mathcal{L} = \log p_\theta(Y, W, V)$ with respect to each parameter.

\paragraph*{Regression coefficients $\beta$}
The gradient with respect to the regression coefficients is
\[
  \nabla_{\beta} \mathcal{L} = \frac{1}{\sigma_{\epsilon}^2} X^\top (Y - X \beta - A W).
\]

\paragraph*{Noise variance $\sigma_{\epsilon}$}
Differentiating with respect to the observation noise variance yields
\[
  \partial_{\sigma_{\epsilon}} \mathcal{L} = -\frac{m}{\sigma_{\epsilon}} + \frac{1}{\sigma_{\epsilon}^3} \|Y - X \beta - A W \|^2.
\]

\paragraph*{Process scale $\sigma$}
Expanding the quadratic term $((KW)_i - \mu V_i + \mu h_i)^2 / V_i$, we obtain:
\begin{align*}
  \partial_\sigma \mathcal{L}
   & = -\frac{n}{\sigma} + \frac{1}{\sigma^3} \sum_{i=1}^n \frac{((KW)_i - \mu (V_i - h_i))^2}{V_i} \\
   & = -\frac{n}{\sigma} + \frac{1}{\sigma^3} \sum_{i=1}^n \left[
    \frac{(KW)_i^2}{V_i} + \mu^2 V_i + \frac{\mu^2 h_i^2}{V_i} - 2\mu(KW)_i + \frac{2\mu h_i (KW)_i}{V_i} - 2\mu^2 h_i
    \right].
\end{align*}
Note that this expression contains the explicit term $+\mu^2 h_i^2/V_i$. While the likelihood term may suppress the singularity at $V_i \to 0$, verifying the posterior integrability of this score function is technically delicate, especially for the GAL case with small shape parameters ($h_i \nu \le 1$) where the prior inverse moment $\mathbb{E}[V_i^{-1}]$ does not exist.

\paragraph*{Drift parameter $\mu$}
The gradient with respect to the drift parameter is
\begin{align*}
  \partial_\mu \mathcal{L}
   & = \frac{1}{\sigma^2} \sum_{i=1}^n \frac{((KW)_i - \mu V_i + \mu h_i)(V_i - h_i)}{V_i} \\
   & = \frac{1}{\sigma^2} \sum_{i=1}^n \left[
    (KW)_i - \mu V_i + 2\mu h_i - \frac{h_i (KW)_i}{V_i} - \frac{\mu h_i^2}{V_i}
    \right].
\end{align*}
Similar to $\partial_\sigma \mathcal{L}$, the presence of terms scaling with $1/V_i$ (such as $-\mu h_i^2/V_i$) poses challenges for establishing $L^1$ convergence guarantees using standard drift conditions.

\paragraph*{Kernel parameters $\zeta_j$}
For each kernel parameter $\zeta_j$, the gradient is given by
\[
  \partial_{\zeta_j} \mathcal{L} = \mathrm{tr}(K^{-1} \partial_{\zeta_j} K) - \frac{1}{\sigma^2} (KW - \mu(V-h))^\top D_V^{-1} (\partial_{\zeta_j} K) W.
\]

\paragraph*{Mixing parameter $\nu$}
The gradient with respect to the mixing parameter depends on the choice of distribution:
\begin{equation}
\partial_\nu \mathcal{L} =
\begin{cases}
\sum_{i=1}^n \left( \frac{1}{2\nu} + h_i - \frac{V_i}{2} - \frac{h_i^2}{2 V_i} \right) & \text{for NIG,}\\
\sum_{i=1}^n \left( h_i \log(\nu) + h_i - h_i \psi(h_i \nu) + h_i \log V_i - V_i \right) & \text{for GAL,}
\end{cases}
\end{equation}
where $\psi(\cdot)$ is the digamma function.

\subsection{Non-Centered Parameterization}

In the non-centered parameterization, we define $M = KW + \mu h$, so that $W = K^{-1}(M-\mu h)$. The likelihood components become:
\begin{align*}
  \log p_\theta(Y \mid M) & = -\frac{m}{2} \log(2\pi) - m \log \sigma_{\epsilon} - \frac{1}{2\sigma_{\epsilon}^2} \|Y - X \beta - A K^{-1} (M - \mu h) \|^2, \\
  \log p_\theta(M \mid V) & = -\frac{n}{2} \log(2\pi) - n \log \sigma - \frac{1}{2} \sum_{i=1}^n \log V_i
  - \frac{1}{2 \sigma^2} \sum_{i=1}^n \frac{(M_i - \mu V_i)^2}{V_i}.
\end{align*}
The gradients for this formulation are derived analogously.

\paragraph*{Regression coefficients $\beta$}
\[
  \nabla_{\beta} \log p_\theta = \frac{1}{\sigma_{\epsilon}^2} X^\top \left(Y - X \beta - A K^{-1} (M - \mu h)\right).
\]

\paragraph*{Noise variance $\sigma_{\epsilon}$}
\[
  \partial_{\sigma_{\epsilon}} \log p_\theta = -\frac{m}{\sigma_{\epsilon}} + \frac{1}{\sigma_{\epsilon}^3} \|Y - X \beta - A K^{-1} (M - \mu h) \|^2.
\]

\paragraph*{Process scale $\sigma$}
\[
  \partial_\sigma \log p_\theta = -\frac{n}{\sigma} + \frac{1}{\sigma^3} \sum_{i=1}^n \frac{(M_i - \mu V_i )^2}{V_i}.
\]
Here the term $(M_i-\mu V_i)^2/V_i$ is a self-normalized quadratic form. Under the Gaussian layer $M\mid V$, it equals $\sigma^2 Z_i^2$ with $Z_i\sim N(0,1)$, and more generally it remains well-behaved near $V_i\to 0$ when working with the conditional Gaussian distribution $M\mid(V,Y)$. This structure makes moment bounds more transparent under the non-centered parametrization.

\paragraph*{Drift parameter $\mu$}
Using the chain rule on both $\log p(Y|M)$ and $\log p(M|V)$:
\begin{align*}
  \partial_\mu \log p_\theta
  = \frac{1}{\sigma^2}\sum_{i=1}^n (M_i-\mu V_i)
  + \frac{1}{\sigma_\epsilon^2}(A K^{-1} h)^\top
  \left( Y - X \beta - A K^{-1} (M - \mu h) \right).
\end{align*}

\paragraph*{Kernel parameters $\zeta_j$}
\[
  \partial_{\zeta_j} \log p_\theta = \frac{1}{\sigma_{\epsilon}^2} \left( Y - X \beta - A K^{-1} (M - \mu h) \right)^\top A K^{-1} (\partial_{\zeta_j} K) K^{-1} (M - \mu h).
\]

The gradient for $\nu$ remains the same as in the centered case. The Monte Carlo and Rao--Blackwellized gradient estimators follow by substituting $(W,V)$ (centered) or $(M,V)$ (non-centered) with Gibbs samples, and, for the Rao--Blackwellized version, replacing the corresponding latent Gaussian variable by its conditional expectation under $W\mid(V,Y)$ or $M\mid(V,Y)$, respectively.

\bibliographystyle{imsart-number}
\bibliography{References.bib}

@article{asar2020,
  author   = {Asar, Ozgür and Bolin, David and Diggle, Peter J. and Wallin, Jonas},
  title    = {Linear mixed effects models for non-Gaussian continuous repeated measurement data},
  journal  = {Journal of the Royal Statistical Society: Series C (Applied Statistics)},
  volume   = {69},
  number   = {5},
  pages    = {1015-1065},
  eprint   = {https://rss.onlinelibrary.wiley.com/doi/pdf/10.1111/rssc.12405},
  year     = {2020}
}

@article{marcus1990,
  title     = {Determinants of sums},
  author    = {Marcus, Marvin},
  journal   = {The College Mathematics Journal},
  volume    = {21},
  number    = {2},
  pages     = {130--135},
  year      = {1990},
  publisher = {Taylor \& Francis}
}

@article{bolin2014,
  author   = {Bolin, David},
  title    = {Spatial Matérn Fields Driven by Non-Gaussian Noise},
  journal  = {Scandinavian Journal of Statistics},
  volume   = {41},
  number   = {3},
  pages    = {557-579},
  doi      = {https://doi.org/10.1111/sjos.12046},
  url      = {https://onlinelibrary.wiley.com/doi/abs/10.1111/sjos.12046},
  eprint   = {https://onlinelibrary.wiley.com/doi/pdf/10.1111/sjos.12046},
  year     = {2014}
}

@article{bolin2020,
  title    = {Multivariate type {G} {Matern} stochastic partial differential equation random fields},
  volume   = {82},
  issn     = {1467-9868},
  url      = {https://onlinelibrary.wiley.com/doi/abs/10.1111/rssb.12351},
  doi      = {10.1111/rssb.12351},
  language = {en},
  number   = {1},
  urldate  = {2023-02-11},
  journal  = {Journal of the Royal Statistical Society: Series B (Statistical Methodology)},
  author   = {Bolin, David and Wallin, Jonas},
  year     = {2020},
  note     = {multi var model},
  pages    = {215--239}
}

@article{INLA2018,
  title      = {Spatial Modeling with {{R-INLA}}: {{A}} Review},
  shorttitle = {Spatial Modeling with {{R-INLA}}},
  author     = {Bakka, Haakon and Rue, Havard and Fuglstad, Geir-Arne and Riebler, Andrea and Bolin, David and Illian, Janine and Krainski, Elias and Simpson, Daniel and Lindgren, Finn},
  year       = {2018},
  journal    = {WIREs Computational Statistics},
  volume     = {10},
  number     = {6},
  pages      = {e1443},
  issn       = {1939-0068},
  doi        = {10.1002/wics.1443},
  urldate    = {2023-11-01},
  copyright  = {2018 Wiley Periodicals, Inc.},
  langid     = {english}
}

@article{fong2010,
  title   = {Bayesian Inference for Generalized Linear Mixed Models},
  author  = {Fong, Youyi and Rue, Havard and Wakefield, Jon},
  year    = {2010},
  month   = jul,
  journal = {Biostatistics (Oxford, England)},
  volume  = {11},
  number  = {3},
  pages   = {397--412},
  issn    = {1465-4644},
  doi     = {10.1093/biostatistics/kxp053},
  urldate = {2023-11-01},
  pmcid   = {PMC2883299},
  pmid    = {19966070}
}

@article{martino2011,
  title    = {Approximate {{Bayesian Inference}} for {{Survival Models}}},
  author   = {Martino, Sara and Akerkar, Rupali and Rue, Havard},
  year     = {2011},
  journal  = {Scandinavian Journal of Statistics},
  volume   = {38},
  number   = {3},
  pages    = {514--528},
  issn     = {1467-9469},
  doi      = {10.1111/j.1467-9469.2010.00715.x},
  urldate  = {2023-11-01},
  langid   = {english}
}

@article{podgorski2016,
  title     = {Convolution-Invariant Subclasses of Generalized Hyperbolic Distributions},
  author    = {Podgorski, Krzysztof and Wallin, Jonas},
  year      = {2016},
  month     = jan,
  journal   = {Communications in Statistics - Theory and Methods},
  volume    = {45},
  number    = {1},
  pages     = {98--103},
  publisher = {{Taylor \& Francis}},
  issn      = {0361-0926},
  doi       = {10.1080/03610926.2013.821489},
  urldate   = {2023-11-05}
}

@article{qin2019,
  title     = {Estimating the Spectral Gap of a Trace-Class {{Markov}} Operator},
  author    = {Qin, Qian and Hobert, James P. and Khare, Kshitij},
  year      = {2019},
  month     = jan,
  journal   = {Electronic Journal of Statistics},
  volume    = {13},
  number    = {1},
  pages     = {1790--1822},
  publisher = {{Institute of Mathematical Statistics and Bernoulli Society}},
  issn      = {1935-7524, 1935-7524},
  doi       = {10.1214/19-EJS1563},
  urldate   = {2023-11-07}
}

@article{roberts1994,
  title     = {Simple Conditions for the Convergence of the {{Gibbs}} Sampler and {{Metropolis-Hastings}} Algorithms},
  author    = {Roberts, G. O. and Smith, A. F. M.},
  year      = {1994},
  journal   = {Stochastic Processes and their Applications},
  volume    = {49},
  number    = {2},
  pages     = {207--216},
  publisher = {{Elsevier}},
  urldate   = {2023-11-07},
  langid    = {english}
}

@article{tierney1994,
  title      = {Markov {{Chains}} for {{Exploring Posterior Distributions}}},
  author     = {Tierney, Luke},
  year       = {1994},
  journal    = {The Annals of Statistics},
  volume     = {22},
  number     = {4},
  eprint     = {2242477},
  eprinttype = {jstor},
  pages      = {1701--1728},
  publisher  = {{Institute of Mathematical Statistics}},
  issn       = {0090-5364},
  urldate    = {2023-11-07}
}

@article{wallin2015,
  title      = {Geostatistical {{Modelling Using Non-Gaussian Mat\'ern Fields}}: {{Non-Gaussian Mat\'ern}} Fields},
  shorttitle = {Geostatistical {{Modelling Using Non-Gaussian Mat\'ern Fields}}},
  author     = {Wallin, Jonas and Bolin, David},
  year       = {2015},
  month      = sep,
  journal    = {Scandinavian Journal of Statistics},
  volume     = {42},
  number     = {3},
  pages      = {872--890},
  issn       = {03036898},
  doi        = {10.1111/sjos.12141},
  urldate    = {2023-01-17},
  langid     = {english}
}

@article{dhull2021,
  title    = {Normal Inverse {{Gaussian}} Autoregressive Model Using {{EM}} Algorithm},
  author   = {Dhull, Monika S. and Kumar, Arun},
  year     = {2021},
  month    = sep,
  journal  = {International Journal of Advances in Engineering Sciences and Applied Mathematics},
  volume   = {13},
  number   = {2},
  pages    = {139--147},
  issn     = {0975-5616},
  doi      = {10.1007/s12572-021-00303-y},
  urldate  = {2023-11-13},
  langid   = {english}
}

@inproceedings{paciorek2003,
  title     = {Nonstationary {{Covariance Functions}} for {{Gaussian Process Regression}}},
  booktitle = {Advances in {{Neural Information Processing Systems}}},
  author    = {Paciorek, Christopher and Schervish, Mark},
  year      = {2003},
  volume    = {16},
  publisher = {{MIT Press}},
  urldate   = {2023-11-13}
}

@article{walder2020,
  title    = {Bayesian Analysis of Spatial Generalized Linear Mixed Models with {{Laplace}} Moving Average Random Fields},
  author   = {Walder, Adam and Hanks, Ephraim M.},
  year     = {2020},
  month    = apr,
  journal  = {Computational Statistics \& Data Analysis},
  volume   = {144},
  pages    = {106861},
  issn     = {0167-9473},
  doi      = {10.1016/j.csda.2019.106861},
  urldate  = {2023-11-13}
}

@article{roberts1997,
  title     = {Geometric {{Ergodicity}} and {{Hybrid Markov Chains}}},
  author    = {Roberts, Gareth and Rosenthal, Jeffrey},
  year      = {1997},
  month     = jan,
  journal   = {Electronic Communications in Probability},
  volume    = {2},
  number    = {none},
  pages     = {13--25},
  publisher = {{Institute of Mathematical Statistics and Bernoulli Society}},
  issn      = {1083-589X, 1083-589X},
  doi       = {10.1214/ECP.v2-981},
  urldate   = {2023-11-17}
}

@book{abramowitz1965,
  author    = {Abramowitz, Milton and Stegun, Irene},
  title     = {Handbook of {{Mathematical Functions}} with {{Formulas}}, {{Graphs}}, and {{Mathematical Tables}}},
  year      = {1965},
  series    = {Dover {{Books}} on {{Advanced Mathematics}}},
  publisher = {{Dover Publications}},
  address   = {{New York}}
}

@article{pal2014,
  title    = {Geometric Ergodicity for Bayesian Shrinkage Models},
  author   = {Pal, Subhadip and Khare, Kshitij},
  year     = {2014},
  month    = jan,
  journal  = {Electronic Journal of Statistics},
  volume   = {8},
  number   = {1},
  issn     = {1935-7524},
  doi      = {10.1214/14-EJS896},
  urldate  = {2024-03-17},
  langid   = {english}
}

@article{gurland1956,
  author   = {Gurland, John},
  title    = {An inequality satisfied by the {G}amma function},
  journal  = {Skand. Aktuarietidskr.},
  fjournal = {Skandinavisk Aktuarietidskrift. Scandinavian Actuarial
              Journal},
  volume   = {39},
  year     = {1956},
  pages    = {171--172},
  issn     = {0037-606X},
  mrclass  = {33.00},
  mrnumber = {95293},
  doi      = {10.1080/03461238.1956.10414947},
  url      = {https://doi.org/10.1080/03461238.1956.10414947}
}

@article{rosenthal1995,
  title      = {Minorization {{Conditions}} and {{Convergence Rates}} for {{Markov Chain Monte Carlo}}},
  author     = {Rosenthal, Jeffrey S.},
  year       = {1995},
  journal    = {Journal of the American Statistical Association},
  volume     = {90},
  number     = {430},
  eprint     = {2291067},
  eprinttype = {jstor},
  pages      = {558--566},
  publisher  = {[American Statistical Association, Taylor \& Francis, Ltd.]},
  issn       = {0162-1459},
  doi        = {10.2307/2291067},
  urldate    = {2024-03-17}
}

@article{rue2017,
  title   = {Bayesian Computing with {INLA}: {A} Review},
  author  = {Rue, Håvard and Riebler, Andrea and Sørbye, Sigrunn H and Illian, Janine B and Simpson, Daniel P and Lindgren, Finn K},
  year    = {2017},
  journal = {Annual Review of Statistics and Its Application},
  volume  = {4},
  number  = {1},
  pages   = {395--421},
  issn    = {2372-4894},
  doi     = {10.1146/annurev-statistics-031016-085240},
  urldate = {2024-03-17}
}

@article{barndorff1978,
  title   = {Hyperbolic {{Distributions}} and {{Distributions}} on {{Hyperbolae}}},
  author  = {Barndorff-Nielsen, Ole},
  year    = {1978},
  journal = {Scandinavian Journal of Statistics},
  volume  = {5},
  number  = {3}
}

@article{cabral2024,
  title   = {Fitting Latent Non-{{Gaussian}} Models Using Variational Bayes and Laplace Approximations},
  author  = {Cabral, Rafael and Bolin, David and Rue, Håvard},
  year    = {2024},
  journal = {Journal of the American Statistical Association},
  volume  = {119},
  number  = {546},
  pages   = {1--13},
  issn    = {0162-1459},
  doi     = {10.1080/01621459.2023.2296704}
}

@book{douc2013nonlinear,
  title     = {Nonlinear Time Series: Theory, Methods and Applications with {R} Examples},
  author    = {Douc, Randal and Moulines, Eric and Stoffer, David},
  year      = {2013},
  publisher = {Chapman and Hall/CRC},
  address   = {New York},
  doi       = {10.1201/b16331},
  isbn      = {9780429112638},
  pages     = {551}
}

@article{carvalho2010horseshoe,
  title   = {The horseshoe estimator for sparse signals},
  author  = {Carvalho, Carlos M. and Polson, Nicholas G. and Scott, James G.},
  journal = {Biometrika},
  volume  = {97},
  number  = {2},
  pages   = {465--480},
  year    = {2010},
  month   = {June},
  doi     = {10.1093/biomet/asq017},
  url     = {https://www.jstor.org/stable/25734098}
}

@article{park2008bayesian,
  title   = {The Bayesian {{Lasso}}},
  author  = {Park, Trevor and Casella, George},
  journal = {Journal of the American Statistical Association},
  volume  = {103},
  number  = {482},
  pages   = {681--686},
  year    = {2008}
}

@article{khare2013,
  title   = {Geometric Ergodicity of the Bayesian {{Lasso}}},
  author  = {Khare, Kshitij and Hobert, James P},
  journal = {Electronic Journal of Statistics},
  volume  = {7},
  pages   = {2150--2163},
  year    = {2013}
}

@article{bhattacharya2015,
  title   = {Dirichlet–{{Laplace}} Priors for Optimal Shrinkage},
  author  = {Bhattacharya, Anirban and Pati, Debdeep and Pillai, Natesh S and Dunson, David B},
  journal = {Journal of the American Statistical Association},
  volume  = {110},
  number  = {512},
  pages   = {1479--1490},
  year    = {2015}
}

@article{bhattacharya2022,
  title   = {Geometric Ergodicity of Gibbs Samplers for the Horseshoe and its Regularized Variants},
  author  = {Bhattacharya, Suman and Khare, Kshitij and Pal, Subhadip},
  journal = {Electronic Journal of Statistics},
  volume  = {16},
  number  = {1},
  pages   = {1--57},
  year    = {2022},
  doi     = {10.1214/21-EJS1932}
}

@article{zhang2019,
  title   = {Trace class Markov chains for the Normal-Gamma Bayesian shrinkage model},
  author  = {Zhang, Liyuan and Khare, Kshitij and Xing, Zeren},
  journal = {Electronic Journal of Statistics},
  volume  = {13},
  number  = {1},
  pages   = {166--207},
  year    = {2019},
  doi     = {10.1214/18-EJS1491}
}

@phdthesis{zhang2017,
  title  = {Trace class markov chains for bayesian shrinkage models},
  author = {Zhang, Liyuan},
  year   = {2017},
  school = {University of Florida}
}

@article{pal2017traceclass,
  author   = {Pal, Subahdip and Khare, Kshitij and Hobert, James P.},
  title    = {Trace Class Markov Chains for Bayesian Inference with Generalized Double Pareto Shrinkage Priors},
  journal  = {Scandinavian Journal of Statistics},
  volume   = {44},
  number   = {2},
  pages    = {307-323},
  doi      = {https://doi.org/10.1111/sjos.12254},
  url      = {https://onlinelibrary.wiley.com/doi/abs/10.1111/sjos.12254},
  year     = {2017}
}

@article{robbins1951,
  title   = {A Stochastic Approximation Method},
  author  = {Robbins, Herbert and Sutton, Monro},
  journal = {Ann. Math. Statist.},
  volume  = {22},
  number  = {3},
  pages   = {400--407},
  year    = {1951}
}

@book{bottou1998,
  title     = {Online Algorithms and Stochastic Approximations},
  author    = {Bottou, Léon},
  year      = {1998},
  publisher = {Cambridge University Press},
  isbn      = {978-0-521-65263-6}
}

@book{kushner2003stochastic,
  title     = {Stochastic Approximation and Recursive Algorithms and Applications},
  author    = {Kushner, Harold J. and Yin, G. George},
  year      = {2003},
  publisher = {Springer New York, NY},
  series    = {Stochastic Modelling and Applied Probability},
  volume    = {35},
  edition   = {2},
  doi       = {10.1007/b97441},
  isbn      = {978-0-387-00894-3}
}

@article{roberts2004,
  author  = {Roberts, Gareth O. and Rosenthal, Jeffrey S.},
  title   = {General state space {Markov} chains and {MCMC} algorithms},
  journal = {Probability Surveys},
  year    = {2004},
  volume  = {1},
  pages   = {20--71},
  doi     = {10.1214/154957804100000024},
  issn    = {1549-5787}
}

\end{document}